\documentclass[a4paper,11pt,twoside, tbtags]{article}
\usepackage{wiaspreprint}

\usepackage[noadjust]{marginnote}

\usepackage{amsmath}
\usepackage{amsthm}
\usepackage{aligned-overset}
\usepackage{amssymb, mathtools}
\usepackage[shortlabels]{enumitem}
\usepackage{cases}
\usepackage[makeroom]{cancel}

\usepackage{outlines}
\usepackage{graphicx}
\usepackage{caption}
\usepackage{subcaption}
\usepackage{comment}
\usepackage{float}
\usepackage{color}
\usepackage{hyperref}
\hypersetup{
    colorlinks=true,
    linkcolor=blue,
    citecolor=red
    }
\usepackage{mathrsfs}
\usepackage{url}

\newcommand{\constNS}{\mathbb{\alpha}}
\newcommand{\constNP}{\mathbb{\beta}}
\newcommand{\constP}{\mathbb{\gamma}}
\newcommand{\R}{\mathbb{R}}

\newcommand{\dS}{\diff \sigma}
\newcommand{\T}{\mathcal{T}}
\newcommand{\Tmax}{T_{\textrm{max}}}

\newcommand{\Diss}{\mathcal{W}}

\newcommand{\N}{\mathbb{N}}

\newcommand{\V}{\mathbb{V}}
\newcommand{\W}{\mathbb{W}}

\newcommand{\E}{\mathcal{E}}

\newcommand\dualP[2]{\left\langle #1,\, #2 \right\rangle}
\newcommand\set[2]{\left \{ #1\, \lvert \, #2 \right \}}
\newcommand{\diff}{\mathrm{\hspace{2pt}d}}

\newcommand{\Rey}{\mathrm{Re}}
\newcommand{\Pe}{\mathrm{Pe}}
\newcommand\norm[1]{\left\lVert #1 \right\rVert}
\newcommand\bs[1]{\boldsymbol{#1}}
\newcommand{\X}{\bs{x}}
\newcommand{\Y}{\bs{y}}
\newcommand{\bv}{\bs{v}}
\newcommand{\tbv}{\tilde{\bs{v}}}
\newcommand{\bn}{\bs{n}}
\newcommand{\bu}{\bs{u}}

\newcommand{\cpm}{c^{\pm}}
\newcommand{\tcpm}{\tilde{c}^{\pm}}

\newcommand{\tpsi}{\tilde{\psi}}

\newcommand{\bd}{\bs{d}}

\newcommand{\lamMa}{\bs{\Lambda}(\bd)}
\newcommand{\varMa}{\bs{\mathcal{E}}(\bd)}
\newcommand\snorm[1]{\big\lVert #1 \big\rVert}
\newcommand{\lamMan}{\bs{\Lambda}(\bd_\kappa)}
\newcommand{\varMan}{\bs{\mathcal{E}}(\bd_\kappa)}

\usepackage{tikz}
\newcommand{\sigmp}[1]{\;
  \begin{tikzpicture}[baseline=(content.base), inner sep=0pt]
    \node[inner sep=0pt] (content) {$\displaystyle #1$};
    \draw ([xshift=-0.1cm]content.west) -- ([xshift=-0.1cm, yshift=0.05cm]content.north west) -- ++(0.15,0);
    \draw  ([xshift=-0.1cm]content.west) -- ([xshift=-0.1cm,yshift=-0.05cm]content.south west) -- ++(0.15,0);
    \draw ([xshift=0.1cm]content.east) -- ([xshift=0.1cm, yshift=0.05cm]content.north east) -- ++(-0.15,0);
    \draw  ([xshift=0.1cm]content.east) -- ([xshift=0.1cm, yshift=-0.05cm]content.south east) -- ++(-0.15,0);
    \filldraw ([xshift=-0.1cm]content.north west) -- ([xshift=-0.05cm]content.west) --  ([xshift=-0.1cm]content.south west)  -- cycle;
    \filldraw ([xshift=0.1cm]content.north east) -- ([xshift=0.05cm]content.east) --  ([xshift=0.1cm]content.south east)  -- cycle;
    \node at ([xshift=0.2cm, yshift=0.03cm]content.north east) (C) {$\scriptscriptstyle \pm$};
    \vspace{-0.2cm}
  \end{tikzpicture}
  \hspace{-0.05cm}
}

\newcommand{\threedotsbin}{\text{\,\multiput(0,-2)(0,2){3}{$\cdot$}}\,\,\,}
\newtheorem{thm}{Theorem}[section]

\newtheorem{definition}[thm]{Definition}
\newtheorem{propo}[thm]{Proposition}
\newtheorem{remark}[thm]{Remark}
\newtheorem{lem}[thm]{Lemma}
\newtheorem{cor}[thm]{Corollary}
\newtheorem{ass}[thm]{Assumption}
\newcommand{\pat}[2]{\frac{\partial #1}{\partial #2}}
\DeclareMathOperator{\e}{e}

\allowdisplaybreaks
\DeclareUnicodeCharacter{0301}{\'{e}}
\begin{document}
\author[D.~Hömberg, R.~Lasarzik, L.~Plato]{%
\nofnmark{}				% no footnote mark; please call before \footnote{}
Dietmar Hömberg\footnote{Weierstrass Institute \\
Mohrenstr. 39 \\ 10117 Berlin \\ Germany \\
E-Mail: dietmar.hoemberg@wias-berlin.de	% Please use firstname.lastname@wias-berlin.de
\\\hphantom{E-Mail:} robert.lasarzik@wias-berlin.de
\\\hphantom{E-Mail:} luisa.plato@wias-berlin.de
}, Robert Lasarzik, 
%\footnote{Weierstrass Institute \\
%Mohrenstr. 39 \\ 10117 Berlin \\ Germany \\
%E-Mail: robert.lasarzik@wias-berlin.de	% Please use firstname.lastname@wias-berlin.de
%\\\hphantom{E-Mail:} second.wiasauthor@wias-berlin.de
%}
%\addmark{,}{n}				% seperator "," mark relating to already existing mark
Luisa Plato
%\footnote{Weierstrass Institute \\
%Mohrenstr. 39 \\ 10117 Berlin \\ Germany \\
%E-Mail: luisa.plato@wias-berlin.de	% Please use firstname.lastname@wias-berlin.de
%\\\hphantom{E-Mail:} second.wiasauthor@wias-berlin.de
%}
}
\title[Weak solutions to an anisotropic electrokinetic flow model]{Existence of suitable weak solutions to an anisotropic electrokinetic flow model}	% in English: lowercase; start with capital letter after colon or dash; short title for page headings
\nopreprint{3104}	% preprint number
%\nopreyear{+++Preprint-Jahr+++}	% preprint year
\selectlanguage{english}		% do not change; important for date format
%\date{+++Datum+++}			% fix date, e.g. February 22, 2017
\subjclass[+++SubjectClassification-Edition+++]{35D30, 35G61, 35Q30}	% Math. Subject Classif.
%\pacs[2008]{}				% Physics Astronomy Classif., if any
\thanks{Funded by the Deutsche Forschungsgemeinschaft (DFG, German Research Foundation) under Germany's Excellence Strategy – The Berlin Mathematics Research Center MATH+ (EXC-2046/1, project ID: 390685689).}				% acknowledgements; period is set automatically!
%% amsart only! abstract before maketitle
%\begin{abstract}Abstract\ldots\end{abstract}
\maketitle
%% article and other classes: abstract after maketitle
\begin{abstract}
	In this article we present a system of coupled non-linear PDEs modeling an anisotropic electrokinetic flow.
	We show the existence of suitable weak solutions in three spatial dimensions, that is weak solutions which fulfill an energy inequality, via a regularized system.
    The flow is modeled by a Navier--Stokes--Nernst--Planck--Poisson system and
    the anisotropy is introduced via space dependent diffusion matrices in the Nernst--Planck and Poisson equations.
\end{abstract}
%%%%%%%%%%%%%%%%%%%%%%%%%%%%%%%%%%%%%%%%%%%%%%%%%%%%%%%%%%%%%%%%%%%%%%
%                                                                    %
%                         Start the article here                     %
%                                                                    %
%%%%%%%%%%%%%%%%%%%%%%%%%%%%%%%%%%%%%%%%%%%%%%%%%%%%%%%%%%%%%%%%%%%%%%

\begin{section}{Introduction}
%Motivierender Satz ;)
%Electrokinetics plays an important role in the design of nano-fluid ``lab-on-a-chip'' devices, \cite{zhao_yang_2012}.
%On these chips, there is a need to mix very small amounts of fluid, but 
%on such small scales mechanical stirring becomes impractical, due to, among other reasons, the high viscosity of the fluid. 
%In order to circumvent this difficulty, one innovative idea is to dissolve ions in the fluid and use an electric field to control the  flow of the ions, and use this for instance for mixing, \cite{zhao_yang_2012}.
Electrokinetics is essential in the design of nano-fluid ``lab-on-a-chip'' devices \cite{zhao_yang_2012}. These devices require the mixing of tiny fluid volumes, but at such small scales, mechanical stirring becomes impractical due to, amongst other reasons, the high viscosity of the fluid. To address this challenge, one innovative approach involves dissolving ions in the fluid and using an electric field to manipulate the flow for effective mixing \cite{zhao_yang_2012}.

%the system + explanation of all appearing variabels
When charged particles $\cpm$ are dissolved in an incompressible fluid with velocity $\bv$ and under the influence of an external electric field $- \nabla \psi$, 
three major effects govern the movement of the charges. 
The charges diffuse, they are transported by the surrounding fluid and the electric field induces a directed movement, called electromigration.
Further, space charge exerts an electric body force on the fluid velocity.
An attempt to model this physical interaction leads us to a coupled Navier--Stokes--Nernst--Planck--Poisson (NSNPP) system, \cite[Chap.~3.4]{probstein}.
More explicitly, we are considering, 
\begin{subequations}\label{eq:system}
\begin{align}
	\mathrm{Re} \left(\partial_t \bv + (\bv \cdot \nabla) \bv \right) 
		- \Delta \bv + \nabla p + \constNS (c^+ - c^-) \nabla \psi &= 0 \quad \text{ in } \Omega \times (0,T),\label{eq:system_ns}\\
	\nabla \cdot \bv &= 0 \quad \text{ in } \Omega \times (0,T), \label{eq:system_div}\\
	\mathrm{Pe} \left( \partial_t c^{\pm} + \nabla \cdot (c^{\pm} \bv)  \right)
		- \nabla \cdot \left( \lamMa (
		\nabla c^{\pm} \pm \constNP c^{\pm} \nabla \psi) \right) &= 0 \quad \text{ in } \Omega \times (0,T), \label{eq:system_np}\\
	- \nabla \cdot \left( \varMa
		\nabla \psi \right) - \constP (c^+ - c^-)
		&= 0 \quad  \text{ in } \Omega \times (0,T) \label{eq:system_p},
\end{align}
\end{subequations}
where $\Omega \subseteq \R^3$ is a smooth and bounded domain, $\Rey, \Pe, \constNS, \constNP, \constP, \lambda, \varepsilon > 0$ 
are positive constants,
$\varMa := \bs{I} + \varepsilon \, \bd \otimes \bd$ and $\lamMa := \bs{I} + \lambda \, \bd \otimes \bd$, 
where $\bd$ is the so-called director 
with $\bd(\X) \in \R^3$ and $ \bd \cdot \bn = 0$ on $\Gamma := \partial \Omega$.
%physical explanation of the system
The evolution of the fluid's velocity field is described by the Navier--Stokes equations for incompressible fluids \eqref{eq:system_ns}--\eqref{eq:system_div}.
The charge densities evolve according to the Nernst--Planck equations \eqref{eq:system_np} including a diffusion term $- \nabla \cdot(\lamMa \nabla \cpm)$ as well as two transport terms, one due to the velocity field $\bv$ and one due to the electric field $- \nabla \psi$.
The notation $\cpm$ means an enumeration, so that \eqref{eq:system_np} in fact denotes two equations: one for the positive charges $c^+$ with a plus in front of the $\nabla \psi$ term and one for the negative charges $c^-$ with a minus  in front of the $\nabla \psi$.
Finally, the electric potential is given by the Poisson equation \eqref{eq:system_p}.

In isotropic fluids the movement of the charges and the formation of the electric field are independent of the direction.
Not so in anisotropic fluids, where the movement of the ions due to diffusion and electromigration, as well as the diffusion of the electric potential 
may depend on the direction. 
In our case they depend on the director $\bd$, which gives the preferred direction of motion.
Our choice  
% of anisotropy matrices comes 
for the matrices $\lamMa$ and $\varMa$ stems 
from the modeling of liquid crystals, most famously known for their application in LCDs.
One possible way to model liquid crystals are the Erickson--Leslie equations \cite{emmrich_klapp_lasarzik_2018}, 
including a time-evolution for the director $\bd$.
The mobilities of the charges vary depending on the motion being parallel or perpendicular to the director,
see \cite{calderer_2016} for an extensive model derivation for nematic electrolytes.
As it was done in \cite{calderer_2016} we choose the anisotropy matrices to be of the form $\varMa$ and $\lamMa$.
% In reality $\varepsilon$ and $\lambda$ may be negative or even have different signs,
% but we will not cover these cases here. 
% Typically the conductivity and thus the mobility of the ions is larger parallel to the directors, \cite[Chap.~5.3]{deGennes},
% and this justifies the assumption $\lambda > 0$.
% The assumption $\varepsilon>0 $ is made for mathematical reasons. 
For simplicity, we assume that  $\varepsilon>0 $ and $\lambda > 0$. There are certain materials for which $\varepsilon$ and $\lambda$ may be negative or even have different signs. 
Typically, the conductivity and thus the mobility of the ions is larger parallel to the directors \cite[Chap.~5.3]{deGennes}.
This justifies the assumption $\lambda > 0$.
For $\norm{\bd}_{L^\infty(\Omega)} \leq 1$, we could also treat the case $\varepsilon > -1 $, but refrain from this technical improvement, since it would further impede the estimates.

Additionally, the system is equipped with the following boundary conditions
\begin{gather*}
	\bv = 0 \; \text{ and } \; \lamMa\left( \nabla \cpm \pm \constNP \cpm \nabla \psi \right) \cdot \bs{n} = 0 \text{ on } \Gamma \times [0,T],\\
	\varMa \nabla \psi \cdot \bn + \tau \psi = \xi \text{ on } \Gamma \times [0,T],
\end{gather*}
where $\tau$ is a positive constant and
$\xi \in W^{1,2}(0,T;W^{2,2}(\Gamma))$ the externally applied electric potential.
These are the standard no-slip boundary conditions for the velocity field, no-flux boundary condition for the charge densities, 
which correspond to the assumption that the charges cannot cross the boundary of the domain and Robin boundary conditions for the electric potential.
As in \cite{bothe_fischer_saal_2014}, we choose Robin boundary conditions
to model the electric double layer, which usually forms at an electrolyte-solid-interface  \cite[Chap.~7]{newman_thomas_alyea_2004}.

%Motivation to look at this system + Numerik?
As described in \cite{calderer_2016} anisotropic fluids, like liquid crystals, have certain advantages to isotropic fluids in producing persistent flows 
under the influence of an alternating electric field.
In experiments it is possible to fix the director of the liquid crystal such that  an alternating field creates vortices in the fluid which enhance mixing \cite{peng_2015}. 
This was our motivation to consider the above system \eqref{eq:system}.
Our numerical simulations of system \eqref{eq:system} give qualitatively similar result to the observed flows in the experiments in \cite{peng_2015}.
The Python code used for the simulations is published on Zenodo \cite{zenodo}.
For example \cite[Fig.~4]{peng_2015} shows a given director field with the observed velocity field.
Using this director field, see Figure~\ref{fig:flow_sim_dir}, our simulation of the velocity field, \textit{cf.}~Figure~\ref{fig:flow_sim_sim}, is in good agreement with the observed one \cite[Fig.~4]{peng_2015}.
Figure~\ref{fig:flow_sim_sim} shows a tracer plot of the time-averaged velocity field.
\begin{figure}[h]\label{fig:flow_sim}
\centering 
\begin{subfigure}[b]{0.45\textwidth}
\includegraphics[height=5cm, width=\linewidth, trim = 23cm 13cm 22cm 13cm, clip]{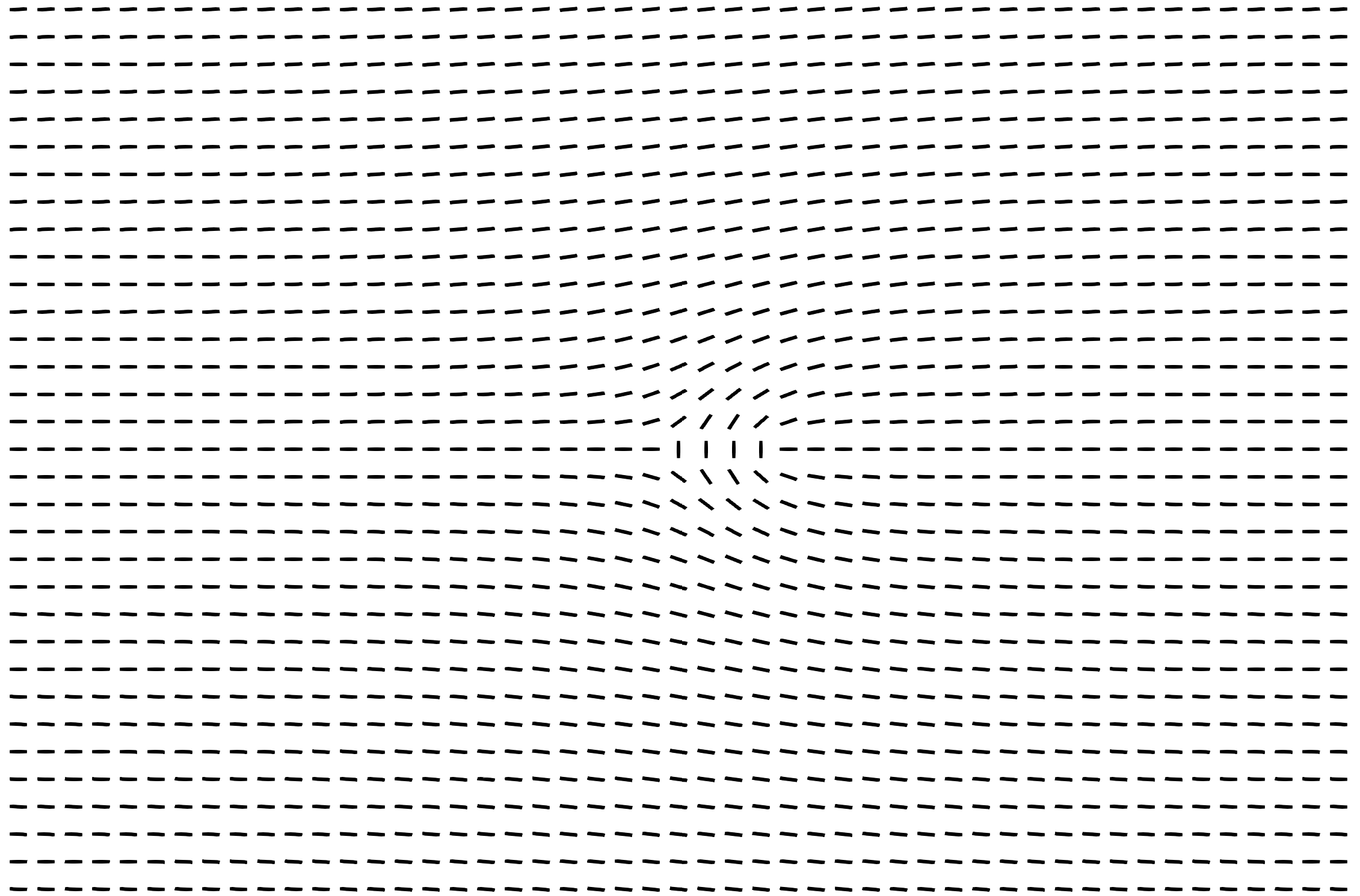}
\caption{Director field.}
\label{fig:flow_sim_dir}
\end{subfigure}\quad\quad
%\begin{subfigure}[b]{0.3\textwidth}
%\includegraphics[height=3.5cm, width=\linewidth]{./pictures/wind_around_egg_velo_physical.png}
%\caption{Velocity field \cite[Fig.~4]{peng_2015}.}
%\label{fig:flow_sim_velo}
%\end{subfigure}
\begin{subfigure}[b]{0.45\textwidth}
\includegraphics[height=5cm, width=\linewidth]{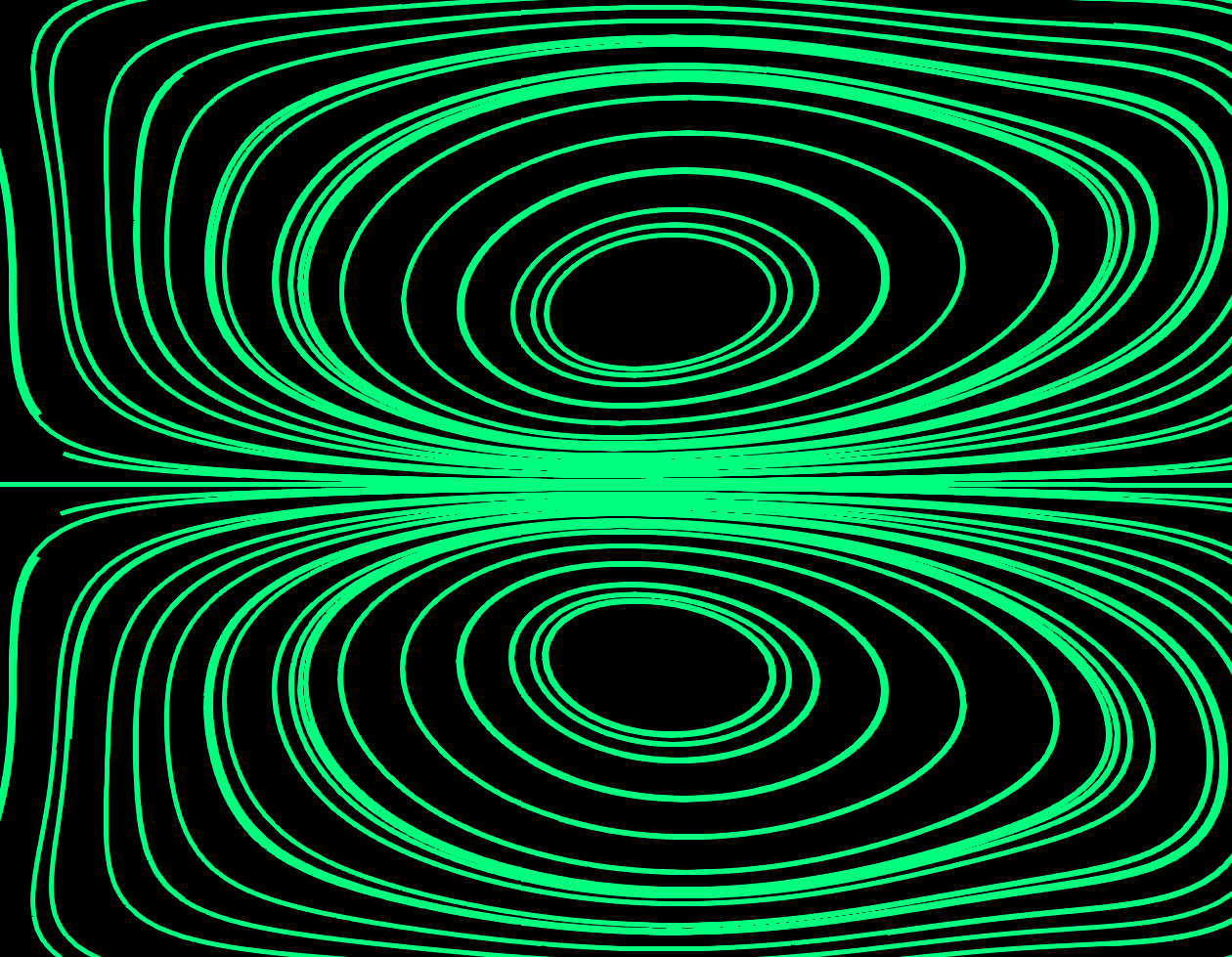}
\caption{Simulated velocity field.}
\label{fig:flow_sim_sim}
\end{subfigure}
\caption{Simulated velocity field for a given director field.}
\end{figure}
The simulation was performed via a finite element discretization using the Python package FEniCS \cite{fenics, fenics_book} 
and the figures were created with the visualization tool ParaView \cite{paraview}. 

%historical background
% In 1822 Claude-Louis Navier, \cite{navier_1822}, and then again George Gabriel Stokes in 1845, \cite{stokes_1845} introduced equations, 
% describing the motion of an incompressible fluid, the now well-known Navier--Stokes equations. 
In case of $c^{\pm}=0=\nabla \psi $, system~\eqref{eq:system} reduces to the well-known Navier--Stokes equations~\cite{Farwig}. 
In his seminal work, Jean Leray showed the existence of weak solutions for the Navier--Stokes equation in three space dimensions, 
which he termed ``turbulent solutions'' \cite{leray_1934}. 
In contrast to the two dimensional case, in three space dimensions these solutions may not be unique \cite{albritton_brue_colombo_2022}, 
but under additional smoothness assumption uniqueness can be shown, as it was already done by Leray himself \cite{leray_1934} and by
James Serrin in 1962 \cite{serrin_1962}.
As the Navier--Stokes equations are a sub-problem of  system \eqref{eq:system} we do not hope to attain uniqueness of our solutions, 
but instead show weak-strong uniqueness, which we will be done in a subsequent article.
The situation for the Nernst--Planck--Poisson equations is somewhat similar in the sense, 
that uniquness is known in dimension two but remains an open probem in dimension three. %Introduced in 1888 by Walther Nernst, \cite{nernst_1889}, and in 1890 by Max Planck, \cite{planck_1890}. 
%relation to more rescent research
In \cite{choi_lui_1995} the stability of steady-state solutions was proven in two dimensions 
and conditionally under the assumption of an $L^\infty(L^2)-$estimate for the ions in dimension three.
Today, the existence of weak solutions is well understood in all space dimensions.
In~\cite{bothe_fischer_pierre_rolland_2014}, the existence of weak solutions for a Nernst--Planck--Poisson system with multiple ion species 
with diffusion coefficients variable in time and space and bounded reactions is shown in all space dimension.
Here again the question of uniqueness remains an open problem \cite{bothe_fischer_pierre_rolland_2014}.

The coupled NSNPP was introduced by Rubinstein in 1990 \cite{rubinstein} 
and there is a myriad of works on this system with scalar diffusion coefficient.  
One of the first mathematically rigorous investigations of the system was performed in \cite{jerome_2002},
where the existence of local strong solutions to an NSNPP system with two ion species with constant scalar diffusion was proven via a semigroup approach.
In~\cite{bothe_fischer_saal_2014}, the local existence of strong solutions as well as the global existence of strong solutions in dimension two is proven
via maximal $L^p-$regularity for multiple ion species.
The two dimensional case was also considered in \cite{constantin_ignatova_2019}, where the existence of global strong solutions as well as convergence to a steady-state was proven
for non-homogeneous Dirichlet boundary conditions for the ions.
In~\cite{schmuck_2009}, the existence and uniqueness of weak solutions to an NSNPP system for two ion species under no-flux boundary conditions for the ions was proven
in dimension two.  
In~\cite{constantin_ignatova_lee_2021}, an NSNPP as well as a Stokes--Nernst--Planck--Poisson (SNPP) system was considered with Dirichlet boundary conditions for the electric potential and the local existence of strong solutions was proven. 
Additionally, the global existence of strong solutions for arbitrary many ion species with identical diffusion coefficients and for two ion species with possibly different diffusion coefficient was proven for the SNPP system and,
under additional regularity assumptions on
% conditionally on more regularity of 
the velocity field, also for the NSNPP system.
A similar result for Robin boundary conditions for the electric potential was obtained in \cite{lee_2022}.

All of the above mentioned results are on systems with scalar diffusion coefficients. 
An anisotropic Nernst--Planck--Poisson system similar to \eqref{eq:system_np}--\eqref{eq:system_p} coupled to Erickson--Leslie equations, modeling nematic liquid crystals, was considered in
\cite{banas_lasarzik_prohl_2021}. Here, the existence of dissipative solutions was proven, but the proof relied on the fact that the anisotropy matrix $\lamMa$ in the 
Nernst--Planck equations \eqref{eq:system_np} and the anisotropy matrix $\varMa$ in the Poisson equation \eqref{eq:system_p} are identical.
This property was also used in the weak-sequential stability proof in~\cite{FeireislRocca}. 

%highlight of our work
The inclusion of diffusion matrices in the Nernst--Planck and Poisson equations is indispensable to handle anisotropy
and to the best of the authors' such a general NSNPP system with diffusion matrices has not yet been considered.
We show that this anisotropic electrokinetic flow model admits a weak solution.
The difference between the anisotropy matrices $\varMa$ and $\lamMa$ is crucial for capturing the distinct physical properties of electric conductivity and electric permittivity,
which are inherently different physical properties. 
Furthermore, we establish that there are weak solutions that satisfy an energy-dissipation inequality. 
Energy dissipation is not only a physically relevant property but it is also essential for proving relative energy inequalities, 
which can be used to prove the weak-strong uniqueness of such solutions. 
This means that if a strong solution exists, any suitable weak solution emanating from the same initial data will coincide with that strong solution, ensuring its uniqueness within the class of suitable weak solutions.

We adapt the \textit{ansatz} from \cite{fischer_saal_2017} to show the existence of suitable weak solutions. 
Due to the differences in the anisotropy matrices new \textit{a priori} estimates 
for the first-order derivatives of the charges and the second-order derivatives of the electric potential are required.
These can be accomplished through integration by parts, necessitating careful control of the boundary integrals. 
This control is achieved using surface differential operators and performing integration by parts on the boundary.

The article is organized as follows. First, in the following section, Section~\ref{sec:main}, we introduce some notation and state our main result. 
Then, we recall some known results for our chosen regularization operators and Sobolev functions on the boundary of a domain, see Section~\ref{sec:pre}.  
In Section~\ref{sec:exist} we give the proof of the main result and auxiliary results are collected in the Appendix.
%\tableofcontents
\end{section}

\begin{section}{Main results} \label{sec:main}
We begin by introducing some basic notation.
By $\Omega$, we denote a smooth bounded domain in $\mathbb{R}^3$ and $\Gamma := \partial \Omega$. 
For all $r \in (1, \infty)$ we define the function spaces
\[
	L^r_\sigma(\Omega) := \overline{C^\infty_{0,\sigma}(\Omega)}^{\norm{\cdot}_{L^r}} \quad
	\text{ and } \quad W_0^{1,r}(\Omega) = \overline{C^\infty_{0}(\Omega)}^{\norm{\cdot}_{W^{1,r}}}.
\]
The function space $ L^ p(\Omega) _+$ is defined via $  L^ p(\Omega) _+ : = \{f \in L^p(\Omega)\mid f \geq 0 \text{ a.e.~in }\Omega\}$.
The norm $ |\cdot |_{\lamMa} $ is defined via $ | \bs a |_{\lamMa}^2 = | \bs a | ^2 + \lambda (\bs d \cdot \bs a)^2$ and similarly for $| \cdot |_{\varMa}$.
The outer normal of $\Omega$ is denoted by $\bs n$.
The standard matrix and matrix-vector multiplication is written without an extra sign for bre\-vi\-ty,
$$\bs A \bs B =\left [ \sum _{j=1}^3 \bs A_{ij}\bs B_{jk} \right ]_{i,k=1}^3 \,, \quad  \bs A \bs a = \left [ \sum _{j=1}^3 \bs A_{ij}\bs a_j \right ]_{i=1}^3\, , \quad  \bs A \in \R^{3\times 3},\,\bs B \in \R^{3\times3} ,\, \bs a \in \R^3 .$$
The outer vector product is given by
 $\bs a \otimes \bs b := \bs a \bs b^T = \left [ \bs a_i  \bs b_j\right ]_{i,j=1}^3$ for two vectors $\bs a , \bs b \in \R^3$ and by $ \bs A \otimes \bs a := \bs A \bs a ^T = \left [ \bs A_{ij}  \bs a_k\right ]_{i,j,k=1}^3 $ for a matrix $ \bs A \in \R^{3\times 3} $ and a vector $ \bs a \in \R^3$. 
 We use  the Nabla symbol $\nabla $  for real-valued functions $f : \R^3 \to \R$, vector-valued functions $ \bs f : \R^3 \to \R^3$ as well as matrix-valued functions $\bs A : \R^3 \to \R^{3\times 3}$ denoting
\begin{align*}
\nabla f := \left [ \pat{f}{\bs x_i} \right ] _{i=1}^3\, ,\quad
\nabla \bs f  := \left [ \pat{\bs f _i}{ \bs x_j} \right ] _{i,j=1}^3 \, ,\quad
\nabla \bs A  := \left [ \pat{\bs A _{ij}}{ \bs x_k} \right ] _{i,j,k=1}^3\, .
\end{align*}
 The divergence of a vector-valued and a matrix-valued function is defined by
\begin{align*}
\nabla \cdot \bs f := \sum_{i=1}^3 \pat{\bs f _i}{\bs x_i} \, , \quad  \nabla \cdot \bs A := \left [\sum_{j=1}^3 \pat{\bs A_{ij}}{\bs x_j}\right] _{i=1}^3\, .
\end{align*}
These definitions give rise to different calculation rules, \textit{e.g.~}$\nabla \cdot ( \bs a \otimes \bs b) = \nabla \bs a \bs b + \nabla \cdot \bs b \bs a $ for $\bs a,\bs b\in W^{1,2}(\Omega)$. 
The use of $\sigmp{\,\cdot\,}$ implies a summation: for example $\sigmp{\pm \cpm} = c^+ - c^-$,  $\sigmp{\mp \cpm} = - c^+ + c^-$, and 
\[
	\sigmp{\cpm (\ln \cpm + 1)} = c^+ (\ln c^+ + 1) + c^- (\ln c^- + 1).
\]
So the sum $\sigmp{\, \cdot \,}$ consists of two terms: one where only the upper signs of $\pm$ and $\mp$ are used, and one where only the lower signs are used.
Outside of these brackets $\pm$ is used as an enumeration, \textit{e.g.~}$\cpm \in L^1(\Omega)$ means $c^+ \in L^1(\Omega)$ and $c^- \in L^1(\Omega)$.

For a given Banach space $V$, the space $C_w([0,T];V )$ denotes the functions on $[0,T]$ taking values in $V$ that are continuous with respect to the weak topology of $V$.
%The Dual space of a Banach space $V$ is denoted by $V^*$, where the dual pairing is denoted %as $\langle \cdot , \cdot \rangle$.

In the case $\Gamma \in C^m$, $m \in \N$ we denote for $p \in (1, \infty)$ by 
		 $S: W^{m,p}(\Omega) \to W^{m-\frac{1}{p}, p}(\Gamma)$,
		the usual trace operator and by 
	 $E: W^{m-\frac{1}{p}, p}(\Gamma) \to W^{m,p}(\Omega)$ its right-inverse such that
		$S(E(f)) = f $ for all $f \in W^{m-\frac{1}{p}, p}(\Gamma)$. 
  
Throughout this work $C > 0$ denotes a generic constant, which may change its value without an indication in the notation. 
We sometimes use $C(\cdot)$ to indicate dependencies of this constant.

With these notations ar hand we come back to our PDE model and give its precise formulation:
\begin{subequations}\label{eq:main_system}
\begin{align}
	\partial_t \bv + (\bv \cdot \nabla) \bv 
		- \Delta \bv + \nabla p + (c^+ - c^-) \nabla \psi &= 0, \text{ in } \Omega \times (0,T), \label{eq:main_system_ns}\\
	\nabla \cdot \bv &= 0, \text{ in } \Omega \times (0,T), \label{eq:main_system_div}\\
	\partial_t c^{\pm} + \nabla \cdot (c^{\pm} \bv)
		- \nabla \cdot \left( \lamMa (
		\nabla c^{\pm} \pm c^{\pm} \nabla \psi) \right) &= 0, \text{ in } \Omega \times (0,T), \label{eq:main_system_np}\\
	- \nabla \cdot \left( \varMa
		\nabla \psi \right) - (c^+ - c^-)
		&= 0, \text{ in } \Omega \times (0,T), \label{eq:main_system_p}
\end{align}
\end{subequations}
equipped with initial and the following boundary conditions
\begin{gather*}
	\bv = 0 \; \text{ and } \; \lamMa\left( \nabla \cpm \pm \cpm \nabla \psi \right) \cdot \bs{n} = 0 \text{ on } \Gamma \times [0,T], \label{eq:bc_nsnp}\\
	\varMa \nabla \psi \cdot \bn + \tau \psi = \xi \text{ on } \Gamma \times [0,T].\label{eq:bc_p}
\end{gather*}
Our proof works for arbitrary positive constants $\Rey, \Pe, \constNS, \constNP$ and $\constP$ in~\eqref{eq:system}
but as these have no major impact on our method of proof we set them all to one for now to improve readability.
\begin{ass}
    \label{ass:1}
    We require that $ \Omega$ is a smooth domain in $\mathbb{R}^3$ with boundary $\Gamma := \partial \Omega$, the constants $\lambda$, $\varepsilon$, and $\tau$ are positive.
    The director field $\bd \in W^{1,\infty}(\Omega;\R^3) $ fulfills $\bd \cdot \bs n = 0$ on $\Gamma$, 
and the externally applied electric field $\xi \in W^{1,2}(0,T;W^{2,2}(\Gamma))$ .
\end{ass}

Using the same regularization as in \cite{fischer_saal_2017} we can prove the existence of suitable weak solutions, 
where we define suitable weak solutions as follows.

\begin{definition}[Suitable weak solutions]\label{def:weak_sol}
	Let $\Omega \subseteq \R^3$ be a bounded, smooth domain, $T>0$,
	$(\bv_0, \cpm_0) \in L_\sigma^2(\Omega) \times L^2(\Omega)_+$,
	$\bd \in W^{1,\infty}(\Omega)$ and $\xi \in W^{1,2}(0,T; W^{2,2}(\Gamma))$.
	We call $(\bv, \cpm, \psi)$ suitable weak solution if 
	\begin{align*}
		\bv &\in C_w([0,T];L^2_\sigma(\Omega)) \cap L^2(0,T;W_0^{1,2}(\Omega))\cap L^{10/3}(\Omega \times (0,T)),\\
		\cpm &\in C_w([0,T];L^1(\Omega)_+) \cap L^{5/4}(0,T;W^{1,5/4}(\Omega))\cap L^{5/3}(\Omega \times (0,T)),\\
		\sqrt{\cpm} &\in L^2(0,T;W^{1,2}(\Omega)),\\
		\psi &\in C_w([0,T];W^{1,2}(\Omega)) 
        \cap L^2(0,T;W^{2,2}(\Omega)),\\
		\sqrt{\cpm} \nabla \psi &\in L^2(0,T;L^2(\Omega)),
	\end{align*}
	and \eqref{eq:main_system} is fulfilled in the weak sense, 
	that is for all test functions $\tbv \in L^2(0,T,W^{1,2}_{0, \sigma}(\Omega)) \cap C^1(\overline{\Omega} \times [0,T])$, 
	$\tcpm \in C^1(\overline{\Omega} \times [0,T])$ 
	and $\tpsi \in W^{1,2}(\Omega) \cap L^\infty(\Omega)$ we have
	\begin{align}
		\int_\Omega (\bv \cdot \tbv)(t) - \bv_0 \cdot \tbv(0)\diff \X + 
		\int_0^t \int_\Omega  
			\nabla \bv : \nabla \tbv 
            - \bv \partial_t \tbv
			+ \left( (\bv \cdot \nabla) \bv
			+ (c^+ - c^-) \nabla \psi \right) \cdot \tbv 
		\diff \X \diff s &= 0\label{eq:weak_form_v}\\
		\int_\Omega (\cpm \; \tcpm)(t) - \cpm_0 \tcpm(0) \diff \X
		+ \int_0^t \int_\Omega
			\left( \lamMa (\nabla \cpm \pm \cpm \nabla \psi) - \cpm \bv \right) \cdot \nabla \tcpm 
            - \cpm \partial_t \tcpm 
		\diff \X \diff s &= 0\label{eq:weak_form_c}\\
		\int_\Omega 
			\varMa \nabla \psi \cdot \nabla \tpsi
		\diff \X
		+ \int_\Gamma 
			(\tau \psi - \xi) \tpsi
		\dS
		- \int_\Omega
			(c^+ - c^-) \tpsi
		\diff \X &= 0\label{eq:weak_form_psi}
	\end{align}
	for all $t \in [0,T]$ and additionally the energy inequality
	\begin{multline}\label{eq:strong_en_ineq_def}
		\left[ \int_\Omega \frac{1}{2} |\bv|^2 + \sigmp{\cpm (\ln \cpm + 1)} + \frac{1}{2} |\nabla \psi|^2_{\varMa} \diff \X
		+ \frac{\tau}{2} \int_\Gamma |\psi|^2 \dS \right] \bigg|_0^t\\
		+ \int_0^t \int_\Omega |\nabla \bv|^2 + \sigmp{ \left| 2 \nabla \sqrt{\cpm} \pm \sqrt{\cpm} \nabla \psi \right|_{\lamMa}^2} \diff \X \diff s
		\leq 
		\int_0^t \int_\Gamma \psi \partial_t \xi \dS \diff s
	\end{multline}
	holds for all $t \in [0,T]$.
\end{definition}

\begin{thm}[Existence of suitable  weak solutions]\label{thm:exist_weak_sol}
	Let Assumption~\ref{ass:1} be fulfilled. 
	For for every $T>0$ and all initial data $(\bv_0, \cpm_0) \in L_\sigma^2(\Omega) \times L^2(\Omega)_+$ there exists a suitable weak 
	solution according to Definition \ref{def:weak_sol} .
\end{thm}
%
% With the help of the energy inequality \eqref{eq:strong_en_ineq_def} we are also able to derive a relative energy inequality, 
% \textit{cf.~}Section~\ref{sec:rel_en}, from which the weak-strong uniqueness directly follows.
% %
% \begin{thm}[Weak-strong uniqueness]\label{thm:weak_strong}
% 	Let $(\bv, \cpm)$ be a weak energy-solution to \eqref{eq:main_system} according to Definition~\ref{def:weak_sol} and
% 	let $(\tbv, \tcpm) \in \left( C^1([0,T] \times \overline{\Omega}) \cap C([0,T]; C^2(\overline{\Omega})) \right)^2$ be 
% 	a classical solution to \eqref{eq:main_system} emanating from the same initial values with $\tcpm \geq l > 0$.
% 	Then $(\bv, \cpm)$  and $(\tbv, \tcpm)$ coincide.
% \end{thm}
% %
% To prove this result we will make use of the relative energy inequality proven in Section~\ref{sec:rel_en} and we therefore also conduct the proof in that section.
\begin{remark}
    We include the energy inequality in the definition of a suitable solution since it is an important property of a solution. 
    It is not only a property a solution should fulfill from a point of view of physics, but it also helps to prove the so-called relative energy inequality,
    which we will consider in a subsequent article. 
    The relative energy inequality is an important tool in the analysis of nonlinear evolution equations. As an easy consequence, it provides the weak-strong uniqueness of solutions, \textit{i.e.,} all weak solutions coincide with a local strong solution emanating from the same initial value as long as the latter exists. 
Moreover, the relative energy inequality has been used to derive singular limits~\cite{singular}, \textit{a posteriori} error estimates~\cite{fischermodelerror}, convergence of numerical approximations~\cite{banas_lasarzik_prohl_2021}, or optimal control~\cite{approx} for an associated solution concept~\cite{lasarzik_2019_dissipative_EL}.
\end{remark}
\begin{remark}
    The main novelty of this result is that we can include an anisotropy in the system, the prescribed director field $\bd$. The main technical difficulty is to prove \textit{a priori} bounds that are strong enough to deduce via compactness results strong converging sequences of our approximate scheme in order to identify the limit as a weak solution in the sense of Definition~\ref{def:weak_sol}. The approximate scheme follows the ideas of~\cite{fischer_saal_2017}, which we extend to the considered setting.  

    We do not consider the cases $ \tau =0$, $\lambda = 0$, or $ \varepsilon =0$, but these cases are significantly easier than the considered case. The cases $\lambda =0$ or $\varepsilon=0$ could be dealt with by combining our proof with the one conducted in~\cite{fischer_saal_2017}. For $\tau=0$ the system is equipped with inhomogeneous Neumann boundary conditions. In this case, the electric potential is only determined up to a constant and this constant should be fixed \textit{a priori}. 
\end{remark}

\end{section}

\begin{section}{Preliminaries}\label{sec:pre}

In order show the existence of weak solutions we regularize our system and let the regularization coefficient $\kappa$ tend to zero. 
We use elliptic regularization operators of the form $(I - \kappa A)^{-1}$, where $A$ is a generator of a bounded $C_0-$semigroup.
For the convenience of the reader we first recall some basic properties of operators of that form, which hold true for all generators
and then we recall results, showing that our chosen operatos, namely the anisotropic Robin Laplacian, the Stokes operator, 
and the root of the Stokes operator indeed generate bounded $C_0-$semigroups.

  \begin{subsection}{Basic properties of the regularization operator}\label{sec:reg_op}

	\begin{lem}[Properties of the regularization operator]\label{lem:prop_reg_op}
		Let $X$ be a Banach space and $A:D(A) \subseteq X \to X$ the infinitesimal generator of a bounded $C_0-$semigroup.
		For all $\kappa > 0$ the operator $R_\kappa := (I - \kappa A)^{-1}: X \to D(A) \subseteq X$ is a well-defined linear and bounded operator
		with the following properties
		\begin{enumerate}
			\item{
				For $\kappa \searrow 0$ and $x_\kappa \to x$ in $X$ we have $R_\kappa(x_\kappa) \to x$ in $X$.
			}\label{item:R_kappa_strong_cont}
			\item{
				For $\kappa \searrow 0$ and $x_\kappa \rightharpoonup x$ in $X$ we have $R_\kappa(x_\kappa) \rightharpoonup x$ in $X$.
			}\label{item:R_kappa_weak_cont}
			\item{
				There exists $C>0$ independent of $\kappa$ such that $\norm{R_\kappa(x)}_X \leq C \norm{x}_X$.
			}\label{item:R_kappa_uniform_bound}
			\item{
				For any Banach space $Y$ such that $D(A) \hookrightarrow Y$ there exists a constant $C > 0$ independent of $\kappa$
				such that for all $x \in X$ we have $\norm{R_\kappa(x)}_Y \leq C ( 1 + 1/\kappa) \norm{x}_X$.		
			}\label{item:R_kappa_bound}
		\end{enumerate}
	\end{lem}
	The proof is a rather straight forward and relies on the uniform bound on the operator norm of the resolvent given in the Hille--Yosida generation theorem, 
	\textit{cf.~}\cite[Thm.~3.5]{engel_nagel}.
	We deferred  the proof into  the Appendix, see Section~\ref{sec:app_add_proofs}
\end{subsection}
\begin{subsection}{The Stokes operator as a generator}

	\begin{definition}[Stokes operator]\label{def:stokes}
		We define the Stokes operator $A_r$ with domain 
		\[
			D(A_r):= W^{2,r}(\Omega) \cap W_0^{1,r}(\Omega) \cap L^r_\sigma(\Omega)
		\]
		as an unbounded operator in $L_\sigma^r(\Omega)$ via
		\[
			A_r: D(A_r) \subseteq L^r_\sigma(\Omega) \to L^r_\sigma(\Omega), \quad v \mapsto -  P_r(\Delta v),
		\]
		where $P_r$ is the $L^r-$Helmholtz projection, \textit{cf.~}\cite[Thm.~1.4]{simader_sohr}.
	\end{definition}
We collect several standard properties of the stokes operator 
\begin{lem}\label{Stokes}
    Let $r\in(0,\infty)$, we infer 
    \begin{enumerate}
        \item For all $r \in (1, \infty)$ the Stokes operator $-A_r$ generates a bounded analytic $C_0-$semigroup on $L^r_\sigma(\Omega)$, 
		denoted by $\left( \e^{- t A_r} \right)_{t \geq 0}$.\label{lem:stokes_semigroup}
  \item  \label{lem:frac_Stokes_semigroup}
		For all $\alpha \geq 0$ the fractional power Stokes operator $- A_r^\alpha$ also generates a bounded analytic $C_0-$semi\-group 
		$\left( \e^{- t A^{\alpha}_r}\right)_{t \geq 0}$ in $L^r_\sigma(\Omega)$. 
  \item Let $\Omega \subseteq \R^d$ with $d \in \N$, be a bounded domain with smooth boundary then 
		$D(A_r^{1/2}) = W_0^{1,r}(\Omega) \cap L^r_{\sigma}(\Omega)$.
  \item \label{lem:fract_stokes_self_adjoint}
		 The operator $A^{1/2}_2$ is symmetric.
    \end{enumerate}
\end{lem}
\begin{proof}
    These reults are classical. The first two points can be found in \cite[Thm.~2]{giga_stokes} and \cite[Prop.~1.2]{giga_miyakawa_1985}.
 	The last two points can be found in \cite[Lem.~2.2.1]{sohr}.
\end{proof}

	% The following classical result can be found in \cite[Thm.~2]{giga_stokes} for smooth domains and was generalized to $C^3-$domains in
	% \cite[Sec.~2.4 and Thm.~3]{noll_saal}.
	% %
	% \begin{lem}\label{lem:stokes_semigroup}
	% 	For all $r \in (1, \infty)$ the Stokes operator $-A_r$ generates a bounded analytic $C_0-$semigroup on $L^r(\Omega)$, 
	% 	denoted by $\left( e^{- t A_r} \right)_{t \geq 0}$.
	% \end{lem}
	% %
	% Next we turn to the fractional power Stokes operator. 
	% For a general introduction to fractional power operators see for example \cite{martinez_sanz}.
	% As for all fractional powers of generators of analytic semigroups also the fractional power Stokes operator again generates an analytic semigroup
	% at least for certain powers $\alpha$, 
	% see for example \cite[Thm.~5.4.1]{martinez_sanz}.
	% \begin{lem}\label{lem:frac_Stokes_semigroup}
	% 	For all $\alpha \in (0, 1/2]$ the fractional power Stokes operator $- A_r^\alpha$ also generates a bounded analytic $C_0-$semigroup 
	% 	$\left( e^{- t A^{\alpha}_r}\right)_{t \geq 0}$ in $L_r(\Omega)$. 
	% \end{lem}
	% %
	% We can explicitly give the domain of the root of the Stokes operator. \marginnote{Cite sth.?}
	% \begin{lem}
	% 	Let $\Omega \subseteq \R^d$ be a bounded domain with smooth boundary then 
	% 	$D(A_r^{1/2}) = W_0^{1,r}(\Omega) \cap L^r_{\sigma}(\Omega)$.
	% \end{lem}
	% %
	% The following well-known result follows from \cite[Lem.~2.2.1]{sohr}.
	% %
	% \begin{lem}\label{lem:fract_stokes_self_adjoint}
	% 	 The operator $A^{1/2}_2$ is symmetric.
	% \end{lem}
\end{subsection}
\begin{subsection}{The Robin Laplacian as a generator}

	\begin{definition}\label{def:robin_laplace}
	Let $\bd \in W^{1,\infty}(\Omega)$, we define the anisotropic Robin Laplacian $\Delta_{\varMa}$ in $L^2(\Omega)$ by
		\begin{align*}
			\Delta_{\varMa}: D(\Delta_{\varMa}) &:= 
			\set{\varphi \in W^{2,2}(\Omega)}{\varMa \nabla \varphi \cdot \bn + \tau \varphi = 0 \text{ on } \Gamma} 
			\subseteq L^2(\Omega) \to L^2(\Omega),\\ 
			\varphi &\mapsto \nabla \cdot \left( \varMa \nabla \varphi \right).
		\end{align*}
	\end{definition}
	The following result is quite standard and can for example be found in \cite[Thm.~3.1.3]{lunardi}.
	\begin{lem}[Robin semigroup in $L^2(\Omega)$]\label{lem:robin_semigroup_L2}
		The anisotropic Robin Laplacian from Definition~\ref{def:robin_laplace} generates a contraction $C_0-$semigroup in $L^2(\Omega)$.
	\end{lem}
	\begin{remark}\label{rem:C_indep_d}
		The constant $C$ from Lemma~\ref{lem:prop_reg_op} item~\ref{item:R_kappa_uniform_bound} and item~\ref{item:R_kappa_bound} 
		with $A = \Delta_{\varMa}$ can be chosen independently of $\bd$.
		This simply follows from the fact that $\Delta_{\varMa}$ generates a $C_0-$semigroup of contractions for all~$\bd\in W^{1,\infty}(\Omega)$.
	\end{remark}
	Following the proof of \cite[Thm.~8]{brezis_strauss} one can show the following Lemma, 
	which is not new but we could not find the result handling our boundary conditions in the literature.
	The proof of \cite[Thm.~8]{brezis_strauss} where the result was shown for Dirichlet boundary conditions instead of Robin boundary conditions,
	works completely analogous with only minor adaptions to accommodate our boundary conditions.
	\begin{lem}[Robin semigroup in $L^1(\Omega)$]\label{lem:robin_semigroup_L1}
		Let $\bd \in W^{1,\infty}(\Omega)$. We define the Robin Laplacian $A_1 = \Delta_{\varMa}$ in $L^1(\Omega)$ with domain 
		\[
			D(A_1) 
   % = D(\Delta^1_{\varMa}) 
   = \set{\varphi \in W^{1,1}(\Omega)}{A_1 \varphi \in L^1(\Omega)},
		\]
		where the Laplacian is defined in the sense of distribution, that is we say $A_1 \varphi = f$ in $L^1(\Omega)$, if
		\[
			- \int_\Omega \varMa \nabla \varphi \cdot \nabla w \diff \X - \tau \int_\Gamma \varphi w \dS = \int_\Omega f w \diff \X
 		\]
 		holds for all $w \in W^{1, \infty}(\Omega)$.
 		Then $A_1$ generates a contracting $C_0-$semigroup 
 		and we have the embedding $D(A_1) \hookrightarrow  W^{1,q}(\Omega) $ for all $q \in [1, 3/2)$
 		and for all $\varphi \in D(A_1)$ it holds
 		\begin{align}\label{eq:ellip_W1q_est_def}
 			\norm{\varphi}_{W^{1,q}(\Omega)} \leq C \norm{A_1 \varphi}_{L^1(\Omega)},
 		\end{align}
 		where the constant $C > 0$ may depend on the domain $\Omega$ but is independent of the director $\bd$.
	\end{lem}
	
	\begin{cor}[Elliptic $L^1-$estimate]\label{cor:ellip_W1q_est}
		Let $\kappa > 0 $ and $f \in L^1(\Omega)$ be arbitrary, 
		then 
		\[
			\left( I - \kappa 
   %\Delta^1_{\varMa}
   A_1
   \right)^{-1}: L^1(\Omega) \to D(A_1) \subseteq L^1(\Omega)
		\]
		is a well-defined, bounded, and linear operator and it holds
		\[
			\norm{\left( I - \kappa 
   % \Delta^1_{\varMa}
    A_1
   \right)^{-1} f}_{W^{1,q}(\Omega)} \leq C \frac{\norm{f}_{L^1(\Omega)}}{\kappa}
		\]
		for $C > 0$ independent of $\kappa$ and $\bd$.
	\end{cor}
	\begin{proof}
		By Lemma~\ref{lem:robin_semigroup_L1} we have that 
  % $\Delta^1_{\varMa}$ 
  $A_1$ generates a contraction $C_0-$semigroup
		and thus the well-definedness follows from Lemma~\ref{lem:prop_reg_op}.
		For the norm inequality we use the norm inequality \eqref{eq:ellip_W1q_est_def} from Lemma~\ref{lem:robin_semigroup_L1} above,
		\begin{align*}
			\norm{\left( I - \kappa 
   % \Delta^1_{\varMa}
   A_1\right)^{-1} f}_{W^{1,q}(\Omega)}
			&\leq
				% \frac{C}{\kappa} 
    C
				\norm{
    % \Delta^1_{\varMa} 
    A_1 \left( 
    % R \left( \frac{1}{\kappa}, \Delta^1_{\varMa} \right)
    (I-\kappa A_1)^{-1}
    f\right) }_{L^1(\Omega)}
			=
				% \frac{C}{\kappa}
    C \norm{\frac{1}{\kappa}
    % R \left( \frac{1}{\kappa}, \Delta^1_{\varMa} \right) f - f
    \left(\left( I - \kappa A_1\right)^{-1}-I \right)f 
    }_{L^1(\Omega)}\\
			&\leq 
				\frac{C \norm{f}_{L^1(\Omega)}}{\kappa} 
				\left( 
					% \frac{1}{\kappa} 
     \norm{
     % R \left( \frac{1}{\kappa}, \Delta^1_{\varMa} \right) 
     \left( I - \kappa A_1\right)^{-1}
     }_{\mathcal{L}(L^1(\Omega))} + 1 
				\right)
			\leq 
				\frac{C \norm{f}_{L^1(\Omega)}}{\kappa},
		\end{align*}
		where we used 
  % $AR(\lambda, A) = \lambda R(\lambda, A) - I$, 
  $ A_1 (I-\kappa A_1)^{-1} = \frac{1}{\kappa}( (I-\kappa A_1)^{-1} -I)$, which follows from $(I-\kappa A_1)^{-1}(I-\kappa A_1)= I  $, as well as
		% see for example \cite[Chap.~IV, below Def.~1.1]{engel_nagel}
		%and the fact that $A_1$ generates a contraction, \textit{cf.}~
		Lemma~\ref{lem:prop_reg_op} point 3. 
		This finishes the proof.
	\end{proof}

\end{subsection}
\begin{subsection}{Trace theorems}\label{subsec:trace_thm}
    We denote the trace operator by $S$ and the trace extension operator by $E$, \textit{cf.~}Section~\ref{sec:main}
and recall two results, which will be important tools in the  \textit{a priori} estimate.
First, one can relate the surface differential operators to the bulk ones.

\begin{thm}\label{thm:surf_grad_vs_tan_pro}
	Let $\Omega \subseteq \R^d$ be a bounded domain with $\Gamma \in C^{0,1}$ and $p \in (1,\infty)$. 
    Let $f \in W^{2,p}(\Omega)$ and $\bv \in W^{2,p}(\Omega)$ . 
    Then the surface gradient, denoted by $\nabla_\Gamma$, 
    is just the tangential projection of the bulk gradient, that is
	\[
		\nabla_\Gamma S(f) = S(\nabla f) - (S(\nabla f) \cdot \bn) \bn
        \quad \text{ and } \quad 
        \nabla_\Gamma S(\bv) = S(\nabla \bv) - S(\nabla \bv) \bn \otimes \bn
	\]
    and for the surface divergence $\nabla_\Gamma \cdot S(\bv)$ it holds
    \(
		\, \nabla_\Gamma \cdot S(\bv) 
			= 
				S(\nabla \cdot \bv) - S(\nabla \bv) \bn \cdot \bn.
	\)
\end{thm}
The proof is based on \cite[Thm.~4.2]{skrepek_2023} and is presented in Appendix.
Secondly, we recall an integration by parts rule on the boundary.

\begin{cor}\label{cor:ibp_boundary_main}
    Let $\Omega \subseteq \R^d$ be a bounded domain with $\Gamma \in C^{2,1}$ and $p \in (1,\infty)$.
	For a general (not necessarily tangential) vector field $\bv \in W^{1,p}(\Gamma)$ and $f \in W^{1,p'}(\Gamma)$
	the following integration by parts holds
	\[
		\int_\Gamma f \; \nabla_\Gamma \cdot \bv 
        - f (\bv \cdot \bn) \nabla_\Gamma \cdot \bn + \bv \cdot \nabla_\Gamma f \diff \sigma  = 0.
	\]
\end{cor}
The proof can be found in \cite[Sec.~2.1]{pruss_simonett_2016} and also in the Appendix, \textit{cf.~}Section~\ref{sec:bd_func}.

\end{subsection}
\end{section}

\begin{section}{Existence of suitable weak solutions}\label{sec:exist}

In this section we prove Theorem~\ref{thm:exist_weak_sol}, where we proceed as follows. 
We introduce a regularized system, using the resolvent of the Stokes operator and the Robin Laplacian as regularization operators.
We then show the existence of a weak solution to the regularized Nernst--Planck--Poisson subsystem via the Schaefer's fixed point theorem.
Next, we show the local existence of a weak solution to the coupled regularized Navier--Stokes--Nernst--Planck--Poisson system 
via semigroup theory and Banach's fixed point theorem.
Deriving appropriate energy estimates we can extend this local solution to a global one
and then extract a convergent subsequence for vanishing regularization, whose limit is a weak solution to the original system~\eqref{eq:main_system}.

\begin{subsection}{Our regularization operators}

	\begin{definition}\label{def:reg_op}
		Let $A_2$ be the $L_\sigma^2(\Omega)$ realization of the Stokes operator with $\nu = 1$. 
		For $\kappa > 0$ we define 
		\begin{align*}
			R_\kappa = \left( 1 + \kappa A_2 \right)^{-1}, \quad 
			R^{1/2}_\kappa := \left( 1 + \kappa A_2^{1/2} \right)^{-1}, \quad
			\text{ and } \quad S_\kappa = \left( 1 - \kappa \Delta_{\varMa} \right)^{-1}.
		\end{align*}
	\end{definition}
	
	\begin{remark}\label{lem:R_alpha_self_adjoint}
		Since $- A_2, - A_2^{1/2}$ and $\Delta_{\varMa}$ generate $C_0-$semigroups in $L_\sigma^2(\Omega)$ and $L^2(\Omega),$ respectively,
		the operators from Definition~\ref{def:reg_op} are well-defined, linear and bounded operators.
  		Moreover, the operator $R^{1/2}_\kappa$ is symmetric on $D(A_2^{1/2})$, which follows from a simple application of the symmetry of the root of the Stokes operator, \textit{cf.~}Lemma~\ref{Stokes} item~\ref{lem:fract_stokes_self_adjoint}.
	\end{remark}
	
	% \begin{lem}\label{lem:R_alpha_self_adjoint}
	% 	 The operator $R^{1/2}_\kappa$ is symmetric on $D(A_2^{1/2})$.
	% \end{lem}
	% \begin{proof}
	% 	This follows from a simple application of the symmetry of the root of the Stokes operator, \textit{cf.~}Lemma~\ref{Stokes} item~\ref{lem:fract_stokes_self_adjoint}. 
	% \end{proof}
\end{subsection}

\begin{subsection}{Local existence of a unique weak solution to a regularized system}

\begin{subsubsection}{The Nernst--Planck--Poisson subsystem}

We begin by fixing a velocity $\bv \in L^\infty(0,T;L^s(\Omega))$ for some $s > 3$.
Now we consider the regularized Nernst--Planck--Poisson system 
\begin{subequations}\label{eq:reg_npp}
\begin{align}
	\partial_t \cpm + \nabla \cdot \left( \cpm \bv - \lamMa\left( \nabla \cpm \pm \cpm \nabla \psi \right) \right) &= 0 
	&&\text{ in } \Omega \times (0,T) ,\label{eq:reg_np}\\
	- \nabla \cdot (\varMa \nabla \psi) &= S_\kappa(c^+ - c^-) =: \varphi &&\text{ in } \Omega \times (0,T) ,\label{eq:reg_p}
\end{align}
\end{subequations}
with the initial conditions $\cpm(0) = \cpm_0 \quad \text{ in } \Omega$ and the boundary conditions
\begin{equation}\label{eq:reg_bc}
	\left( \cpm \bv - \lamMa\left( \nabla \cpm \pm \cpm \nabla \psi \right) \right) \cdot \bn 
	= 0 \quad \text{and} \quad
	\varMa \nabla \psi \cdot \bn + \tau \psi = \xi \quad \text{ on } \Gamma \times [0,T],
\end{equation}
coupling the Nernst--Planck equation and the Poisson equation.
Our first goal is to show that this coupled system possesses a unique weak solution, 
which we will prove in the next proposition using the Schaefer's fixed point theorem.
For later reference we introduce the shorthand notation
\[
	\W(0,T) := L^2(0,T;W^{1,2}(\Omega)) \cap W^{1,2}(0,T;(W^{1,2}(\Omega))^*).
\]
\begin{propo}[Existence of weak solutions to system \eqref{eq:reg_npp}--\eqref{eq:reg_bc}]\label{lem:reg_exis_npp}
	Let $\Omega \subseteq \R^3$ be a bounded domain with $\Gamma \in C^{3,1}$, 
	$\bd \in C^{2,1}(\overline{\Omega})$,
	$\xi \in C^1([0,T];W^{3,2}(\Gamma))$, 
	$s>3$, 
	$T \in (0,\infty)$, 
	$\bv \in L^\infty(0,T;L^s(\Omega))$ 
	and $\cpm_0 \in L^2(\Omega)_+$.
	Then there exists a unique weak solution $(\cpm, \psi)$ such that
	\begin{gather*}
		\cpm \in \W(0,T), \quad \cpm \geq 0 \quad \text{and} \quad
		\psi \in C([0,T]; W^{4,2}(\Omega)),
	\end{gather*}
	equation \eqref{eq:reg_np} is fulfilled in the distributional sense and \eqref{eq:reg_p} is fulfilled pointwise.
\end{propo}%
\begin{remark}\label{rem:psi_c_cont}
	It is meaningful to require \eqref{eq:reg_p} to be fulfilled pointwise since 
	\(
		\cpm \in \W(0,T) \hookrightarrow C([0,T];L^2(\Omega))
	\)
	and then, by Lemma~\ref{lem:prop_reg_op}, $S_\kappa(c^+-c^-) \in C([0,T];W^{2,2}(\Omega)) \hookrightarrow C(\overline{\Omega} \times [0,T])$.
\end{remark}
The proof of Proposition~\ref{lem:reg_exis_npp} can be conducted by following the proof of \cite[Lem.~4.1]{fischer_saal_2017}.
% closely with minor adaptions to include the anisotropy matrix $\varMa$.
For completeness we 
will
% still feel the necessity to 
give a proof here but we keep
it
% the proof
rather short.
We would like to point out that this proof holds without any assumptions on $\nabla \cdot \bv$. 
\begin{proof}[Proof (of Proposition~\ref{lem:reg_exis_npp})]
	Let $r > 3$ and $X:= L^\infty(0,T;W^{1,r}(\Omega))$. For any $\widehat{\psi} \in X$ there exists a unique weak solution
	$\widehat{c}^{\pm} \in \W(0,T)$ 
	of \eqref{eq:reg_np} with $\psi$ replaced by $\widehat{\psi}$, see \cite[Thm.~8.30 and Thm.~8.34]{roubicek}.
	For any such $\widehat{c}^{\pm}$ there is a unique weak solution $\psi \in C([0,T];W^{4,2}(\Omega))$ to \eqref{eq:reg_p} 
	with right-hand side $S_\kappa(\widehat{c}^+ -  \widehat{c}^-)$, see \cite[Thm.~2.5.1.1]{grisvard}.
	Thus the iterative solution operator 
	\[
		\T: X \to X, \qquad \widehat{\psi} \mapsto \psi
	\]
	is well-defined.
	Further, the non-negativity of the initial condition $\cpm_0$ transfers to the solution $\widehat{c}^{\pm}$, which can be seen by testing \eqref{eq:reg_np} with 
	$\max(0, \widehat{c}^{\pm})$ using Gargliardo--Nirenberg's and Gronwall's inequalities in the same way as in~\cite[p.~1395]{constantin_ignatova_2019}.
	The continuity and compactness follows by testing an applying an Aubin--Lions lemma, see \cite[Cor.~4]{simon}.

	To show the existence of a weak solution to the coupled system \eqref{eq:reg_npp} we make use of Scheafer's fixed point theorem, \textit{cf.~}\cite[Sec.~9, Thm.~4]{evans}. 
	To that aim, we need to show that the set
	\[
		M := \set{\psi \in X}{\exists \sigma \in [0,1] : \sigma \T(\psi) = \psi}
	\]
	is bounded.
	So we take an arbitrary $\psi \in M$ and $\sigma \in [0,1] $ such that $\sigma \T(\psi) = \psi$.
	Let $\cpm$ be the solution to the Nernst--Planck equation \eqref{eq:reg_np}  corresponding to $\psi$. 
	To get a bound of $\cpm$ independent of $\psi$ we multiply \eqref{eq:reg_np} by $1$ and integrate to find
	\[
		\int_\Omega \cpm(t) \diff \X = \int_\Omega \cpm_0 \diff \X
	\]
	for almost all $t \in [0,T]$.
	By the non-negativity of $\cpm$ we find $\norm{\cpm}_{L^\infty(0,T;L^1(\Omega))} \leq \norm{\cpm_0}_{L^1(\Omega)}$.
	Using the elliptic $L^1-$regularity of the Robin Laplacian, \textit{cf.~}Lemma~\ref{lem:robin_semigroup_L1}, 
	we find that $S_\kappa(c^+ - c^-) = \varphi$
	is in $L^\infty(0,T;W^{1,q}(\Omega))$ for all $q \in [1, 3/2)$ and by Corollary~\ref{cor:ellip_W1q_est} and Sobolev's embedding theorem
	we find
	\[
		\norm{\varphi}_{L^\infty(0,T;L^t(\Omega))} \leq
		C \norm{\varphi}_{L^\infty(0,T;W^{1,q}(\Omega))} 
		\leq \frac{C}{\kappa} \norm{c^+ - c^-}_{L^\infty(0,T;L^1(\Omega))}
		\leq \frac{C}{\kappa} \sigmp{\norm{\cpm_0}_{L^1(\Omega)}}.
	\]
	for all $t \in [1,3)$.
	Then by \cite[Thm.~2.4 and 2.6]{grisvard} we find $\T(\psi) \in L^\infty(0,T;W^{2,t}(\Omega))$ and by classical Agmon–Douglis–Nirenberg
	estimates, see for example \cite[Thm.~3.1.1 iii)]{lunardi}, we have 
	\begin{align*}
		\norm{\T(\psi)}&_{ L^\infty(0,T;W^{2,t}(\Omega))}\\
		&\leq
			C \left( 
				\norm{- \nabla \cdot (\varMa\nabla \T(\psi))}_{L^\infty(0,T;L^t(\Omega))}
				+ \norm{\T(\psi)}_{L^\infty(0,T;L^t(\Omega))}
				+ \norm{E(\xi)}_{L^\infty(0,T;W^{1,t}(\Omega))}
			\right)\\
		&\leq 
			C \left( 
				\norm{\varphi}_{L^\infty(0,T;L^t(\Omega))}
				+ \norm{\xi}_{L^\infty(0,T;W^{1,3}(\Gamma))}
			\right)
		\leq
			\frac{C}{\kappa} \left( \sigmp{\norm{\cpm_0}_{L^1(\Omega)}} + 1 \right),
	\end{align*}
	where $E$ is the trace extension operator, see Section~\ref{sec:main}.
	Here we used that we find an $L^\infty(0,T;W^{1,2}(\Omega))-$bound for $\T(\psi)$ dependent only on $\varphi$ and $\xi$ 
	by testing \eqref{eq:reg_p} with $\T(\psi)$.
	Thus
	\[
		\norm{\psi}_{ L^\infty(0,T;W^{2,t}(\Omega))} 
		= \sigma \norm{\T(\psi)}_{ L^\infty(0,T;W^{2,t}(\Omega))}
		\leq \frac{C}{\kappa} \left( \sigmp{\norm{\cpm_0}_{L^1(\Omega)}} + 1\right).
	\]
	Since the choice of $\psi \in M$ was arbitrary and for all $ r \in (3,\infty)$ there is $t \in [1,3)$ close enough to three such that
	\[
		L^\infty(0,T;W^{2,t}(\Omega)) \hookrightarrow L^\infty(0,T;W^{1,r}(\Omega)),
	\]
	the boundedness of $M$ in $X$ follows and the Schaefer's fixed point theorem guarantees
	the existence of a fixed point of $\T$, which then solves our system \eqref{eq:reg_npp}.
	
	Finally, the uniqueness of this fixed point, that is the uniqueness of weak solutions to the system \eqref{eq:reg_npp},
	follows by taking two solutions $(\cpm, \psi)$ and $(\tcpm, \tpsi)$ emanating from the same initial data,
	subtracting \eqref{eq:reg_np} for $\cpm$ and $\tcpm$ and testing with the difference $(\cpm - \tcpm)$.
\end{proof}
To show the existence of a unique weak solution to the whole coupled system we would like to use Banach's fixed point theorem.
In order to be able to prove that the considered mapping is a contraction, 
we need an explicit formulation of some bounds for $\cpm$, which we will prove next.
With explicit we mean explicit in $\bv$.
\begin{lem}[Bounds for $\cpm$]\label{lem:bounds_for_c}
	Let $(\cpm, \psi)$ be the weak solution to \eqref{eq:reg_npp} with
	$\bv \in L^\infty(0,T;W_{0, \sigma}^{1,2}(\Omega))$.
	We then have 
	\begin{align}\label{eq:bounds_for_c}
		\norm{\cpm}_{L^\infty(0,T;L^2(\Omega))} +  \norm{\cpm}_{L^3(0,T;L^3(\Omega))} + \norm{\psi}_{L^\infty(0,T;W^{4,2}(\Omega))} 
		\leq C(\kappa, T),
	\end{align}
	for some function $C(\kappa, \cdot): (0,\infty) \to \R$ which is strictly positive, monotonically increasing, independent of $\bv$
	and $C(\kappa, T) < \infty$ for all $T < \infty$ and all $\kappa > 0$.
\end{lem}
The proof is based on elliptic regularity and a Gronwall argument and can be found in the Appendix. 
\end{subsubsection}

\begin{subsubsection}{The fully regularized system}

We consider the full regularized system, similar to \cite{fischer_saal_2017},
\begin{subequations}\label{eq:full_reg_system}
\begin{align}
	\partial_t \bv + A_2 \bv + P \left((R_\kappa(\bv) \cdot \nabla ) \bv \right)
		+ R^{1/2}_\kappa \left(P\left( (c^+ -c^-) \nabla \psi \right) \right) &= 0 &&\text{ in } \Omega \times (0,T) , \label{eq:full_reg_ns}\\
	\partial_t \cpm + \nabla \cdot \left( \cpm R^{1/2}_\kappa(\bv) - \lamMa\left( \nabla \cpm \pm \cpm \nabla \psi \right) \right) 
	&= 0 &&\text{ in } \Omega \times (0,T) ,\label{eq:full_reg_np}\\
	- \nabla \cdot (\varMa \nabla \psi) &= S_\kappa(c^+ - c^-) &&\text{ in } \Omega \times (0,T),\label{eq:full_reg_p}
\end{align}
\end{subequations}
equipped with the initial conditions $\bv(0) = \bv_0 \text{ and } \cpm(0) = \cpm_0 \text{ in } \Omega$
and the boundary conditions
\begin{equation}\label{eq:full_reg_system_bc}
		\bv = 0, \quad \left( \cpm \bv - \lamMa\left( \nabla \cpm \pm \cpm \nabla \psi \right) \right) \cdot \bn = 0,
		\quad \text{ and } \quad \varMa \nabla \psi \cdot \bn + \tau \psi = \xi 
		\quad \text{ on } \Gamma \times [0,T].
\end{equation}
We can show the existence of a unique local weak solution to system \eqref{eq:full_reg_system} via the variation of constants formula 
and the contraction mapping theorem.
We define 
\begin{align*}
	F_\kappa: D(A_2^{1/2}) &\to L^2(\Omega), \quad &&\bv \mapsto P\left( (R_\kappa(\bv) \cdot \nabla) \bv \right),\\
	G_\kappa: D(A_2^{1/2}) &\to L^2(\Omega), \quad &&\bv \mapsto R^{1/2}_\kappa \left(P\left( (c^+ -c^-) \nabla \psi \right) \right),
\end{align*}
where $(\cpm, \psi)$ in $G_\kappa$ denotes the unique solution to the Nernst--Planck--Poisson subsystem from Proposition~\ref{lem:reg_exis_npp}.
Additionally, we define the space 
\[
	X_T := C\left([0,T]; D(A^{1/2}_2) \right).
\]
Equipped with the norm $\norm{\cdot}_{X_T}$ given by
\[
	\norm{\bv}_{X_T} 
	:= 
		\sup_{t \in [0,T]} \left( \norm{\bv(t)}_{L^2(\Omega)} + \norm{A_2^{1/2} \bv(t)}_{L^2(\Omega)} \right)
	= 
		\sup_{t \in [0,T]} \left( \norm{\bv(t)}_{L^2(\Omega)} + \norm{\nabla \bv(t)}_{L^2(\Omega)} \right)
\]
$X_T$ is a Banach space, see \cite[Lem.~7.2.1]{emmrich}, where the equality follows from \cite[Lem.~2.2.1]{sohr}.
For $\bv_0 \in D(A_2^{1/2})$ we now define
\begin{equation}\label{eq:var_of_const}
	H_\kappa: X_T \to X_T, \quad \quad
	H_\kappa(\bv) 
	:= 
		\e^{- t A_2} \bv_0 - \int_0^t \e^{-(t-s)A_2} \left( 
			F_\kappa(\bv(s)) + G_\kappa(\bv(s))
		\right)\diff s,
\end{equation}
where the integral is understood in the Bochner sense with values in $L_\sigma^2(\Omega)$.
The integrand of \eqref{eq:var_of_const} is indeed Bochner integrable, since
$t \mapsto \e^{- t A_2} g$ is continuous from $[0,T]$ to $L_\sigma^2(\Omega)$ for all $g \in L_\sigma^2(\Omega)$ by the strong continuity of analytic semigroups, \cite[Prop.~4.3]{engel_nagel}
and $t \mapsto F_\kappa(\bv(t)) + G_\kappa(\bv(t))$ is continuous from $[0,T]$ to $L_\sigma^2(\Omega)$ by the boundedness of 
$P, R_\kappa$ and $R^{1/2}_\kappa$, \textit{cf.~}\cite[Thm.~1.4]{simader_sohr} and Lemma~\ref{lem:prop_reg_op}, respectively, the definition of $X_T$,
and the fact that by Remark~\ref{rem:psi_c_cont} we  have $(\cpm, \psi ) \in C([0,T];L^2(\Omega)) \times C([0,T];W^{4,2}(\Omega))$.
Furthermore, we note that $H_\kappa$ indeed maps to $X_T$. First, we see that for all $\bv \in X_T$ we have 
$H_\kappa(\bv)(t) \in D(A_2^{1/2})$,
since $H_\kappa(\bv)(0) = \bv_0 \in D(A_2^{1/2})$ and for $t>0$ we have $\e^{- t A_2}g \in D(A_2) \subseteq D(A^{1/2}_2)$ for all $g \in L^2_\sigma(\Omega)$, \textit{cf.}~\cite[Lem.~2.2.1]{sohr}. 
The continuity of $t \mapsto H_\kappa(\bv(t))$ from $[0,T]$ to $W^{1,2}(\Omega)$ follows from \cite[Prop.~4.2.1]{lunardi}.

For the existence of regularized solutions we take smoother data and collect all needed assumptions in the following:
\begin{ass}\label{ass:reg}
	We require $\Omega \subseteq \R^3$ to be a bounded smooth domain and take 
	${\bd \in W^{4,\infty}(\Omega)}$ with $\bd \cdot \bn = 0$ on $\Gamma$, 
	${\xi \in C^1([0,T];W^{3,10/3}(\Gamma))}$, $\varepsilon, \lambda > 0$, ${\cpm_0 \in L^2(\Omega)_+}$, and $\bv_0 \in D(A_2^{1/2})$. 
\end{ass}

\begin{lem}[Local weak solution to system \eqref{eq:full_reg_system}]\label{lem:contraction}
	Let $T_0 > 0$ and Assumption~\ref{ass:reg} hold, then there exists $M > \norm{\bv_0}_{D(A_2^{1/2})}$ and $T^\star  \in (0,T_0)$ such that $H_\kappa$ is a contraction on
	\[
		Z(M,T^\star) := \{\bv \in X_{T^\star} | \bv(0) = \bv_0, \; \norm{\bv}_{X_{T^\star}} \leq M \}
	\]
	and has a unique fixed point which solves the regularized system~\eqref{eq:full_reg_system}. Here  $T^\star$ may depend on $M$.
\end{lem}
This proof is very similar to the proof of \cite[Lem.~4.2]{fischer_saal_2017}.
% We simple restate here for completeness and make the small alteration necessary to adapt it to our anisotropy of the Nernst--Planck--Poisson subsystem.

\begin{proof}
	For $T \in (0,T_0)$ the set $Z(M,T)$ is a closed subset of the Banach space $X_T$ and thus 
	by Banach's fixed point theorem, \textit{cf.~}\cite[Thm.~A.2.2]{emmrich}, 
	it is enough to show that $H_\kappa$ is a self-map from $Z(M,T)$ to $Z(M,T)$ and contracting.
	We first show the self-map property. We note that for $\bv \in Z(M,T)$ we have $H_\kappa(\bv)(0) = \bv_0$, 
	thus the only thing we need to show is that $\norm{H_\kappa(\bv)}_{X_T} \leq M$.
	We start by estimating $F_\kappa$ and $G_\kappa$ separately. For almost all $t \in (0,T)$ we have
	\begin{align*}
		\norm{F_\kappa(\bv(t))}_{L^2(\Omega)}
		&\leq 
			\norm{\nabla \bv(t) R_\kappa(\bv(t))}_{L^2(\Omega)}
		\leq 
			\norm{\nabla \bv(t)}_{L^2(\Omega)}
			\norm{R_\kappa(\bv(t))}_{L^\infty(\Omega)}\\
		&\leq	
			C(\kappa) \norm{\nabla \bv(t)}_{L^2(\Omega)}
			\norm{\bv(t)}_{L^2(\Omega)}
		\leq
			C(\kappa) \left( \snorm{A_2^{1/2} \bv(t)}_{L^2(\Omega)}^2 + \norm{\bv(t)}_{L^2(\Omega)}^2 \right)\\
		&\leq 
			C(\kappa) M^2
	\end{align*}
	by Lemma~\ref{lem:prop_reg_op} since $D(A_2) \hookrightarrow W^{2,2}(\Omega) \hookrightarrow L^\infty(\Omega)$.
	Additionally, we have
	\begin{align*}
		\norm{G_\kappa(\bv(t))}_{L^2(\Omega)} 
		&\leq
			C \norm{(c^+(t) - c^-(t)) \nabla \psi(t)}_{L^2(\Omega)}\\
		&\leq
			C \sigmp{\norm{\cpm}_{L^\infty(0,T;L^2(\Omega))}} \norm{\psi}_{L^\infty(0,T;W^{4,2}(\Omega))}
		\leq 
			C(\kappa, T_0)
	\end{align*}
	by \eqref{eq:bounds_for_c}, 
	where the upper bound on the right is independent of $\bv$ and thus of $M$ but depends on the boundary and initial data and follows from Lemma \ref{lem:bounds_for_c}.
	Next, we estimate
	\begin{align*}
		\norm{H_\kappa(\bv(t))}_{L^2(\Omega)}
		&\leq
			\norm{\e^{- t A_2}}_{\mathcal{L}(L_\sigma^2(\Omega))} \norm{\bv_0}_{L^2(\Omega)}\\
			&\hspace{1em}+ \int_0^t 
				\snorm{\e^{- (t-s) A_2}}_{\mathcal{L}(L_\sigma^2(\Omega))} 
				\left( \norm{F_\kappa(\bv(s))}_{L^2(\Omega)} + \norm{G_\kappa(\bv(s))}_{L^2(\Omega)} \right)
			\diff s\\
		&\leq
			C_S \norm{\bv_0}_{L^2(\Omega)}
			+  T (C(\kappa)M^2 + C(\kappa,T_0)),
	\end{align*}
	where we used the boundedness of the semigroup generated by the Stokes operator, see Lemma~\ref{Stokes} item~\ref{lem:stokes_semigroup}, 
	which means that there is a $C_S > 0$ such that for all $t \in (0,\infty)$ we have $\norm{\e^{-t A_2}}_{\mathcal{L}(L^2(\Omega))} \leq C_S$.
	Further, we have
	\begin{align*}
		\snorm{A_2^{1/2} H_\kappa(\bv(t))}_{L^2(\Omega)}
		&\leq 
			\snorm{A_2^{1/2}\e^{- t A_2} \bv_0}_{L^2(\Omega)}\\
			&\hspace{1em}+ \int_0^t 
				\snorm{A_2^{1/2}\e^{- (t-s) A_2}}_{\mathcal{L}(L_\sigma^2(\Omega))} 
				\left( \norm{F_\kappa(\bv(s))}_{L^2(\Omega)} + \norm{G_\kappa(\bv(s))}_{L^2(\Omega)} \right)
			\diff s\\
		&\leq
			\snorm{\e^{- t A_2}A_2^{1/2} \bv_0}_{L^2(\Omega)}
			+ \int_0^t 
				\frac{1}{\sqrt{t-s}} \left( C(\kappa) M^2 + C(\kappa,T_0) \right)
			\diff s\\
		&\leq 
			C_S \snorm{A_2^{1/2} \bv_0}_{L^2(\Omega)} + \sqrt{T} \left( C(\kappa) M^2 + C(\kappa, T_0) \right),
	\end{align*}
	where we used the estimate 
	\begin{equation}\label{eq:est_frac_stokes_semigroup}
		\norm{A_2^{1/2}\e^{- t A_2}}_{\mathcal{L}(L_\sigma^2(\Omega))} \leq \frac{C}{\sqrt{t \; }}
	\end{equation}
	 for all $t > 0$ for the root of the Stokes operator, see for example \cite[Prop.~1.2]{giga_miyakawa_1985}
	and the fact that for every $g \in D(A_2^{1/2})$ we have 
	\begin{equation}\label{eq:interchange_A_alpha_e_tA}
		A_2^{1/2} \e^{- t A_2} g = \e^{- t A_2} A_2^{1/2} g,
	\end{equation}
	which follows from \cite[Chap.~IV, below Lem.~1.5.1]{sohr}.
	First choosing $M>0$ large enough such that
	\[
		C_S \left( \norm{\bv_0}_{L^2(\Omega)} + \snorm{A_2^{1/2} \bv_0}_{L^2(\Omega)} \right)  \leq \frac{M}{2}
	\]
	and then choosing $T^\star>0$ small enough so that 
	\[
		\Big( T^\star + \sqrt{T^\star}\, \Big) \left( C(\kappa) M^2 + C(\kappa, T_0) \right) \leq \frac{M}{2},
	\]
	the self-map property of $H_\kappa$ follows.
	% For theses choices of $M$ and $T$ the map $H_\kappa$ becomes a self-map.
	Next, we show that $H_\kappa$ is a contraction.
	For that we take $\bv_1, \bv_2 \in Z(M,T)$ and by Lemma~\ref{lem:prop_reg_op} we estimate
	\begin{align}\label{eq:contraction_F_est}
    \begin{split}
		&\norm{F_\kappa(\bv_1(t)) - F(\bv_2(t))}_{L^2(\Omega)}\\
		&\hspace{1em}\leq
			\norm{R_\kappa(\bv_1(t))}_{L^\infty(\Omega)} \norm{\nabla (\bv_1(t) - \bv_2(t))}_{L^2(\Omega)}
			+ \norm{R_\kappa(\bv_1(t)) - R_\kappa(\bv_2(t))}_{L^\infty(\Omega)} \norm{\nabla \bv_2(t)}_{L^2(\Omega)}\\
		&\hspace{1em}\leq
			C(\kappa) \norm{\bv_1(t)}_{L^2(\Omega)} \snorm{A_2^{1/2}(\bv_1(t) - \bv_2(t))}_{L^2(\Omega)}
			+ C(\kappa) M \norm{\bv_1(t) - \bv_2(t)}_{L^2(\Omega)}\\
		&\hspace{1em}\leq 
			C(\kappa) M \norm{\bv_1 - \bv_2}_{X_T}.
    \end{split}
	\end{align}
	Let $(\cpm_1, \psi_1)$ and $(\cpm_2, \psi_2)$ be the solutions to the Nernst--Planck--Poisson system associated to $\bv_1$ and $\bv_2$ respectively.
	Subtracting the Nernst--Planck equations for $\cpm_1$ and $\cpm_2$, testing with their difference, and adding the equations for the positve and negative charges,
	we obtain 
    \begin{align}\label{eq:L2_diff_c}
    \begin{split}
		&\sigmp{\frac{1}{2} \norm{\cpm_1(t) - \cpm_2(t)}^2_{L^2(\Omega)} 
		+ \int_0^t \int_\Omega |\nabla(\cpm_1 - \cpm_2)|^2_{\lamMa}\diff \X \diff s}\\
		&\hspace{1em}=
			\sigmp{\int_0^t \int_\Omega 
				\left( 
					(\cpm_1 R_\kappa^{1/2}(\bv_1) - \cpm_2 R_\kappa^{1/2}(\bv_2))
					\mp \left( \cpm_1 \lamMa \nabla \psi_1 - \cpm_2 \lamMa \nabla \psi_2 \right) 
				\right) \cdot \nabla (\cpm_1 - \cpm_2)
			\diff \X \diff s}\\
		&\hspace{1em}=
			\sigmp{\int_0^t \int_\Omega 
				\cpm_1 R_\kappa^{1/2}(\bv_1 - \bv_2) \cdot \nabla (\cpm_1 - \cpm_2)
				+ (\cpm_1 - \cpm_2) R_\kappa^{1/2}(\bv_2) \cdot \nabla (\cpm_1 - \cpm_2) 
			\diff \X \diff s}\\
			&\hspace{2em} + \sigmp{\mp \int_0^t \int_\Omega
				\cpm_1 \lamMa \nabla (\psi_1 - \psi_2 ) \cdot \nabla (\cpm_1 - \cpm_2)
				+ (\cpm_1 - \cpm_2) \lamMa \nabla \psi_2 \cdot \nabla (\cpm_1 - \cpm_2)
			\diff \X \diff s}\\
		&\hspace{1em}=
			\sigmp{\int_0^t \int_\Omega \cpm_1 R_\kappa^{1/2}(\bv_1 - \bv_2) \cdot \nabla (\cpm_1 - \cpm_2) \diff \X \diff s
			\mp \int_0^t \int_\Omega
				\cpm_1 \lamMa \nabla (\psi_1 - \psi_2 ) \cdot \nabla (\cpm_1 - \cpm_2)
			\diff \X \diff s}\\
			&\hspace{2em} + \sigmp{\mp \int_0^t \int_\Omega
				(\cpm_1 - \cpm_2) \lamMa \nabla \psi_2 \cdot \nabla (\cpm_1 - \cpm_2)
			\diff \X \diff s}\\
		&\hspace{1em}\leq
			\sigmp{\int_0^t 
				\norm{\cpm_1}_{L^3(\Omega)} 
				\norm{R_\kappa^{1/2}(\bv_1 - \bv_2)}_{L^6(\Omega)} 
				\norm{\nabla (\cpm_1 - \cpm_2)}_{L^2(\Omega)}
			\diff s}\\
			&\hspace{2em}+ \sigmp{\int_0^t
				\norm{\cpm_1}_{L^2(\Omega)} 
				\norm{\lamMa \nabla (\psi_1 - \psi_2)}_{L^\infty(\Omega)} 
				\norm{\nabla (\cpm_1 - \cpm_2)}_{L^2(\Omega)}
			\diff s}\\
			&\hspace{2em}+ \sigmp{\int_0^t
				\norm{\lamMa \nabla \psi_2}_{L^\infty(\Omega)}
				\norm{\cpm_1 - \cpm_2}_{L^2(\Omega)}
				\norm{\nabla (\cpm_1 - \cpm_2)}_{L^2(\Omega)}
			\diff s}\\
		&\hspace{1em}\leq
			\sigmp{\frac{1}{2} \int_0^t \norm{\nabla (\cpm_1 - \cpm_2)}^2_{L^2(\Omega)} \diff s
			+ \int_0^t 
				\norm{\cpm_1}^2_{L^3(\Omega)} 
				\norm{R_\kappa^{1/2}(\bv_1 - \bv_2)}^2_{L^6(\Omega)}
			\diff s}\\	
			&\hspace{2em}+ \sigmp{2 \int_0^t
				\norm{\cpm_1}^2_{L^2(\Omega)} 
				\norm{\lamMa \nabla (\psi_1 - \psi_2)}^2_{L^\infty(\Omega)} 
			\diff s}\\
			&\hspace{2em}+ \sigmp{2 \int_0^t
				\norm{\lamMa \nabla \psi_2}^2_{L^\infty(\Omega)}
				\norm{\cpm_1 - \cpm_2}^2_{L^2(\Omega)}
			\diff s},
    \end{split}
	\end{align}
	where we integrated by parts, see \cite[Cor.~8.1.10]{emmrich}, used that $\bv$ is divergence free and applied Young's inequality. 
	To upper bound the $L^\infty(\Omega)-$norm of $\nabla (\psi_1 - \psi_2)(t)$ we use elliptic estimates.
	Subtracting the equation for $\varphi_1$ and $\varphi_2$ and testing with the difference we find
	\begin{align}\label{eq:L2_diff_phi}
		\norm{(\varphi_1 - \varphi_2)(t)}_{L^2(\Omega)} \leq \sigmp{\norm{(\cpm_1 - \cpm_2)(t)}_{L^2(\Omega)}}
	\end{align}
	and thus with elliptic regularity, \textit{cf.}~\cite[Thm.~3.1.1]{lunardi},
	\begin{align}\label{eq:W22_diff_phi}
		\norm{(\varphi_1 - \varphi_2)(t)}_{W^{2,2}(\Omega)} 
		\leq 
			C(\kappa) \sigmp{\norm{(\cpm_1 - \cpm_2)(t)}_{L^2(\Omega)}}.
	\end{align}
	Subtracting \eqref{eq:full_reg_p} for $\psi_1$ and $\psi_2$ and testing with the difference, we obtain 
    \begin{equation}\label{eq:L2_diff_psi}
		\int_\Omega |\nabla (\psi_1 - \psi_2)(t)|^2_{\varMa} + \tau \int_\Gamma |(\psi_1 - \psi_2)(t)|^2 \dS
		\leq	
			\norm{(\varphi_1 - \varphi_2)(t)}_{L^2(\Omega)} \norm{(\psi_1 - \psi_2)(t)}_{L^2(\Omega)}.
	\end{equation}
	Using the generalized Friedrich's inequality, \textit{cf.~}\cite[Lem.~2.5]{fredi}, 
	\[
		\norm{f}^2_{W^{1,2}(\Omega)} = \int_\Omega |\nabla f|^2 + |f|^2 \diff \X 
		\leq 
			C_F \left( \int_\Omega |\nabla f|^2 \diff \X + \int_\Gamma |f|^2 \dS\right)
	\]
	for $f \in W^{1,2}(\Omega)$, the inequality \eqref{eq:L2_diff_psi} gives us
	\begin{align*}
		&\frac{\min(1, \tau)}{C_F} \norm{(\psi_1 - \psi_2)(t)}^2_{L^2(\Omega)}\\
		&\hspace{1em}\leq 
			\min(1, \tau) \left( \int_\Omega |\nabla (\psi_1 - \psi_2)(t)|^2 + \int_\Gamma |(\psi_1 - \psi_2)(t)|^2 \dS \right)\\
		&\hspace{1em}\leq
			\norm{(\varphi_1 - \varphi_2)(t)}_{L^2(\Omega)} \norm{(\psi_1 - \psi_2)(t)}_{L^2(\Omega)}
		\leq
			C \sigmp{\norm{(\cpm_1 - \cpm_2)(t)}_{L^2(\Omega)}} \norm{(\psi_1 - \psi_2)(t)}_{L^2(\Omega)}\\
		&\hspace{1em}\leq
			C \frac{C_F}{2 \min(1, \tau)} \sigmp{\norm{(\cpm_1 - \cpm_2)(t)}^2_{L^2(\Omega)}} 
			+ \frac{\min(1,\tau)}{2 C_F} \norm{(\psi_1 - \psi_2)(t)}^2_{L^2(\Omega)}
	\end{align*}
	by Young's inequality and \eqref{eq:L2_diff_phi}.
	Absorbing the last term on the right-hand side into the left-hand side gives us
	\begin{align*}
		\norm{(\psi_1 - \psi_2)(t)}_{L^2(\Omega)} 
		\leq 
			C \sigmp{\norm{(\cpm_1 - \cpm_2)(t)}_{L^2(\Omega)}}
	\end{align*}
	and again using \cite[Thm.~3.1.1]{lunardi} together with \eqref{eq:L2_diff_phi} we find
	\begin{align}\label{eq:W22_diff_psi}
    \begin{split}
		\norm{(\psi_1 - \psi_2)(t)}_{W^{2,2}(\Omega)} 
        &\leq 
			C \left(  \norm{(\varphi_1 - \varphi_2)(t)}_{L^2(\Omega)}  + \norm{(\psi_1 - \psi_2)(t)}_{L^2(\Omega)} \right)\\
		&\leq 
			C \sigmp{\norm{(\cpm_1 - \cpm_2)(t)}_{L^2(\Omega)}},
    \end{split}
	\end{align}
	where there is no boundary term on the right-hand side, since $\psi_1 - \psi_2$ fulfills Robin boundary conditions with zero right-hand side.
	Using the the higher order elliptic estimate \cite[Rem.~2.5.1.2]{grisvard}, we obtain
	\begin{align}\label{eq:W42_diff_psi}
    \begin{split}
		\norm{\nabla (\psi_1 - \psi_2)(t)}_{L^\infty(\Omega)} 
		&\leq 
			C \norm{(\psi_1 - \psi_2)(t)}_{W^{4,2}(\Omega)}\\
		&\leq		
			C \left( \norm{(\psi_1 - \psi_2)(t)}_{W^{2,2}(\Omega)} + \norm{(\varphi_1 - \varphi_2)(t)}_{W^{2,2}(\Omega)} \right)\\
		&\leq 
			C(\kappa) \sigmp{\norm{(\cpm_1 - \cpm_2)(t)}_{L^2(\Omega)}},
    \end{split}
	\end{align}
	where we used \eqref{eq:W22_diff_psi} and \eqref{eq:W22_diff_phi} for the last inequality.
	Using Young's inequality in \eqref{eq:L2_diff_c} to absorb the gradient terms of $\cpm_1 - \cpm_2$ into the left-hand side 
	and inserting the bounds from Lemma~\ref{lem:bounds_for_c},
	we obtain
	\begin{align*}
		&\sigmp{\frac{1}{2} \norm{(\cpm_1 - \cpm_2)(t)}^2_{L^2(\Omega)} 
		+ \frac{1}{2} \int_0^t \int_\Omega |\nabla(\cpm_1 - \cpm_2)|^2_{\lamMa}\diff \X \diff s}\\
		&\hspace{1em}\leq
			C(\kappa) \sigmp{\norm{\cpm_1}^2_{L^2(0,T;L^3(\Omega))}} \norm{\bv_1 - \bv_2}^2_{L^\infty(0,T;L^2(\Omega))}\\
			&\hspace{2em}+ C(\kappa) \norm{\lamMa}^2_{L^\infty(\Omega)} 
				 \sigmp{\left(\norm{\cpm_1}^2_{L^\infty(0,T;L^2(\Omega))} 
				+ \norm{\nabla \psi_2}^2_{L^\infty(0,T;L^\infty(\Omega))}
			\right) \int_0^t \norm{\cpm_1 - \cpm_2}^2_{L^2(\Omega)} \diff s}\\
		&\hspace{1em}\leq
			C(\kappa, T_0) \norm{\bv_1 - \bv_2}^2_{X_T} 
			+ C(\kappa, T_0) \int_0^t \sigmp{\norm{\cpm_1 - \cpm_2}^2_{L^2(\Omega)}} \diff s, 
	\end{align*}
	and with Gronwall's inequality \cite[Lem.~7.3.1]{emmrich} we get 
	\begin{equation}\label{eq:c_v_est}
		\sigmp{\norm{(\cpm_1 - \cpm_2)(t)}^2_{L^2(\Omega)}} 
            \leq C(\kappa, T_0) \norm{\bv_1 - \bv_2}^2_{X_T} \e^{C (\kappa, T_0) T_0 }.
	\end{equation}
	Now, we can estimate $G_\kappa$ by
	\begin{align}\label{eq:contraction_G_est}
    \begin{split}
		&\norm{G_\kappa(\bv_1(t)) - G_\kappa(\bv_2(t))}_{L^2(\Omega)}
		\leq 
			C \norm{((c^+_1 - c^-_1) \nabla \psi_1 - (c^+_2 - c^-_2) \nabla \psi_2)(t)}_{L^2(\Omega)}\\
		&\hspace{1em}\leq
			C \norm{
				((c^+_1 - c^+_2) \nabla \psi_1 
				+ c^+_2 \nabla (\psi_1 - \psi_2) 
				+ (c^-_2 - c^-_1) \nabla \psi_1 
				- c_2^- \nabla (\psi_1 - \psi_2))(t)}_{L^2(\Omega)}\\
		&\hspace{1em}\leq 
			C \sigmp{
				\norm{(\cpm_1 - \cpm_2) \nabla \psi_1}_{L^2(\Omega)}
				+ \norm{\cpm_2 \nabla( \psi_1 - \psi_2 )}_{L^2(\Omega)}
			}\\
		&\hspace{1em}\leq 
			C \left(
				\norm{\nabla \psi_1}_{L^\infty(0,T;L^\infty(\Omega))} \sigmp{\norm{(\cpm_1 - \cpm_2)(t)}_{L^2(\Omega)}}
				+ \sigmp{\norm{\cpm_2}_{L^\infty(0,T;L^2(\Omega))}} \norm{\nabla(\psi_1 - \psi_2)(t)}_{L^\infty(\Omega)}
			\right)\\
		&\hspace{1em}\leq
			C(\kappa, T_0) \norm{\bv_1 - \bv_2}_{X_T} \e^{C (\kappa, T_0) T_0 },
    \end{split}
	\end{align}
	where we used \eqref{eq:bounds_for_c}, \eqref{eq:W42_diff_psi} and \eqref{eq:c_v_est}.
	Putting \eqref{eq:contraction_F_est} and \eqref{eq:contraction_G_est} together we get the contraction property of $H_\kappa$ for
	\[
		T^\star \in (0,T_0) \text{ such that } 
		\left( C(\kappa, T_0) \e^{C (\kappa, T_0) T_0 } + C(\kappa) M \right) \left( T^\star + \sqrt{T^\star} \right) \leq \frac{1}{2},
	\]
	which follows from
	\begin{align*}
		&\norm{H_\kappa(\bv_1) - H_\kappa (\bv_2)}_{X_T}\\
		&\hspace{1em}=
			\sup_{t \in [0,T]} \left( \norm{H_\kappa(\bv_1(t)) - H_\kappa(\bv_2(t))}_{L^2(\Omega)}
			+ \norm{A^{1/2}_2\left( H_\kappa(\bv_1(t)) - H_\kappa(\bv_2(t)) \right) }_{L^2(\Omega)} \right)\\
		&\hspace{1em}\leq
			C_S \int_0^T 
				\norm{F_\kappa(\bv_1(t)) - F(\bv_2(t))}_{L^2(\Omega)}
				+ \norm{G_\kappa(\bv_1(t)) - G_\kappa(\bv_2(t))}_{L^2(\Omega)} 
			\diff t\\
			&\hspace{2em}+ C \int_0^T 
				\frac{1}{\sqrt{t-s}}\left( 
					\norm{F_\kappa(\bv_1(t)) - F(\bv_2(t))}_{L^2(\Omega)}
					+ \norm{G_\kappa(\bv_1(t)) - G_\kappa(\bv_2(t))}_{L^2(\Omega)}
				\right)
			\diff t\\
		&\hspace{1em}\leq 
			\left( C(\kappa, T_0) \e^{C (\kappa, T_0) T_0 } + C(\kappa) M \right)  \left( T + \sqrt{T} \right) \norm{\bv_1 - \bv_2}_{X_T}
		\leq 
			\frac{1}{2} \norm{\bv_1 - \bv_2}_{X_T},
	\end{align*}
	where we again used the boundedness of the semigroup generated by the Stokes operator, 
	the interchanging of the Stokes semigroup and the root of the Stokes operator, \textit{cf.~}\eqref{eq:interchange_A_alpha_e_tA},
	and the estimate for the root of the Stokes operator~\eqref{eq:est_frac_stokes_semigroup}.
	\end{proof}

% \begin{cor}
%     \label{lem:reg_weak_sol}
% 	The fixed point from Lemma~\ref{lem:contraction} is a weak solution of the regularized system \eqref{eq:full_reg_system}.
% \end{cor}
% \begin{proof}
% 	Let $\bv^* \in X_T$ be the unique fixed point of $H_\kappa$ given by Lemma~\ref{lem:contraction} and
% 	\[
% 		(\cpm, \psi) \in L^2(0,T;W^{1,2}(\Omega)) \times L^\infty(0,T;W^{4,2}(\Omega))
% 	\]
% 	the associated weak solution to the Nernst--Planck--Poisson subsystem \eqref{eq:full_reg_np}--\eqref{eq:full_reg_p},
% 	given by Proposition~\ref{lem:reg_exis_npp} choosing the $\bv$ from the Lemma to be 
% 	$\bv = R^{1/2}(\bv^*) \in L^\infty(0,T;W^{1,2}(\Omega)) \hookrightarrow L^\infty(0,T;L^6(\Omega))$.
% 	For $(\bv^*, \cpm)$ to be a weak solution to \eqref{eq:full_reg_system} we still need to show that the 
% 	Navier--Stokes equation \eqref{eq:full_reg_ns} is fulfilled in the weak sense.
% 	This is a simple consequence of the variation of constants formula \eqref{eq:var_of_const} and \cite[Prop.~4.15]{lunardi}, 
% 	by which we have
% 	\begin{equation}\label{eq:fp_mild_sol}
% 		\bv^*(t) 
% 		= 
% 			\bv_0 - A_2 \int_0^t \bv^*(s) \diff s 
% 			+ \int_0^t  f(s) \diff s,
% 	\end{equation}
% 	for all $t \in [0,T]$ with $f(s) := P( (R_\kappa(\bv^*) \cdot \nabla) \bv^*) + R^{1/2}_\kappa (P((c^+ - c^-) \nabla \psi ))$. 
% 	Testing \eqref{eq:fp_mild_sol} with $\tbv \in C^1([0,T];C^\infty_{0,\sigma}(\Omega))$ with $\tbv(T, \cdot ) \equiv 0$ 
% 	and integrating by parts yields the result. 
% \end{proof}
\end{subsubsection}

\end{subsection}

\begin{subsection}{Energy estimates}

\begin{propo}[Energy inequality I for  \eqref{eq:full_reg_system}]\label{propo:reg_en_ineq_1}
	Let Assumption~\ref{ass:reg} hold and let $(\bv, \cpm, \psi)$ be the weak solution to \eqref{eq:full_reg_system} on the maximal time interval $[0,T_{\textrm{max}})$ 
	for some $\kappa > 0$ 
	given by Lemma~\ref{lem:contraction}. Then the regularized energy
	\begin{equation}\label{eq:reg_en}
		\E_{\mathrm{reg}}(\bv, \cpm, \psi) 
			:= 
				\int_\Omega \frac{1}{2} |\bv|^2 + \sigmp{\cpm (\ln \cpm + 1)} + \frac{1}{2} |\nabla \psi|^2_{\varMa} 
				+ \frac{\kappa}{2} |\varphi|^2 \diff \X
				+ \frac{\tau}{2} \int_\Gamma |\psi|^2 \dS
	\end{equation}
	fulfills the inequality
	\begin{multline}\label{eq:reg_en_ineq_1}
		\E_{\mathrm{reg}}(\bv, \cpm, \psi)(t) + \int_0^t \Diss(\bv, \cpm, \psi) \diff s\\
		\leq 
			\e^t \left( \E_{\mathrm{reg}}(\bv_0, \cpm_0, \psi_0) 
			+ C \left( 
				\norm{\xi}^2_{W^{1,2}(0,\Tmax;L^2(\Gamma))}
				+ \norm{\xi}_{W^{1,1}(0,\Tmax;L^\infty(\Gamma))}
				+ \Tmax
			\right) \right)
	\end{multline}
	for all $t \in (0,\Tmax)$, where $C > 0$ is a constant independent of $\xi, \bd$ and $\kappa$ and  
	\[
		\Diss(\bv, \cpm, \psi)  = \int_\Omega  |\nabla \bv|^2 
		+ \sigmp{ \big| 2 \nabla \sqrt{\cpm} \pm \sqrt{\cpm} \nabla \psi \big|^2_{\lamMa}} \diff \X.
	\]
\end{propo}
\begin{proof}
	We start off by testing \eqref{eq:full_reg_ns} with $\bv$. 
	Since $\bv \in X_{T_{\textrm{max}}}$ we find by using the equation \eqref{eq:full_reg_ns} that 
	$\partial_t \bv \in L^\infty(0,T_{\textrm{max}}, W^{-1,2}(\Omega))$
	and thus by the weak* density of $C^\infty_c(0,T_{\textrm{max}}) \otimes D(A^{1/2}_2)$ 
	in $L^\infty(0,T_{\textrm{max}}; D(A^{1/2}_2))$ \cite[Lem.~A.2.2]{dimitri},
	we are allowed to test \eqref{eq:full_reg_ns} with $\bv$ it self.
	By an integration by parts, see for example \cite[Cor.~8.1.10]{emmrich}, we obtain
	\[
		\frac{1}{2} \norm{\bv}^2_{L^2(\Omega)} \bigg|_0^t
		+ \int_0^t \int_\Omega  |\nabla \bv|^2 
		+ R_\kappa^{1/2}(P((c^+ - c^- ) \nabla \psi)) \cdot \bv \diff \X \diff s = 0
	\]
	 for almost all $t \in [0,T_{\textrm{max}})$. 
	 Since $R^{1/2}_\kappa$ is symmetric on $D(A^{1/2}_2)$ due to Remark~\ref{lem:R_alpha_self_adjoint}, we can rewrite this as 
	 \begin{equation}\label{eq:full_reg_ns_tested}
	 	\frac{1}{2} \norm{\bv}^2_{L^2(\Omega)} \bigg|_0^t
		+\int_0^t \int_\Omega  |\nabla \bv|^2
		+ (c^+ - c^- ) \nabla \psi \cdot R^{1/2}_\kappa(\bv) \diff \X \diff s = 0,
	 \end{equation}
	 where we can omit the Helmholtz projection since $R_\kappa^{1/2}(\bv)$ is already divergence free.
	 Next, we test \eqref{eq:full_reg_np} with $(\ln(\cpm + \delta) + 2 \pm \psi)$ for some $\delta \in (0,1)$, 
	 which is a well-defined test function since $\cpm \geq 0$ by Proposition~\ref{lem:reg_exis_npp}.
	 By maximal $L^p-$regularity for the diffusion part of the Nernst--Planck equations \eqref{eq:full_reg_np} \cite[Thm.~2.1]{denk_hieber_pruess},
	 and since
	 \[
	 	\partial_t \cpm - \nabla \cdot (\lamMa \nabla \cpm) 
	 	= 
	 		- \nabla \cdot \left( \cpm R^{1/2}_\kappa(\bv) \right) 
	 		\pm \nabla \cdot \left( \cpm \lamMa \nabla \psi \right) \in L^{3/2}(0,T_{\textrm{max}};L^{3/2}(\Omega))
	 \]
	 we find $\partial_t \cpm \in L^{3/2}(0,T_{\textrm{max}};L^{3/2}(\Omega))$
	 and
   by Lemma~\ref{lem:bounds_for_c}, we have
   $\cpm \in L^3(0;T_{\textrm{max}};L^3(\Omega))$ 
	 and $\psi \in L^\infty(0,T_{\textrm{max}};W^{4,2}(\Omega)) \hookrightarrow L^3(0,T_{\textrm{max}};L^3(\Omega))$ such that 
	 we know that $(\ln(\cpm + \delta) + 2 \pm \psi)$ is an admissible test function.
	 We add the equations for $c^+$ and $c^-$ to obtain
	 \begin{align}\label{eq:full_reg_np_tested_delta_pre}
        \begin{split}
	 	\int_0^t \int_\Omega &\sigmp{\partial_t \cpm (\ln(\cpm + \delta) + 2)} \diff \X \diff s
	 	+ \int_0^t \int_\Omega \partial_t (c^+ - c^-) \psi \diff \X  \diff s\\
	 	&- \int_0^t \int_\Omega \sigmp{\frac{\cpm}{\cpm + \delta} R^{1/2}_\kappa(\bv) \cdot \nabla \cpm} \diff \X \diff s
	 	- \int_0^t \int_\Omega (c^+ - c^-) R^{1/2}_\kappa(\bv) \cdot \nabla \psi \diff \X \diff s\\
	 	&+ \int_0^t \int_\Omega \sigmp{\lamMa (\nabla \cpm \pm \cpm \nabla \psi) \cdot \nabla (\ln(\cpm + \delta ) \pm \psi)} \diff \X \diff s
	 	= 0.
        \end{split}
	 \end{align}
	 We note that 
	 \begin{gather*}
	 	\nabla \cpm = 2 \sqrt{\cpm + \delta} \; \nabla \sqrt{\cpm + \delta}, \qquad
	 	\sqrt{\cpm + \delta} \; \nabla \ln (\cpm + \delta) = 2 \nabla \sqrt{\cpm + \delta},\\
	 	\partial_t \left( (\cpm + \delta) (\ln(\cpm + \delta) + 1 )\right) = \partial_t \cpm (\ln(\cpm + \delta) + 2),
	 \end{gather*}
	 which holds by the product and the chain rule for weak derivatives, \cite[Prob.~21.3\,d)+e)]{zeidler2}.
	 % Since the map $x \mapsto \ln(x + \delta)$ for $x \geq 0$ can be extended to a continuous differentiable function with bounded derivative 
	 % on the whole $\R$ we are indeed allowed to apply the chain rule here.
	 Now, we can rewrite \eqref{eq:full_reg_np_tested_delta_pre} as
	 \begin{align}\label{eq:full_reg_np_tested_delta}
        \begin{split}
	 	&\sigmp{\int_\Omega (\cpm + \delta) (\ln(\cpm + \delta) + 1) \diff \X \bigg|_0^t} 
	 	+ \int_0^t \int_\Omega \partial_t (c^+ - c^-) \psi
	 	- \sigmp{\frac{\cpm}{\cpm + \delta} R^{1/2}_\kappa(\bv) \cdot \nabla \cpm} \diff \X \diff s\\
	 	&- \int_0^t \int_\Omega 
	 		(c^+ - c^-) R^{1/2}_\kappa(\bv) \cdot \nabla \psi 
	 	\diff \X \diff s
	 	+ \int_0^t \int_\Omega
	 		\sigmp{\big| 2 \nabla \sqrt{\cpm + \delta} \pm \sqrt{\cpm + \delta } \nabla \psi \big|^2_{\lamMa}}
	 	\diff \X \diff s\\
	 	&+ \sigmp{ \mp \int_0^t \int_\Omega 	 
	 		\left( \frac{\delta}{ \sqrt{\cpm + \delta} } \right) \lamMa \nabla \psi \cdot 
	 		\left( 2 \nabla \sqrt{\cpm + \delta} \pm \sqrt{\cpm + \delta } \nabla \psi \right)
	 	\diff \X \diff s }
	 	= 0. 
    \end{split}
	\end{align}
	Using H\"{o}lder's and Young's inequalities, we find that $\nabla \sqrt{\cpm + \delta}$ is bounded in $L^2(0,T_{\textrm{max}};L^2(\Omega))$
	independently of~$\delta$, since
	\begin{align*}
		&\hspace{-1em} \sigmp{\left( \int_\Omega (\cpm + \delta) (\ln(\cpm + \delta) + 1) + \frac{1}{\e^2} \diff \X \right) (t)
	 	+ \int_0^t \int_\Omega
	 		\big| 2 \nabla \sqrt{\cpm + \delta}\big|^2_{\lamMa}
	 		+ (\cpm + \delta) |\nabla \psi|_{\lamMa}^2
	 	\diff \X \diff s}\\
	 	&\leq
	 		\sigmp{\int_\Omega (\cpm_0 + \delta) (\ln(\cpm_0 + \delta) + 1) + \frac{1}{\e^2} \diff \X
	 		+ \norm{\partial_t \cpm}_{L^{3/2}(0,\Tmax;L^{3/2}(\Omega))} \norm{\psi}_{L^3(0,T;L^3(\Omega))}}\\
	 		&\hspace{1em}+ \sigmp{\norm{\frac{\cpm}{\cpm + \delta}}_{L^\infty(\Omega \times (0,\Tmax))} 
	 		\norm{R^{1/2}_\kappa(\bv)}_{L^2(0,\Tmax;L^2(\Omega))} \norm{\nabla \cpm}_{L^2(0,\Tmax;L^2(\Omega))}}\\
	 		&\hspace{1em}+ \sigmp{\norm{\cpm}_{L^2(0,\Tmax;L^2(\Omega))}} \norm{R^{1/2}_\kappa(\bv)}_{L^\infty(0,\Tmax;L^2(\Omega))} 
	 		\norm{\nabla \psi}_{L^2(0,\Tmax;L^\infty(\Omega))}\\
	 		&\hspace{1em}+ \int_0^t \int_\Omega \sigmp{ 2 |\nabla \cpm \cdot \nabla \psi|_{\lamMa}^2
	 		+ \delta |\nabla \psi|_{\lamMa}^2
	 		+ \frac{1}{2} \left( \frac{\delta}{ \sqrt{\cpm + \delta} } \right)^2 |\lamMa \nabla \psi|^2
	 		+ \frac{1}{2} \big|2 \nabla \sqrt{\cpm + \delta}\big|^2} \diff \X \diff s.
	\end{align*}
	The last term on the right-hand side can be absorbed into the left-hand side and all other terms are bounded independently of $\delta$, 
	since $\cpm_0 \in L^2(\Omega)$.
	Thus we have that $\nabla \sqrt{\cpm + \delta}$ is bounded in $L^2(\Omega \times (0,\Tmax))$ independently of $\delta$ and so, 
	by the reflexivity of this space, we find 
	$\nabla \sqrt{\cpm + \delta} \rightharpoonup a$ in $L^2(\Omega \times (0,\Tmax))$ for $\delta \searrow 0$ and some 
	$a \in L^2(\Omega \times (0,\Tmax))$.
	Since $\sqrt{\cpm + \delta} \to \sqrt{\cpm}$ pointwise almost everywhere for $\delta \searrow 0$ and 
	$|\sqrt{\cpm + \delta}| \leq \sqrt{\cpm + 1} \in L^2(\Omega \times (0,\Tmax))$, 
	since we chose $\delta \in (0,1)$, we find by the dominated convergence theorem
	%\begin{multline*}
	%	\int_0^{T_{\textrm{max}}} \int_\Omega a_i \; b \diff \X \diff t
	%	= 
	%		\lim_{\delta \searrow 0} \int_0^{T_{\textrm{max}}} \int_\Omega \partial_i \sqrt{\cpm + \delta} \; b \diff \X \diff t\\
	%	= 
	%		\lim_{\delta \searrow 0} - \int_0^{T_{\textrm{max}}} \int_\Omega  \sqrt{\cpm + \delta} \; \partial_i b \diff \X \diff t
	%	=
	%		- \int_0^{T_{\textrm{max}}} \int_\Omega  \sqrt{\cpm} \;  \partial_i b \diff \X \diff t
	%\end{multline*}
	%for $i = 1, \dots, d$ and all $b \in C^\infty_0((0,T_{\textrm{max}}) \times \Omega)$ and thus we can infer 
    $a = \nabla \sqrt{\cpm}$.
	Now, we can pass to the limit with $\delta \searrow 0$ in \eqref{eq:full_reg_np_tested_delta}. 
	We use dominated convergence for the term including $\delta$ and $\bv$ in the first line since $|\cpm/(\cpm + \delta)| \leq 1$.
	We use a generalized Fatou's lemma for functions bounded from below on finite domains, \cite[Chap.~5, Ex~5.4]{elstrodt2018}, 
	for the terms in the first line dependent on $t$.
	Using $\cpm_0 \in L^2(\Omega)$ we can apply dominated convergence to find
	\[
		\lim_{\delta \searrow 0} \int_\Omega (\cpm_0 + \delta) (\ln(\cpm_0+\delta) + 1) \diff \X = \int_\Omega \cpm_0 (\ln \cpm_0 + 1) \diff \X.
	\]
	Using $\ln(x) \leq x - 1$ and $x (\ln(x) + 1) \geq -1/\e^2$ for all $x > 0$ we can find a majorant by estimating,
	\[
		\left| (\cpm_0 + \delta) (\ln(\cpm_0+\delta) + 1) + \frac{1}{\e^2} \right| 
		\leq 
			(\cpm_0 + \delta)^2 + \frac{1}{\e^2} 
		\leq 
		\left( \cpm_0 \right)^2 + 2 \cpm_0 + 1 + \frac{1}{\e^2} \in L^1(\Omega).
	\]
	The integral in the last line of \eqref{eq:full_reg_np_tested_delta} vanishes for $\delta \searrow 0$, since 
	$|\delta/\sqrt{\cpm + \delta}| = \sqrt{\delta} \; \big| \sqrt{\delta}/\sqrt{\cpm + \delta} \big| \leq \sqrt{\delta}$ and thus
	\begin{align*}
		& \left| \sigmp{ \mp \int_0^t \int_\Omega 	 
	 		\left( \frac{\delta}{ \sqrt{\cpm + \delta} } \right) \lamMa \nabla \psi \cdot 
	 		\left( 2 \nabla \sqrt{\cpm + \delta} \pm \sqrt{\cpm + \delta } \nabla \psi \right)
	 	\diff \X \diff s } \right|\\
	 	&\hspace{1em}\leq 
	 		\norm{\lamMa}_{L^\infty(\Omega)} 
	 		\sigmp{
	 			\left( \delta + \frac{\sqrt{\delta}}{2} \right) \norm{\nabla \psi}^2_{L^2(0,T_{\textrm{max}}; L^2(\Omega))}
	 			+ \frac{\sqrt{\delta}}{2} \norm{2 \nabla \sqrt{\cpm + \delta}}^2_{L^2(0,T_{\textrm{max}}; L^2(\Omega))} 
	 		}\\
	 	&\hspace{1em}\leq C \left( \delta + \sqrt{\delta} \right),
	\end{align*}
	since weakly convergent sequences are bounded.
	Thus we obtain by the weak lower semicontinuity of the norm, taking $\delta \searrow 0$ in \eqref{eq:full_reg_np_tested_delta}
	\begin{multline}\label{eq:full_reg_np_tested}
	 	\int_\Omega \sigmp{\cpm (\ln \cpm+ 1)} \diff \X \bigg|_0^t 
	 	+ \int_0^t \int_\Omega \partial_t (c^+ - c^-) \psi \diff \X \diff s
	 	- \int_0^t \int_\Omega \sigmp{R^{1/2}_\kappa(\bv) \cdot \nabla \cpm} \diff \X \diff s\\
	 	- \int_0^t \int_\Omega 
	 		(c^+ - c^-) R^{1/2}_\kappa(\bv) \cdot \nabla \psi 
	 	\diff \X \diff s
	 	+ \int_0^t \int_\Omega
	 		\sigmp{\big| 2 \nabla \sqrt{\cpm}  \pm \sqrt{\cpm} \nabla \psi \big|^2_{\lamMa}}
	 	\diff \X \diff s
	 	\leq 0,
	\end{multline}
	and since $R^{1/2}_\kappa(\bv)$ is divergence free the last term on the first line vanishes.
	To rewrite the remaining term including the time derivative of $\cpm$ we differentiate \eqref{eq:full_reg_p} in time and test with $\psi$.
	By maximal $L^p-$regularity we already have $\partial_t \cpm \in L^{3/2}(0,\Tmax; L^{3/2}(\Omega))$. 
	We then find $\varphi \in W^{1,3/2}(0,\Tmax; W^{2,3/2}(\Omega))$
	and $\psi \in W^{1,3/2}(0,\Tmax; W^{4,3/2}(\Omega))$
	by elliptic regularity, see for example \cite[Thm.~2.5.1.1]{grisvard}
	and thus we can perform the following integration by parts, 
	\begin{multline}\label{eq:full_reg_p_tested_1}
		\int_0^t \int_\Omega \partial_t (c^+ - c^-) \psi \diff \X \diff s 
		= 
		\int_0^t \int_\Omega 
			\varMa \nabla \partial_t \psi \cdot \nabla \psi 
			+ \kappa \varMa \nabla \partial_t \varphi \cdot \nabla \psi 
		\diff \X \diff s\\
		- \int_0^t \int_\Gamma 
			\psi \left( \varMa \nabla \partial_t \psi + \kappa \varMa \nabla \partial_t \varphi \right) \cdot \bn 
		\dS \diff s.
	\end{multline}
	We can also differentiate the Robin boundary conditions for $\psi$ and $\varphi$ in time and since $\bd$ is independent of time we obtain
	\begin{equation}\label{eq:time_deriv_bc}
		\varMa \nabla \partial_t \psi \cdot \bn + \tau \partial_t \psi = \partial_t \xi \quad
		\text{ and } \quad \varMa \nabla \partial_t \varphi \cdot \bn + \tau \partial_t \varphi = 0
		\quad \text{ on } \Gamma.
	\end{equation}
	Testing \eqref{eq:full_reg_p} with $\partial_t \varphi$ we obtain
	\begin{equation}\label{eq:full_reg_p_tested_2}
		\int_0^t \int_\Omega 
			\varMa \nabla \psi \cdot \nabla \partial_t \varphi
		\diff \X \diff s 
		- \int_0^t \int_\Gamma
			\partial_t \varphi \; \varMa \nabla \psi \cdot \bn 
		\dS \diff s
		= \frac{1}{2} \int_\Omega |\varphi|^2 \diff \X \bigg|_0^t.
	\end{equation}
	Plugging \eqref{eq:time_deriv_bc} and \eqref{eq:full_reg_p_tested_2} back into \eqref{eq:full_reg_p_tested_1} we obtain
	\begin{align}\label{eq:full_reg_p_tested_final}
    \begin{split}
		\int_0^t &\int_\Omega \partial_t (c^+ - c^-) \psi \diff \X \diff s\\
		&= 
			\frac{1}{2} \int_\Omega |\nabla \psi|^2_{\varMa} \diff \X \bigg|_0^t 
			+ \frac{\kappa}{2} \int_\Omega |\varphi|^2 \diff \X \bigg|_0^t
			+  \int_0^t \int_\Gamma
				\kappa \partial_t \varphi (\xi - \tau \psi)
				- \psi (\partial_t \xi - \tau \partial_t \psi - \tau \kappa \partial_t \varphi ) 
			\dS \diff s\\
		&= 
			\left( \frac{1}{2} \int_\Omega |\nabla \psi|^2_{\varMa} \diff \X 
			+ \frac{\kappa}{2} \int_\Omega |\varphi|^2 \diff \X 
			+ \frac{\tau}{2} \int_\Gamma |\psi|^2 \dS \right) \bigg|_0^t
			+ \int_0^t \int_\Gamma 
				\kappa \partial_t \varphi \xi - \psi \partial_t \xi
			\dS \diff s.
   \end{split}
	\end{align}
	Inserting \eqref{eq:full_reg_p_tested_final} into \eqref{eq:full_reg_np_tested} and adding \eqref{eq:full_reg_ns_tested} we obtain
	\begin{multline}\label{eq:full_reg_tested_added}
		\left( 
			\int_\Omega 
				\frac{1}{2} |\bv|^2 + \sigmp{\cpm (\ln \cpm + 1)} + \frac{1}{2} |\nabla \psi|^2_{\varMa} + \frac{\kappa}{2} |\varphi|^2
			\diff \X 
			+ \frac{\tau}{2}\int_\Gamma |\psi|^2 \dS
		\right) \bigg|_0^t
		+ \int_0^t \Diss(\bv, \cpm, \psi) \diff s\\
		\leq 
			- \int_0^t \int_\Gamma 
				\kappa \partial_t \varphi \xi - \psi \partial_t \xi
			\dS \diff s
		= 
			\int_0^t \int_\Gamma 
				\psi \partial_t \xi 
				+ \kappa \partial_t \xi \varphi
			\dS \diff s
			- \kappa \int_\Gamma \xi \varphi \dS \bigg|_0^t
	\end{multline}
	By mass conservation,
	\[
		\int_\Omega \cpm(t) \diff \X = \int_\Omega \cpm_0 \diff \X,
	\]
	and elliptic $L^1-$regularity of the Robin Laplacian, 
	\textit{cf.~}Lemma~\ref{lem:robin_semigroup_L1} together with Corollary~\ref{cor:ellip_W1q_est}, we have
	\[
		\norm{\varphi(t)}_{W^{1,1}(\Omega)} \leq \frac{1}{\kappa} \sigmp{\norm{\cpm(t)}_{L^1(\Omega)}}
		\leq \frac{1}{\kappa} \sigmp{\norm{\cpm_0}_{L^1(\Omega)}}
		\leq \frac{C}{\kappa}
	\]
	and thus by the trace theorem for $W^{1,1}(\Omega)$, \textit{cf.~}\cite[Thm.~18.18]{leoni_2017},
    we have $\norm{\varphi(t)}_{L^1(\Gamma)} \leq C/\kappa$.
	Plugging this back into \eqref{eq:full_reg_tested_added} we obtain
	\begin{align}\label{eq:full_reg_tested_added_estimated}
    \begin{split}
		&\hspace{-1em}\left( 
			\int_\Omega 
				\frac{1}{2} |\bv|^2 + \sigmp{\cpm (\ln \cpm + 1)} + \frac{1}{2} |\nabla \psi|^2_{\varMa} + \frac{\kappa}{2} |\varphi|^2
			\diff \X 
			+ \frac{\tau}{2}\int_\Gamma |\psi|^2 \dS 
		\right) \bigg|_0^t
		+ \int_0^t \Diss(\bv, \cpm, \psi) \diff s\\
		&\leq 
			C \norm{\xi(0)}_{L^\infty(\Gamma)} + C \norm{\xi(t)}_{L^\infty(\Gamma)}
			+ \int_0^t 
				\frac{\tau}{2} \norm{\psi}^2_{L^2(\Gamma)} 
				+ \frac{1}{2 \tau} \norm{\partial_t \xi}^2_{L^2(\Gamma)}
				+ C \norm{\partial_t \xi}_{L^\infty(\Gamma)}
			\diff s\\
		&\leq 
			C \left( 
				\norm{\xi}_{C([0,\Tmax];L^\infty(\Gamma))} 
				+ \norm{\xi}^2_{W^{1,2}(0,\Tmax;L^2(\Gamma))}
				+ \norm{\xi}_{W^{1,1}(0,\Tmax;L^\infty(\Omega))}
			\right)
			+ \int_0^t 
				\frac{\tau}{2} \norm{\psi}^2_{L^2(\Gamma)} 
			\diff s\\
		&\leq 
			C \left( 
				\norm{\xi}^2_{W^{1,2}(0,\Tmax;L^2(\Gamma))}
				+ \norm{\xi}_{W^{1,1}(0,\Tmax;L^\infty(\Omega))}			
			\right)
			+ \int_0^t 
				\E_{\mathrm{reg}}(\bv, \cpm, \psi)(s) + \frac{2|\Omega|}{\e^2}
			\diff s,
    \end{split}
	\end{align}
	with $C>0$ independent of $\kappa, \bd$ and $\xi$, 
	where we added the constant $\frac{2|\Omega|}{\e^2}$ to make the energy non-negative ($\cpm(\ln \cpm + 1) + 1/\e^2 \geq 0$),
    so that we can estimate $\norm{\psi}^2_{L^2(\Gamma)} \leq \E_{\mathrm{reg}}(\bv, \cpm, \psi) + 2|\Omega|/\mathrm \e^2$,
	and used the embedding
	\[
		W^{1,1}(0,\Tmax;L^\infty(\Gamma)) \hookrightarrow C([0,\Tmax]; L^\infty(\Gamma)),
	\]
	\textit{cf.~}\cite[Lem.~7.1]{roubicek}.
	Now we can apply Gronwall's inequality to infer
	\begin{multline*}
		\E_{\mathrm{reg}}(\bv, \cpm, \psi)(t) + \int_0^t \Diss(\bv, \cpm, \psi) \diff s\\
		\leq
			\e^t \left( \E_{\mathrm{reg}}(\bv_0, \cpm_0, \psi_0) 
			+ C \left(  
				\norm{\xi}^2_{W^{1,2}(0,\Tmax;L^2(\Gamma))}
				+ \norm{\xi}_{W^{1,1}(0,\Tmax;L^\infty(\Omega))}
				+ \Tmax
			\right) \right).
	\end{multline*}
    By the continuity of $(\bv, \cpm, \psi)$, \textit{cf.~}Remark~\ref{rem:psi_c_cont} and Lemma~\ref{lem:contraction}
    this inequality holds for all $t \in [0,\Tmax)$,
	which finishes our proof.
\end{proof}
From Proposition~\ref{propo:reg_en_ineq_1} we can derive a second energy estimate.
This is achieved through integration by parts, with the primary technical challenge being the control of second-order derivatives of the electric potential on the boundary. 
We address this challenge using the surface gradient, surface divergence, and integration by parts on the boundary. 
The necessary results on surface differential operators are provided in the Appendix, \textit{cf.~}Section~\ref{sec:bd_func}, for the reader's convenience. 
The integration by parts on the boundary then enables us to plug in the Robin boundary condition for the electric potential, 
thereby reducing the order of derivatives of the electric potential on the boundary to a controllable level. 
A crucial assumption in this step is the tangentiality of the director field $\bd \cdot \bn = 0 $ on $\Gamma$, ensuring that
\[
	\lamMa \nabla \psi \cdot \bn = \varMa \nabla \psi \cdot \bn = \nabla \psi \cdot \bn + \varepsilon (\nabla \psi \cdot \bd) (\bd \cdot \bn) = \nabla \psi \cdot \bn.
\]
This gives us the necessary flexibility for incorporating the Robin boundary condition.
\begin{propo}[Energy inequality II for \eqref{eq:full_reg_system}]\label{propo:reg_en_ineq_2}
Let Assumption~\ref{ass:reg} hold and let  $(\bv, \cpm, \psi)$ be the weak solution to \eqref{eq:full_reg_system} on the maximal time interval $[0,T_{\textrm{max}})$ 
	for some $\kappa > 0$ 
	given by Lemma~\ref{lem:contraction}.
Then for $\kappa > 0 $ small enough 
(that is $ \kappa C ( 1 + \norm{\bd}^2_{W^{2,\infty}(\Omega)}) \leq 1/32$ for some constant $C > 0$), 
there exists a constant $C>0$ such that 
\begin{align}\label{eq:reg_en_ineq_2}
	\int_0^{\Tmax} \int_\Omega 
		\sigmp{|\nabla \sqrt{\cpm}|_{\lamMa}^2 + \cpm |\nabla \psi|^2_{\lamMa}} 
		+ |\nabla^2 \psi |^2 \diff \X \diff t
	\leq
 		C\,.
\end{align}
% holds.
\end{propo}

\begin{proof}
	This follows from extending the first energy inequality from Proposition~\ref{propo:reg_en_ineq_1}.
	Extending the squared sum of the dissipation $\Diss$, the first energy inequality \eqref{eq:reg_en_ineq_1} implies
	\begin{multline}\label{eq:first_expand}
		\int_0^{\Tmax} \int_\Omega
	 		\sigmp{4 |\nabla \sqrt{\cpm}|^2_{\lamMa} 
	 		+ \cpm |\nabla \psi|^2_{\lamMa}}
	 		+ 2 \lamMa \left(\nabla c^+ - \nabla c^- \right) \cdot \nabla \psi
	 	\diff \X \diff t\\
	 	\leq 
	 		\e^{\Tmax} \left( \E_{\mathrm{reg}}(\bv_0, \cpm_0, \psi_0) 
			+ C \left(  
				\norm{\xi}^2_{W^{1,2}(0,T;L^2(\Gamma))}
				+ \norm{\xi}_{W^{1,1}(0,T;L^\infty(\Gamma))}
				+ \Tmax
			\right) \right),
	\end{multline}
	for some $C>0$ independent of $\kappa$.
	Using the elliptic $L^2-$estimate for the Robin Laplacian, \textit{cf.~}Lemma~\ref{lem:robin_semigroup_L2} and Lemma~\ref{lem:prop_reg_op}, 
	we can estimate 
	\[
		\frac{\kappa}{2} \norm{\varphi_0}^2_{L^2(\Omega)} \leq \frac{\kappa}{2} \sigmp{\norm{\cpm_0}^2_{L^2(\Omega)}}
	\]
	and thus $\E_{\mathrm{reg}}(\bv_0, \cpm_0, \psi_0) \leq C$.
	The first two terms on the left-hand side of~\eqref{eq:first_expand} already have a the right sign so we turn directly to the third term 
	and integrate it by parts to see that in the end it also gives a ``good'' term at least in the terms quadratic in the highest order derivatives
	of $\psi$:
	\begin{multline}\label{eq:en_ineq_sq_expan_1}
		\int_\Omega
	 		2 \lamMa \left(\nabla c^+ - \nabla c^- \right) \cdot \nabla \psi
	 	\diff \X\\
	 	=
	 		- \int_\Omega
	 			2 (c^+ -  c^-) \nabla \cdot \left( \lamMa \nabla \psi \right)
	 		\diff \X 
	 		+\int_\Gamma
	 			2 (c^+ -  c^-) ( \xi - \tau \psi) 
	 		\dS,
	\end{multline}
	where we used the Robin boundary condition for the electric potential and the fact that $\bd \cdot \bn = 0 $ on $\Gamma$ and thus
	\[
		\lamMa \nabla \psi \cdot \bn = \varMa\nabla \psi \cdot \bn.
	\]
	Using the regularized Poisson equation \textit{cf.}~\eqref{eq:full_reg_p}, we can rewrite the volume integral of \eqref{eq:en_ineq_sq_expan_1} as
	\begin{multline}\label{eq:en_ineq_sq_expan_2}
		- \int_\Omega
	 		2 (c^+ - c^-) \nabla \cdot \left( \lamMa \nabla \psi \right)
	 	\diff \X
	 	= 
	 		- \int_\Omega
	 			2 \varphi \; \nabla \cdot \left( \lamMa \nabla \psi \right)
	 		\diff \X\\
	 		- \kappa\int_\Omega
	 			2 \varMa \nabla \varphi \cdot \nabla \left( \nabla \cdot \left( \lamMa \nabla \psi \right) \right)
	 		\diff \X 
	 		+ \kappa \int_\Gamma
	 			2 \nabla \cdot \left( \lamMa \nabla \psi \right) \; \varMa \nabla \varphi \cdot \bn
	 		\dS.
	\end{multline}
	The first term on the right-hand side turns out to give us the term with second order derivatives of $\psi$ with a ``good'' sign 
	needed for \eqref{eq:reg_en_ineq_2}.
	To see that, we again use the regularized Poisson equation \eqref{eq:full_reg_p},
	\begin{multline}\label{eq:en_ineq_sq_expan_3}
		- \int_\Omega
	 		2 \varphi \; \nabla \cdot \left( \lamMa \nabla \psi \right)
	 	\diff \X
	 	=
	 		2 \int_\Omega
	 			 \nabla \cdot \left( \varMa \nabla \psi \right) \; \nabla \cdot \left( \lamMa \nabla \psi \right)
	 		\diff \X\\
	 	=
	 		2 \int_\Omega
	 			 (\nabla \cdot \nabla \psi) (\nabla \cdot \nabla \psi)
	 			 + \varepsilon \lambda |\nabla \cdot ((\bd \cdot \nabla \psi) \bd)|^2
	 			 + (\varepsilon + \lambda) (\nabla \cdot \nabla \psi) (\nabla \cdot ((\bd \cdot \nabla \psi) \bd))
	 		\diff \X.
	\end{multline}
	The first and second term already have a good sign but to get a term with a good sign in the full second order derivative of $\psi$ 
	we integrate the first term by parts two times
	and to see that the third term also has a positive sign in the term quadratic in the second order derivative of $\psi$
	we also integrate it by parts two times.
	Starting with the first term on the right-hand side of \eqref{eq:en_ineq_sq_expan_3}, we find
	\begin{align}\label{eq:full_sec_deriv}
    \begin{split}
		\int_\Omega (\nabla \cdot \nabla \psi) (\nabla \cdot \nabla \psi) \diff \X
		&=
			- \int_\Omega  \nabla (\nabla \cdot \nabla \psi) \cdot \nabla \psi \diff \X
			+ \int_\Gamma (\nabla \cdot \nabla \psi) \nabla \psi \cdot \bn \dS\\
		&=
			\int_\Omega  (\nabla^2 \psi)^T : \nabla^2 \psi \diff \X
			-  \int_\Gamma (\nabla^2 \psi)^T \bn \cdot \nabla \psi \dS
			+ \int_\Gamma (\nabla \cdot \nabla \psi) \nabla \psi \cdot \bn \dS\\
		&=
			\int_\Omega  |\nabla^2 \psi|^2 \diff \X
			- \int_\Gamma (\nabla^2 \psi) \bn \cdot \nabla \psi \dS
			+ \int_\Gamma (\nabla \cdot \nabla \psi) (\xi - \tau \psi) \dS.
    \end{split}
	\end{align}
	The first term is exactly the one we aimed for so we move to the boundary terms.
	To be able to take the full gradient of the outer normal field $\bn$ we introduce the trace extension operator $E$ and the trace operator $S$,
    this can be seen in the following calculation. For the sake of readability we will again omit $E$ and $S$ in all other calculations. 
    It should be clear from the context if one needs to consider a function on the boundary or in the bulk.
	Using the characterization of the surface gradient $\nabla_\Gamma$, 
	\textit{cf.~}Theorem~\ref{thm:surf_grad_vs_tan_pro},
	we can rewrite the first boundary integral on the right-hand side of \eqref{eq:full_sec_deriv} as 
	\begin{align}\label{eq:bd_exp_1}
    \begin{split}
		- &\int_\Gamma (\nabla^2 \psi) \bn \cdot \nabla \psi \dS
		=
			- \int_\Gamma S \left( [\nabla(\nabla \psi \cdot E(\bn)) - \nabla E(\bn)^T \nabla \psi] \cdot \nabla \psi \right) \dS\\
		&=
			- \int_\Gamma 
				\left[ 
					\nabla_\Gamma (S(\nabla \psi) \cdot \bn) 
					+  ( S(\nabla (\nabla \psi \cdot E(\bn))) \cdot \bn ) \bn
					- S(\nabla E(\bn)^T \nabla \psi) 
				\right] \cdot S(\nabla \psi) 
			\dS\\
		&=
			 \int_\Gamma 
			 	S(\nabla E(\bn)^T \nabla \psi \cdot \nabla \psi)
				- \nabla_\Gamma (\xi - \tau S(\psi))  \cdot S(\nabla \psi)
				- (S(\nabla (\nabla \psi \cdot E(\bn))) \cdot \bn ) (\xi - \tau S(\psi))
			\dS\\
		&=
			\int_\Gamma 
				S(\nabla E(\bn)^T \nabla \psi \cdot \nabla \psi)
				- \nabla_\Gamma (\xi - \tau S(\psi))  \cdot \nabla_\Gamma S(\psi)
				+  (S(\nabla (\nabla \psi \cdot E(\bn))) \cdot \bn ) (\xi - \tau S(\psi))
			\dS\\
		&=
			\int_\Gamma 
				\tau |\nabla_\Gamma S(\psi)|^2 
				- \nabla_\Gamma \xi \cdot \nabla_\Gamma S(\psi)
				- (S(\nabla (\nabla \psi \cdot E(\bn))) \cdot \bn ) (\xi - \tau S(\psi))
			\dS\\
			&\hspace*{1em}+ \int_\Gamma 
				S(\nabla E(\bn)^T \nabla \psi \cdot \nabla \psi)
			\dS,
    \end{split}
	\end{align}
	where we used the integration by parts rule from Corollary~\ref{cor:ibp_boundary_main} two times to plug in the Robin boundary condition
	and since $\nabla \psi$ is not necessarily tangential on the boundary a curvature term appears,
	but as we integrated by parts two times it vanishes again.
	The first term has a good sign, the second and last term can be estimated due to the assumed regularity of $\xi$ and $\bn$
	and the third term cancels with part of the second boundary integral in \eqref{eq:full_sec_deriv}.
	Which follows from rewriting the second boundary integral on the right-hand side of \eqref{eq:full_sec_deriv}
	with the help of the surface divergence Theorem~\ref{thm:surf_grad_vs_tan_pro},
	\begin{align}\label{eq:bd_exp_1_1}
    \begin{split}
		\int_\Gamma (\nabla \cdot \nabla \psi) (\xi - \tau \psi) \dS
		&=
			\int_\Gamma [\nabla_\Gamma \cdot (\nabla \psi) + \nabla (\nabla \psi) \bn \cdot \bn] (\xi - \tau \psi) \dS\\
		&=
			\int_\Gamma
				- \nabla_\Gamma \psi \cdot \nabla_\Gamma (\xi - \tau \psi)
				+ (\xi - \tau \psi) (\nabla \psi \cdot \bn) \nabla_\Gamma \cdot \bn
			\dS\\
			&\hspace{1em}+ \int_\Gamma
				(\nabla (\nabla \psi \cdot \bn) \cdot \bn) (\xi - \tau \psi)
				- \nabla \bn^T \nabla \psi \cdot \bn (\xi - \tau \psi)
			\dS\\
		&=
			\int_\Gamma
				\tau |\nabla_\Gamma \psi|^2 
				- \nabla_\Gamma \psi \cdot \nabla_\Gamma \xi
				+ (\xi - \tau \psi)^2 \nabla_\Gamma \cdot \bn
			\dS\\
			&\hspace{1em}+ \int_\Gamma
				(\nabla (\nabla \psi \cdot \bn) \cdot \bn) (\xi - \tau \psi)
				- \nabla \bn^T \nabla \psi \cdot \bn (\xi - \tau \psi)
			\dS.
    \end{split}
	\end{align}
	The first term again has a good sign, the second, third and last can be absorbed 
	and the fourth cancels with the third term on the right-hand side of \eqref{eq:bd_exp_1}.
	We now turn to the third term on the right-hand side of \eqref{eq:en_ineq_sq_expan_3} and proceed similarly.
	Integrating this term by parts two times, we obtain
	\begin{align}\label{eq:bd_exp_2}
    \begin{split}
		 \int_\Omega
	 		(\nabla \cdot \nabla \psi) (\nabla \cdot ((\bd &\cdot \nabla \psi) \bd))
	 	\diff \X\\
	 	&=
	 		- \int_\Omega
	 			\nabla \psi \cdot (\nabla \cdot \nabla((\bd \cdot \nabla \psi) \bd)^T)
	 		\diff \X
	 		+ \int_\Gamma
	 			\nabla \psi \cdot \bn \; (\nabla \cdot ((\bd \cdot \nabla \psi) \bd))
	 		\dS\\
	 	&=
	 		\int_\Omega
	 			\nabla^2 \psi : \nabla((\bd \cdot \nabla \psi) \bd)^T
	 		\diff \X
	 		- \int_\Gamma
	 			(\nabla((\bd \cdot \nabla \psi) \bd)^T \bn) \cdot \nabla \psi
	 		\dS\\
	 		&\hspace{1em}+ \int_\Gamma
	 			(\xi - \tau \psi) (\nabla \cdot ((\bd \cdot \nabla \psi) \bd))
	 		\dS.
    \end{split}
	\end{align}
	All terms quadratic in the second derivative of $\psi$ in the matrix-scalar product have a good sign 
	and the boundary integrals can partly be absorbed and partly cancel each other out.
	The following calculation shows that the terms quadratic in the second derivative of $\psi$ in the matrix-scalar product have a good sign.
	Using the symmetry of $\nabla^2 \psi$ we have $\nabla^2 \psi : \bs A^T = \nabla^2 \psi : \bs A $ for any matrix $\bs A \in \R^{3\times 3}$ and
	\[
		\nabla \left( (\bd \cdot \nabla \psi) \bd \right) 
			=
			(\bd \cdot \nabla \psi) \nabla \bd + \bd \otimes  (\nabla \bd^T \nabla \psi) + \bd \otimes (\nabla^2 \psi \bd).
	\]
	Thus we obtain
	\begin{align}\label{eq:matrix_exp_2}
    \begin{split}
		\nabla^2 \psi : \nabla((\bd \cdot \nabla \psi) \bd)^T
		&=
			\nabla^2 \psi : \nabla((\bd \cdot \nabla \psi) \bd)\\
		&=
			\nabla^2 \psi : \left( \bd \otimes (\nabla^2 \psi \bd) + (\bd \cdot \nabla \psi) \nabla \bd + \bd \otimes  (\nabla \bd^T \nabla \psi) \right)\\
		&=
			|\nabla^2 \psi \bd|^2 + \nabla^2 \psi : \left( (\bd \cdot \nabla \psi) \nabla \bd + \bd \otimes  (\nabla \bd^T \nabla \psi) \right).
    \end{split}
	\end{align}
    The first term on the right-hand side of \eqref{eq:matrix_exp_2} has a good sign and the second term can be handled by Young's inequality.
	We collect the lower order terms ($\mathrm{l.o.t.}$) in 
	\begin{align}\label{eq:omega_lot}
		\mathrm{l.o.t.}_\Omega 
		= 
			(\varepsilon + \lambda ) \nabla^2 \psi : \left( (\bd \cdot \nabla \psi) \nabla \bd + \bd \otimes  (\nabla \bd^T \nabla \psi) \right).
	\end{align}
	We now turn to the boundary integrals on the right-hand side of \eqref{eq:bd_exp_2}.
	The first boundary term on the right-hand side of \eqref{eq:bd_exp_2} can be rewritten as
	\begin{align}\label{eq:bd_exp_3}
    \begin{split}
		- &\int_\Gamma
	 		(\nabla((\bd \cdot \nabla \psi) \bd)^T \bn) \cdot \nabla \psi
	 	\dS
	 	=
	 		- \int_\Gamma
	 			\left[ \nabla((\bd \cdot \nabla \psi) \bd \cdot \bn) - (\bd \cdot \nabla \psi) \nabla \bn^T \bd \right] \cdot \nabla \psi
	 		\dS\\
	 	&=
	 		- \int_\Gamma
	 			\nabla_\Gamma((\bd \cdot \nabla \psi) \bd \cdot \bn) \cdot \nabla \psi
	 			+ (\nabla((\bd \cdot \nabla \psi) \bd \cdot \bn) \cdot \bn) \bn \cdot \nabla \psi
	 			- (\bd \cdot \nabla \psi) \nabla \bn^T \bd  \cdot \nabla \psi
	 		\dS\\
	 	&=
	 		- \int_\Gamma
	 			(\nabla((\bd \cdot \nabla \psi) \bd \cdot \bn) \cdot \bn) (\xi - \tau \psi)
	 			- (\bd \cdot \nabla \psi) \nabla \bn^T\bd  \cdot \nabla \psi
	 		\dS,
    \end{split}
	\end{align}
	where we used the fact that $\bd \cdot \bn = 0$ on $\Gamma$.
	For the second boundary term on the right-hand side of \eqref{eq:bd_exp_2} we note
	\begin{align}\label{eq:bd_exp_4}
    \begin{split}
		\int_\Gamma
	 		(\xi - \tau \psi) (\nabla &\cdot ((\bd \cdot \nabla \psi) \bd))
	 	\dS
	 	=
	 		\int_\Gamma
	 			(\xi - \tau \psi) \left[ \nabla_\Gamma \cdot ((\bd \cdot \nabla \psi) \bd) + \nabla((\bd \cdot \nabla \psi) \bd) \bn \cdot \bn \right]
	 		\dS\\
	 	&=
	 		\int_\Gamma
	 			- \nabla_\Gamma (\xi - \tau \psi) ((\bd \cdot \nabla \psi) \bd) 
	 			+ (\xi - \tau \psi) \nabla((\bd \cdot \nabla \psi) \bd \cdot \bn) \cdot \bn
            \dS\\
            &\hspace{2em}-\int_\Gamma
	 			(\xi - \tau \psi) (\bd \cdot \nabla \psi) \nabla \bn^T \bd \cdot \bn
	 		\dS\\
	 	&=
	 		\int_\Gamma
	 			\tau |\nabla_\Gamma \psi \cdot \bd|^2
	 			- (\nabla_\Gamma \xi \cdot \bd) (\nabla_\Gamma \psi \cdot \bd) 
            \dS\\
            &\hspace{1em}+\int_\Gamma
	 			(\xi - \tau \psi) \nabla((\bd \cdot \nabla \psi) \bd \cdot \bn) \cdot \bn
	 			- (\xi - \tau \psi) (\bd \cdot \nabla \psi) \nabla \bn^T \bd \cdot \bn
	 		\dS,
    \end{split}
	\end{align}
	where we again used the integration by parts rule on the boundary from Corollary~\ref{cor:ibp_boundary_main},
	since the vector field $\bd$ is indeed tangential there is no curvature term.
	The first term has a good sign, the second and last term can be absorbed and the third term cancels with 
	the first term on the right-hand side of \eqref{eq:bd_exp_3}. 
	We collect the lower order boundary terms of \eqref{eq:bd_exp_1}, \eqref{eq:bd_exp_1_1}, \eqref{eq:bd_exp_3} and \eqref{eq:bd_exp_4} in
	\begin{align}\label{eq:gamma_lot}
	\begin{split}
		\mathrm{l.o.t.}_{\Gamma} &:= 
			- 2 \nabla_\Gamma \xi \cdot \nabla_\Gamma \psi
			+ \nabla \bn^T \nabla \psi \cdot \nabla \psi
			+ (\xi - \tau \psi)^2 \nabla_\Gamma \cdot \bn
			- \nabla \bn^T \nabla \psi \cdot \bn (\xi - \tau \psi)\\
			&\hspace*{1em}+ (\varepsilon + \lambda)\left( 
				(\bd \cdot \nabla \psi) \nabla \bn^T\bd  \cdot \nabla \psi
				- (\nabla_\Gamma \xi \cdot \bd) (\nabla_\Gamma \psi \cdot \bd)
				- (\xi - \tau \psi) (\bd \cdot \nabla \psi) \nabla \bn^T \bd \cdot \bn
			\right).
	\end{split}
	\end{align}
	Now, we turn to the volume integral of \eqref{eq:en_ineq_sq_expan_2} with prefactor $\kappa$ to see that this gives a term with good sign for the 
	third order derivatives of $\psi$. 
	\begin{align}\label{eq:kappa_exp}
	\begin{split}
		- \int_\Omega
	 		\varMa \nabla \varphi \cdot \nabla \left( \nabla \cdot \left( \lamMa \nabla \psi \right) \right)
	 	\diff \X
	 	&=
	 		\int_\Omega
	 			\nabla (\nabla \cdot (\varMa \nabla \psi))
	 			\cdot \varMa \nabla \left( \nabla \cdot \left( \lamMa \nabla \psi \right) \right)
	 		\diff \X\\
	 	&=
	 		\int_\Omega
	 			|\nabla (\Delta \psi)|^2_{\varMa}
	 			+ \varepsilon \lambda |\nabla (\nabla \cdot ((\bd \cdot \nabla \psi) \bd))|^2_{\varMa}
			\diff \X\\
			&\hspace{1em}+ \int_\Omega
	 			(\varepsilon + \lambda) \nabla (\Delta \psi) \cdot \varMa \nabla (\nabla \cdot ((\bd \cdot \nabla \psi) \bd))
	 		\diff \X.
	\end{split}
	\end{align}
	To see that the third term also gives a good sign for the critical order term 
	we proceed as above, integrating by parts two times and estimating the occurring boundary term.
	This calculation is quite lengthy and uses the same technique as above thus we perform it in the Appendix
	for thoroughness and for the interested reader.
	By Lemma~\ref{lem:reg_en_terms} from the Appendix we have
	\begin{align*}
		(\varepsilon + \lambda) \int_\Omega
	 		\nabla (\Delta \psi) \cdot \varMa \nabla (\nabla \cdot ((\bd \cdot \nabla \psi) \bd))
	 	\diff \X
	 	=
	 		(\varepsilon + \lambda) \int_\Omega
	 			|\nabla^3 \psi \cdot \bd |^2 + \varepsilon |(\nabla^3 \psi \cdot \bd) \bd|^2
	 		\diff \X
	 		+ \mathrm{l.o.t}_{\kappa}.
	\end{align*}
	Finally, we consider the boundary integral of \eqref{eq:en_ineq_sq_expan_2},
	\begin{align}\label{eq:en_ineq_sq_exan_kappa}
    \begin{split}
		 \kappa \int_\Gamma
	 		2 \nabla \cdot &\left( \lamMa \nabla \psi \right) \; \varMa \nabla \varphi \cdot \bn
	 	\dS\\
	 	&= 
	 		- 2 \kappa \tau \int_\Gamma
	 			\nabla \cdot \left( \lamMa \nabla \psi \right)  \varphi
	 		\dS
	 	= 
	 		2 \kappa \tau \int_\Gamma
	 			\nabla \cdot \left( \lamMa \nabla \psi \right)  \nabla \cdot (\varMa \nabla \psi)
	 		\dS\\
	 	&=
	 		2 \kappa \tau \int_\Gamma
	 			|\Delta \psi|^2 + \varepsilon \lambda |\nabla \cdot ((\bd \cdot \nabla \psi) \bd )|^2
	 			+ (\varepsilon + \lambda) \Delta \psi \cdot \left( \nabla \cdot ((\bd \cdot \nabla \psi) \bd) \right)
	 		\dS.
    \end{split}
	\end{align}
	Collecting the transformations from \eqref{eq:en_ineq_sq_expan_2}--\eqref{eq:en_ineq_sq_exan_kappa} 
	we can rewrite \eqref{eq:en_ineq_sq_expan_1}
	\begin{align*}
    \begin{split}
		\int_\Omega
	 		2 \lamMa &\left(\nabla c^+ - \nabla c^- \right) \cdot \nabla \psi
	 	\diff \X
	 	=
	 		2 \int_\Omega
	 			|\nabla^2 \psi|^2 
	 			+ \varepsilon \lambda  |\nabla \cdot ((\nabla \psi \cdot \bd) \bd)|^2 
	 			+ (\varepsilon + \lambda ) |\nabla^2 \psi \bd|^2
	 		\diff \X\\
	 		&+ 2 \kappa \int_\Omega
	 			|\nabla (\Delta \psi)|_{\varMa}^2 
	 			+ \varepsilon \lambda |\nabla( \nabla \cdot ((\nabla \psi \cdot \bd) \bd))|_{\varMa}^2 
            \diff \X\\
            &+ 2 (\varepsilon + \lambda)\kappa \int_\Omega
	 			|\nabla^3 \psi \cdot \bd|^2
	 			+  \varepsilon |(\nabla^3 \psi \cdot \bd) \bd|^2
	 		\diff \X\\
	 		&+ 2 \tau \int_\Gamma 
	 			2 |\nabla_\Gamma \psi|^2 
	 			+ (\varepsilon + \lambda)   |\nabla_\Gamma \psi \cdot \bd|^2
	 		\dS
	 		+ 2 \kappa \tau \int_\Gamma 
	 			|\Delta \psi|^2 + \varepsilon \lambda |\nabla \cdot ((\bd \cdot \nabla \psi) \bd)|^2
	 		\dS\\
	 		&+ 2 \int_\Gamma
	 			(c^+ -  c^-) ( \xi - \tau \psi) 
	 		\dS
	 		+ 2 \kappa \tau \int_\Gamma
	 			(\varepsilon + \lambda) \Delta \psi \; \nabla \cdot ((\bd \cdot \nabla \psi) \bd)
	 		\dS\\
	 		&+ 2 \left( \int_\Omega \mathrm{l.o.t.}_\Omega \diff \X + \int_\Gamma \mathrm{l.o.t.}_\Gamma \dS
	 		+ \kappa \mathrm{l.o.t.}_\kappa \right).
    \end{split}
	\end{align*}
	Inserting this back into \eqref{eq:first_expand} we obtain 
	\begin{align}\label{eq:reg_en_ineq_2_collection}
    \begin{split}
		&\hspace{-3em}\int_{0}^{\Tmax} \int_\Omega
	 		\sigmp{4 |\nabla \sqrt{\cpm}|^2_{\lamMa} 
	 		+ \cpm |\nabla \psi|^2_{\lamMa}}
	 		+ 2 |\nabla^2 \psi|^2 
	 		+ 2 \varepsilon \lambda  |\nabla \cdot ((\nabla \psi \cdot \bd) \bd)|^2 
	 	\diff \X \diff t \\
        &+ 2 \int_0^{\Tmax} \int_\Omega 
            (\varepsilon + \lambda ) |\nabla^2 \psi \bd|^2
            + \kappa |\nabla (\Delta \psi)|_{\varMa}^2 
	 		+ \kappa \varepsilon \lambda |\nabla( \nabla \cdot ((\nabla \psi \cdot \bd) \bd))|_{\varMa}^2 
        \diff \X\\
	 	&+ 2 \kappa (\varepsilon + \lambda) \int_0^{\Tmax} \int_\Omega
	 		|\nabla^3 \psi \cdot \bd|^2
	 		+  \varepsilon |(\nabla^3 \psi \cdot \bd) \bd|^2
	 	\diff \X \diff t\\
	 	&+ 2 \tau \int_0^{\Tmax} \int_\Gamma2
	 		|\nabla_\Gamma \psi|^2 
	 		+ (\varepsilon + \lambda)   |\nabla_\Gamma \psi \cdot \bd|^2
	 	\dS \diff t\\
	 	&+ 2 \kappa  \tau \int_0^{\Tmax} \int_\Gamma 
	 		|\Delta \psi|^2 + \varepsilon \lambda |\nabla \cdot ((\bd \cdot \nabla \psi) \bd)|^2
	 	\dS \diff t\\
	 	&\hspace{-1em}\leq 
	 		C \e^{\Tmax} \left(
				\norm{\xi}^2_{W^{1,2}(0,\Tmax;L^2(\Gamma))}
				+ \norm{\xi}_{W^{1,1}(0,\Tmax;L^\infty(\Gamma))}
				+ \Tmax
				+ 1
			\right)\\
	 		&\hspace{1em}\underbrace{- 2 \int_0^{\Tmax} \int_\Gamma
	 			(c^+ -  c^-) ( \xi - \tau \psi) 
	 		\dS \diff t}_{=:\hypertarget{firstest}{\mathrm{\,I \, }}}
	 		\underbrace{- 2 \kappa \tau \int_0^{\Tmax} \int_\Gamma
	 			(\varepsilon + \lambda) \Delta \psi \; \nabla \cdot ((\bd \cdot \nabla \psi) \bd)
	 		\dS \diff t}_{=: \hypertarget{fsecondest}{\mathrm{\,II}}}\\
	 		&\hspace{1em}- \underbrace{ 
	 			2 \int_0^{\Tmax} \int_\Omega \mathrm{l.o.t.}_\Omega \diff \X 
	 			+ \int_\Gamma \mathrm{l.o.t.}_\Gamma \dS + \kappa \mathrm{l.o.t.}_\kappa \diff t}_{=: \hypertarget{thirdest}{\mathrm{\,III}}}.
    \end{split}
	\end{align}
	Now, we can start to estimate the right-hand side.
	For the boundary integral in $\hyperlink{firstest}{\mathrm{I}}$ we use the trace estimate \cite[Prop.~8.2]{diBenedetto},
	which gives us that for all $p \in [1, 3)$ there exists $C>0$ such that for all $\delta > 0$ and all $u \in W^{1,p}(\Omega)$ 
	\begin{align}\label{eq:diBene_trace_est}
		\norm{u}_{L^q(\Gamma)} \leq \delta \norm{\nabla u}_{L^p(\Omega)} + C \left( 1 + \frac{1}{\delta} \right) \norm{u}_{L^p(\Omega)}
	\end{align}
	holds, for $q \in \left[1,2p/(3 - p)\right]$.
	Choosing
	\(
		\delta 
		= (C ( \norm{\xi}_{L^\infty(0,\Tmax;L^\infty(\Gamma))} + \norm{\psi}_{L^\infty(0,\Tmax;W^{1,2}(\Omega))} )  + 1 )^{-1/2} 
	\)
	we obtain
	\begin{align*}
		|\hyperlink{firstest}{\mathrm{I}}|
	 	&\leq
		 	2 \int_0^{\Tmax} 
		 		\sigmp{\norm{\cpm}_{L^{\frac{4}{3}}(\Gamma)}} \norm{\xi - \tau \psi}_{L^4(\Gamma)}
		 	\diff t\\
        &\leq
		 	C \left( \norm{\xi}_{L^\infty(0,\Tmax ;L^\infty(\Gamma))} + \norm{\psi}_{L^\infty(0,\Tmax;W^{1,2}(\Omega))} \right)
		 	\int_0^{\Tmax} 
		 		\sigmp{\snorm{\sqrt{\cpm \,}}^2_{L^{\frac{8}{3}}(\Gamma)}}
		 	\diff t\\
		 \underset{\eqref{eq:diBene_trace_est}}&{\leq} 
		 	C \left( \norm{\xi}_{L^\infty(0,\Tmax;L^\infty(\Gamma))} + \norm{\psi}_{L^\infty(0,\Tmax;W^{1,2}(\Omega))} \right)\\
		 	&\hspace{2em}\times \int_0^{\Tmax} 
		 		\sigmp{\delta^2 \snorm{\nabla \sqrt{\cpm}}^2_{L^2(\Omega)} 
		 		+ C \left( 1 + \frac{1}{\delta^2} \right) 
		 		\snorm{\sqrt{\cpm}}^2_{L^2(\Omega)}}
		 	\diff t\\
		&\leq
		 	\int_0^{\Tmax} 
		 		\sigmp{\snorm{\nabla \sqrt{\cpm}}^2_{L^2(\Omega)} 
		 		+ C \left( 
		 			1 +\norm{\xi}^2_{L^\infty(0,\Tmax;L^\infty(\Gamma)))} 
		 			+ \norm{\psi}^2_{L^\infty(0,\Tmax;W^{1,2}(\Omega))} 
		 		\right) 
		 		\snorm{\sqrt{\cpm}}^2_{L^2(\Omega)}}
		 	\diff t\\
		 &\leq 
		 	\sigmp{\snorm{\nabla \sqrt{\cpm}}^2_{L^2(0,\Tmax;L^2(\Omega))}}\\
		 	&\hspace{2em}+ C \left( 
		 			1 +\norm{\xi}^2_{L^\infty(0,\Tmax;L^\infty(\Gamma)))} 
		 			+ \norm{\psi}^2_{L^\infty(0,\Tmax;W^{1,2}(\Omega))} 
		 		\right) 
		 		\sigmp{\snorm{\sqrt{\cpm}}^2_{L^2(0,\Tmax;L^2(\Omega))}}.
	\end{align*}
	Next, we turn to the boundary integral in $\hyperlink{secondest}{\mathrm{II}}$, where we again use the trace estimate
	\cite[Prop.~8.2]{diBenedetto}, \textit{cf.}~\eqref{eq:diBene_trace_est} with 
	$u = \Delta \psi$ and $\delta = \left( \frac{\varepsilon \lambda}{(\varepsilon + \lambda)^2 \tau} \right)^{1/2}$. 
	We find
	\begin{align*}
		|\hyperlink{secondest}{\mathrm{II}}|
	 	\underset{\mathrm{Young}}&{\leq}
	 		\int_0^{\Tmax} 
	 			\kappa \frac{(\varepsilon + \lambda)^2 \tau}{2 \varepsilon \lambda}\norm{\Delta \psi}^2_{L^2(\Gamma)}
	 			+ 2 \kappa \tau \varepsilon \lambda \norm{\nabla \cdot ((\bd \cdot \nabla \psi) \bd)}^2_{L^2(\Gamma)}
	 		\diff t\\
		&\leq
			\int_0^{\Tmax} 
	 			\frac{\kappa}{2} \norm{\nabla (\Delta \psi)}^2_{L^2(\Omega)}
	 			+ \kappa  C(\varepsilon, \lambda, \tau) \norm{\Delta \psi}^2_{L^2(\Omega)}
	 			+ 2 \kappa \tau \varepsilon \lambda \norm{\nabla \cdot ((\bd \cdot \nabla \psi) \bd)}^2_{L^2(\Gamma)}
	 		\diff t.
	\end{align*}
	The first term can be absorbed into the term 
    $\kappa \norm{\nabla (\Delta \psi)}^2_{L^2(\Omega)}$ 
	on the left-hand side of \eqref{eq:reg_en_ineq_2_collection},
	the second term can be absorbed into the good term of $\norm{\nabla^2 \psi}^2_{L^2(\Omega)}$ on the left-hand side of \eqref{eq:reg_en_ineq_2_collection} 
	(without prefactor $\kappa$) using the smallness assumption of $\kappa$
	and the last term can also be absorbed into the according term on the left-hand side of \eqref{eq:reg_en_ineq_2_collection}.
	Using Lemma~\ref{lem:all_lower_ord_terms} from the Appendix, we find
	\begin{multline*}
		|\hyperlink{thirdest}{\mathrm{III}}| 
		\leq 
			C
			+ \frac{\tau}{2}  \norm{\nabla_\Gamma \psi}^2_{L^2(0,\Tmax;L^2(\Gamma))}
			+ \tau \frac{\varepsilon + \lambda}{2} \norm{\nabla_\Gamma \psi \cdot \bd}^2_{L^2(0,\Tmax;L^2(\Gamma))}\\
			+ \frac{7}{8} \norm{\nabla^2 \psi}^2_{L^2(0,\Tmax;L^2(\Omega))}
			+ \frac{\kappa}{2} \norm{\nabla (\Delta \psi)}^2_{L^2(0,\Tmax;L^2(\Omega))},
	\end{multline*}
	Putting the estimates for $\hyperlink{firstest}{\mathrm{I}}-\hyperlink{thirdhest}{\mathrm{III}}$ together, 
	we can rewrite \eqref{eq:reg_en_ineq_2_collection} as
	\begin{align*}
		\int_0^{\Tmax} \int_\Omega
	 		\sigmp{|\nabla \sqrt{\cpm}|^2_{\lamMa} 
	 		+ \cpm |\nabla \psi|^2_{\lamMa}}
	 		+ |\nabla^2 \psi|^2
	 	\diff \X \diff t
	 	\leq
	 		C.
	\end{align*}
	This finishes our proof.
\end{proof}
\end{subsection}

\begin{subsection}{Limit passage}

\begin{lem}\label{lem:global_sol}
	Under the Assumption~\ref{ass:reg}, $\kappa > 0$ as in Lemma~\ref{lem:contraction}, and all $T \in (0,\infty)$ there exists a unique weak solution to the regularized system \eqref{eq:full_reg_system}.
\end{lem}
The proof is conducted analogously to the proof of \cite[Lem.~4.4]{fischer_saal_2017}.
\begin{proof}
	Let $\bar{T}>0$ be such that we have a weak solution $\bv \in C([0,\bar{T}]; D(A_2^{1/2}))$ to \eqref{eq:full_reg_ns} 
	given by the fixed point of Lemma~\ref{lem:contraction}.
	By \cite[Prop.~7.1.8]{lunardi} there exists a maximal interval of existence $[0, T_{\max})$ for some $T_{\max} > \bar{T}$, such that
	there is a solution $\bv$ on $[0,T]$ for all $T< \Tmax$ and for $T_{\max} < \infty$ we have
	\[
		\lim_{t \nearrow T_{\max}} \norm{\bv(t)}_{D(A_2^{1/2})}  \nearrow \infty.
	\]
	By the two energy inequalities from Propositions~\ref{propo:reg_en_ineq_1} and~\ref{propo:reg_en_ineq_2}, 
	we can infer the boundedness $\norm{\bv(t)}_{D(A_2^{1/2})} \leq C(\Tmax)$ with $C(\Tmax) < \infty$ for $\Tmax < \infty$ and all $t \in [0,T_{\max})$ by testing the Navier--Stokes equation \eqref{eq:full_reg_ns} with $A_2(\bv)$ and thus $T_{\max} = \infty$ follows.
	We now derive this bound of $\norm{\bv(t)}_{D(A_2^{1/2})}$.
	By the energy inequalities \eqref{eq:reg_en_ineq_1} and \eqref{eq:reg_en_ineq_2} there exists a constant $C>0$
	depending on initial values, the external field $\xi$ and the regularization coefficient $\kappa$, such that
	\begin{enumerate}
		\item{ 
			$\bv$ is bounded in $L^\infty(0,T_{\max}; L^2(\Omega)) \cap L^2(0,T_{\max}; D(A_2^{1/2}))$ by $C \e^{C \Tmax}$.
		}
		\item{
			$\sqrt{\cpm}$ is bounded in $L^\infty(0,T_{\max}; L^2(\Omega)) \cap L^2(0,T_{\max}; W^{1,2}(\Omega))$ by $C \e^{C \Tmax}$.
			This implies that $\cpm$ is bounded in $L^2(0,T_{\max}; L^{3/2}(\Omega))$ by $C \e^{C \Tmax}$.
		}
		\item{
			By elliptic $L^1-$regularity we have that $\psi$ is bounded in $L^\infty(0,T_{\max}; W^{3,q}(\Omega))$ for all $q \in [1, 3/2)$
			and thus $\nabla \psi$ is bounded in $L^\infty(0,T_{\max}; L^q(\Omega))$ for all $q \in [1,\infty)$ by $C \e^{C \Tmax}$.
		}
	\end{enumerate}
	Using these bounds and the $L^p-$realization of the fractional power Stokes operator, \textit{cf.~}Definition~\ref{def:stokes}, 
	Lemma~\ref{Stokes}, and Lemma~\ref{lem:prop_reg_op}, we obtain
	\begin{align*}
		\norm{R^{1/2}_\kappa(P((c^+ - c^-) \nabla \psi ))}_{L^2(\Omega \times (0,T))}
		&\leq 
			C \norm{R^{1/2}_\kappa(P((c^+ - c^-) \nabla \psi ))}_{L^2(0,T;W^{1,6/5}(\Omega))}\\
		&\leq 
			C(\kappa) \norm{(c^+ - c^-) \nabla \psi}_{L^2(0,T;L^{6/5}(\Omega))}\\
		&\leq 
			C(\kappa) \sigmp{\norm{\cpm}_{L^2(0,T;L^{3/2}(\Omega))}} \norm{\nabla \psi}_{L^\infty(0,T;L^6(\Omega))}
		\leq 
			C(\kappa) \e^{C \Tmax}
	\end{align*}
	and
	\(
		\norm{\nabla \bv R_\kappa(\bv)}_{L^2(0,T;L^2(\Omega))} 
		\leq 
			C \norm{\nabla \bv}_{L^2(0,T;L^2(\Omega))} \norm{R_\kappa(\bv)}_{L^\infty(\Omega \times (0,T))}
		\leq 
			C(\kappa) \e^{C \Tmax}.
	\)
	Now, we test the regularized Navier--Stokes equation \eqref{eq:full_reg_ns} with $A_2(\bv)$.
	This is indeed an admissible test function by the maximal $L^p-$regularity of the Stokes operator, 
	\textit{cf.~}\cite[Thm.~4.2]{solonnikov}.
	Thus we find, using Young's inequality, that
	\begin{align*}
		\frac{1}{2} \frac{\diff}{\diff t} \norm{\nabla \bv}^2_{L^2(\Omega)} 
		&\leq 
			C \left( 
				\norm{\nabla \bv R_\kappa(\bv)}^2_{L^2(\Omega)} + \norm{R^{1/2}_\kappa(P((c^+ - c^-) \nabla \psi ))}^2_{L^2(\Omega)}
			\right)\\
		&\leq
			C \left(
				\norm{\nabla \bv}^2_{L^2(\Omega)} \norm{R_\kappa(\bv)}^2_{L^\infty(\Omega)}
				+ \norm{R^{1/2}_\kappa(P((c^+ - c^-) \nabla \psi ))}^2_{W^{1,6/5}(\Omega)}
			\right)\\
		&\leq
			C(\kappa)\left(
				\norm{\nabla \bv}^2_{L^2(\Omega)} \norm{\bv}^2_{L^2(\Omega)}
				+ \norm{(c^+ - c^-) \nabla \psi}^2_{L^{6/5}(\Omega)}
			\right)\\
		&\leq
			C(\kappa) \left(
				\norm{\nabla \bv}^2_{L^2(\Omega)} \norm{\bv}^2_{L^2(\Omega)}
				+ \sigmp{\norm{\cpm}^2_{L^{3/2}(\Omega)}} \norm{\nabla \psi}^2_{L^6(\Omega)}
			\right).
	\end{align*}
	Integrating over $(0,t)$, we find
	\begin{align*}
		\norm{\nabla \bv(t)}^2_{L^2(\Omega)}
		&\leq 
			C(\kappa) \norm{\nabla \bv}^2_{L^2(0,T;L^2(\Omega))} \norm{\bv}^2_{L^\infty(0,T;L^2(\Omega))}\\
			&\hspace{1em}+ C(\kappa) \sigmp{\norm{\cpm}^2_{L^2(0,T;L^{3/2}(\Omega))}} \norm{\nabla \psi}^2_{L^\infty(0,T;L^6(\Omega))}	
			+ \norm{\nabla \bv_0}^2_{L^2(\Omega)}
		\leq 
			C(\kappa) \e^{C \Tmax} < \infty,
	\end{align*}
	for all $T_{\max} < \infty$ and thus the global existence of a regularized solution to \eqref{eq:full_reg_system} follows.
\end{proof}

\begin{propo}\label{propo:convergences}
	Let Assumption~\ref{ass:1} hold and let $\{ \bd_\kappa \} \subseteq W^{4,\infty}(\Omega)$ and $\{ \kappa \}$ be sequences with $\kappa > 0 $ and $\kappa \searrow 0$ such that $\bd_\kappa \cdot \bn = 0$ on $\Gamma$, $\bd_{\kappa} \to \bd$ pointwise almost everywhere and 
	strongly in $L^\infty(\Omega)$
	and $\nabla \bd_\kappa \rightharpoonup^* \nabla \bd $ in $L^\infty(\Omega)$ and 
	\begin{equation}\label{eq:d_W2inf_bound}
		C \kappa \left(1 + \norm{\bd_\kappa}^2_{W^{2,\infty}(\Omega)} \right) \leq 1/32.
	\end{equation}
    Additionally, take $\{ \cpm_{0 \kappa} \} \subseteq L^2(\Omega)_+ \cap W^{1,2}(\Omega)$ 
	and $\{ \bv_{0 \kappa} \} \subseteq D(A_2)$ to be such that
	\begin{align*}
	 \cpm_{0 \kappa} \to \cpm_0 \quad \text{ in } L^2(\Omega)_+ \quad \text{ and } \quad \bv_{0 \kappa} \to \bv_0 \quad \text{ in } L^2_\sigma(\Omega),
	\end{align*}
	and $\{ \xi_\kappa \} \subseteq \, C^\infty([0,T]) \otimes C^\infty(\Gamma)$ such that $\, \xi_\kappa \to \xi $ in $W^{1,2}(0,T;W^{2,2}(\Gamma))$.
	Then there exists a subsequence of global weak solutions to \eqref{eq:full_reg_system}, from Lemma~\ref{lem:global_sol}, 
	which we call $(\bv_\kappa, \cpm_\kappa, \psi_\kappa)$ such that
	\begin{subequations}\label{eq:convergences}
	\begin{align}
		\psi_{0 \kappa} &\to \psi_0 &&\text{ in } W^{1,2}(\Omega),\label{eq:strong_conv_psi_0}\\
		\bv_\kappa &\rightharpoonup \bv &&\text{ in } L^2(0,T;W^{1,2}_\sigma(\Omega)), \label{eq:weak_conv_v}\\
		\bv_\kappa &\to \bv &&\text{ in } C_w([0,T];L^2_\sigma (\Omega)),\label{eq:weak_cont_conv_v}\\
		\bv_\kappa &\to \bv &&\text{ in } L^2(0,T;L^2_\sigma(\Omega)),\label{eq:strong_conv_v}\\
		%\bv_\kappa(t) &\to \bv(t) &&\text{ in } L^2_\sigma (\Omega) \text{ for almost all } t \in (0,T),\label{eq:strong_conv_v_t}\\
		\sqrt{\cpm_\kappa} &\rightharpoonup \sqrt{\cpm} &&\text{ in } L^2(0,T;W^{1,2}(\Omega)),\label{eq:weak_conv_sqrt_c}\\
		\cpm_\kappa &\to \cpm &&\text{ in } L^{5/3}(\Omega \times (0,T)) \cap L^1(0,T;L^p(\Omega)) \text{ for all } p \in [1,3), \label{eq:strong_conv_c}\\
		\nabla \cpm_\kappa &\rightharpoonup \nabla \cpm &&\text{ in } L^{5/4}(\Omega\times (0,T)),\label{eq:weak_conv_nabla_c}\\
		% \cpm_\kappa(t) &\to \cpm(t) &&\text{ in } L^p(\Omega) \text{ for almost all } t \in (0,T) 
		% \text{ and all } p \in [1, 3),\label{eq:strong_conv_c_t}\\
        \cpm_\kappa &\rightharpoonup \cpm &&\text{ in } W^{1,10/9}(0,T;W^{-1,10/9}(\Omega)), \label{eq:weak_star_cont_conv_c}\\
        \cpm_\kappa(t) &\rightharpoonup \cpm(t) &&\text{ in } L^1(\Omega) \text{ for \underline{all} }  t \in [0,T], \label{eq:weak_conv_c_ALL_t}\\
		%\cpm_\kappa(t) &\to \cpm(t) \text{ pointwise almost everywhere in } \Omega \text{ for almost all } t \in (0,T), \label{eq:ptw_conv_c}\\
		% \psi_\kappa &\rightharpoonup \psi_\kappa \text{ in } L^q(0,T;W^{1,2}(\Omega)) \text{ for all } q \in [1, \infty), \label{eq:weak_conv_psi}\\
         \psi_\kappa(t) &\rightharpoonup \psi(t) &&\text{ in } W^{1,2}(\Omega) \text{ for \underline{all} }  t \in [0,T],
        \label{eq:weak_conv_psi_ALL_t}\\
		% \nabla \psi_\kappa &\rightharpoonup^* \nabla \psi &&\text{ in } L^\infty(0,T;L^2(\Omega)), \label{eq:weak_star_conv_psi}\\
		\psi_\kappa &\rightharpoonup \psi &&\text{ in } L^2(0,T;W^{2,2}(\Omega)), \label{eq:weak_conv_psi_W22}\\
		\sqrt{\cpm_\kappa} \nabla \psi_\kappa 
		&\rightharpoonup \sqrt{\cpm} \nabla \psi &&\text{ in } L^2(0,T;L^2(\Omega)),\label{eq:weak_conv_c_nabla_psi}
  % \\
		% \psi_\kappa &\to \psi \text{ in } L^1(0,T;W^{1,2}(\Omega)), \label{eq:strong_conv_psi}\\
		% \psi_\kappa(t) &\to \psi(t) \text{ in } W^{1,2}(\Omega) \text{ for almost all } t \in (0,T), \label{eq:strong_conv_psi_t}
	\end{align}
	\end{subequations}
	for $\kappa \searrow 0$, 
	whose limit fulfills the weak formulation \eqref{eq:weak_form_v}--\eqref{eq:weak_form_psi} for all $t \in [0,T]$.
\end{propo}
% The proof can again be conducted analogous to the proof of \cite[Thm.~2.1]{fischer_saal_2017}.
\begin{remark}[Existence of approximate sequences for the director field and the external electric potential]
    The existence of a sequence $\{ \bd_\kappa \}$ fulfilling the properties assumed in Proposition~\ref{propo:convergences} can be proven via
a standard mollification, \cite[Sec.~C.4]{evans}. 
In \cite[Sec.~5.3.3]{evans} the density of $C^\infty(\overline{\Omega})$ in $W^{1,p}(\Omega)$ for all $p \in [1, \infty)$ is proven.
A straight forward calculation proves that the approximating sequence given in the proof there also fulfills the convergence $\nabla \bd_\kappa \rightharpoonup^* \nabla \bd$ in $L^\infty(\Omega)$. The standard mollification can be altered, to ensure that all elements of the approximating sequence already 
fulfill the boundary condition, that is $\bd_\kappa \cdot \bn = 0 $ on $\Gamma$ for all $\kappa$. 
The existence of a sequence $\{ \xi_\kappa \}$ with the properties from the lemma follows from the density
of $C^\infty(\Gamma)$ in $W^{2,2}(\Gamma)$, \textit{cf.~}\cite[Sec.~4.3]{lions_magenes_1961}, 
then \cite[Thm.~5.12]{bochner} gives the density of $C^\infty([0,T]) \otimes C^\infty(\Gamma)$ in $W^{1,2}(0,T;W^{2,2}(\Gamma))$.
\end{remark}

\begin{proof}
Let $\tilde{\varphi}_{0 \kappa}$ be the solution to
\begin{subequations}
\begin{align}
	& &\tilde{\varphi} - \kappa \nabla \cdot (\varMa \nabla \tilde{\varphi}) &= c^+_{0 \kappa} - c^-_{0 \kappa} &  &\text{ in } \Omega,
	\label{eq:approx_approx}\\
	& &\varMa \nabla \tilde{\varphi} \cdot \bn + \tau \tilde{\varphi} &= 0 &  &\text{ on } \Gamma.
\end{align}
\end{subequations}
Then by Lemma~\ref{lem:prop_reg_op}, point~\ref{item:R_kappa_strong_cont},  we have $\tilde{\varphi}_{0 \kappa} \to c^+_0 - c^-_0$ as $\kappa \to 0$.
Note, that here we took $A$ from Lemma~\ref{lem:prop_reg_op} to be $ \nabla \cdot (\varMa \nabla \cdot)$, 
and thus independent of $\bd_\kappa$. Thus the lemma is indeed applicable.
To find the convergence of $\varphi_{0 \kappa}$, the solution to
\begin{subequations}\label{eq:approx_sys}
\begin{align}
	& &\varphi - \kappa \nabla \cdot (\varMan \nabla\varphi) &= c^+_{0 \kappa} - c^-_{0 \kappa} & &\text{ in } \Omega,
	\label{eq:approx}\\
	& &\varMan \nabla \varphi \cdot \bn + \tau \varphi &= 0 & &\text{ on } \Gamma.
\end{align}
\end{subequations} 
to $c^+_0 - c^-_0$, we show that $\varphi_{0 \kappa} - \tilde{\varphi}_{0 \kappa}$ goes to zero.
With that, the convergence $\varphi_{0 \kappa} \to c^+_0 - c^-_0$ follows by the triangle inequality.
We now show $\varphi_{0 \kappa} - \tilde{\varphi}_{0 \kappa} \to 0$, by subtracting \eqref{eq:approx_approx}, 
the equation for $\tilde{\varphi}_{0 \kappa}$, 
from the equation \eqref{eq:approx} for $\varphi_{0 \kappa}$, and testing with the difference
$\varphi_{0 \kappa} - \tilde{\varphi}_{0 \kappa}$, which yields
\begin{align*}
	\int_\Omega |\varphi_{0 \kappa} &- \tilde{\varphi}_{ 0 \kappa }|^2 \diff \X
	+ \kappa \int_\Omega 
		\left( 
			\varMan \nabla \varphi_{0 \kappa} - \varMa \nabla \tilde{\varphi}_{0 \kappa}
		\right) 
		\cdot \nabla (\varphi_{0 \kappa} - \tilde{\varphi}_{0 \kappa})
	\diff \X
	+ \tau \int_\Gamma 
		|\varphi_{0 \kappa} - \tilde{\varphi}_{0 \kappa}|^2 
	\dS\\
	&=
		\int_\Omega |\varphi_{0 \kappa} - \tilde{\varphi}_{ 0 \kappa }|^2 \diff \X
		+ \tau \int_\Gamma |\varphi_{0 \kappa} - \tilde{\varphi}_{0 \kappa}|^2 \dS
		+ \kappa \int_\Omega 
			|\nabla (\varphi_{0 \kappa} - \tilde{\varphi}_{0 \kappa})|^2_{\varMan}
		\diff \X\\
		&\hspace{1em}+ \kappa  \int_\Omega 
			\left( \varMan - \varMa \right) \nabla \tilde{\varphi}_{0 \kappa} 
			\cdot \nabla  (\varphi_{0 \kappa} - \tilde{\varphi}_{0 \kappa})
		\diff \X
	=
		0.
\end{align*}
Using Young's inequality on the last term and the estimate 
\[
	\kappa \norm{\nabla \tilde{\varphi}_{0 \kappa}}^2_{L^2(\Omega)} \leq \sigmp{\norm{\cpm_{0 \kappa}}^2_{L^2(\Omega)}},
\]
which follows simply by testing \eqref{eq:approx_approx} with $\tilde{\varphi}_{0 \kappa}$, we can estimate
\begin{align}\label{eq:strong_conv_phi_0}
\begin{split}
    	\int_\Omega |\varphi_{0 \kappa} - \tilde{\varphi}_{ 0 \kappa }|^2 \diff \X
		&+ \tau \int_\Gamma |\varphi_{0 \kappa} - \tilde{\varphi}_{0 \kappa}|^2 \dS
		+ \kappa \int_\Omega 
			|\nabla (\varphi_{0 \kappa} - \tilde{\varphi}_{0 \kappa})|^2_{\varMan}
		\diff \X\\
	&\leq	
		\frac{\kappa}{2} \norm{\left( \varMan - \varMa \right) \nabla \tilde{\varphi}_{0 \kappa} }^2_{L^2(\Omega)}
		+ \frac{\kappa}{2} \norm{\nabla  (\varphi_{0 \kappa} - \tilde{\varphi}_{0 \kappa})}^2_{L^2(\Omega)}\\
	&\leq
		\frac{1}{2} \norm{\varMan - \varMa}^2_{L^\infty(\Omega)}
		\sigmp{\norm{\cpm_{0 \kappa}}^2_{L^2(\Omega)}}
		+\frac{\kappa}{2} \norm{\nabla  (\varphi_{0 \kappa} - \tilde{\varphi}_{0 \kappa})}^2_{L^2(\Omega)}.
\end{split}
\end{align}
The first term on the right-hand side goes to zero by the boundedness of $\norm{\cpm_{0 \kappa}}^2_{L^2(\Omega)}$ and the strong convergence
of $\bd_\kappa$ in $L^\infty(\Omega)$ and the second term can be absorbed into the left-hand side.
Thus we get $\varphi_{0 \kappa} - \tilde{\varphi}_{ 0 \kappa } \to 0$ and 
\begin{equation}\label{eq:strong_conv_phi_0_final}
	\varphi_{0 \kappa} \to c^+_0 - c_0^-  \text{ in } L^2(\Omega).
\end{equation}

Now we turn to the convergence in \eqref{eq:strong_conv_psi_0}.
Let $\psi_0$ be the solution to \eqref{eq:main_system_p} with right-hand side $c^+_0 - c^-_0$. 
Subtraction the equation for $\psi_0$ and $\psi_{0 \kappa}$ and testing with the difference, we obtain by Young's inequality
\begin{align}
    \begin{split}
        \label{eq:conv_psi_0_test}
	\norm{\psi_0 - \psi_{0 \kappa}}^2_{W^{1,2}(\Omega)}
	\leq{}&
		C \left( \int_\Omega |\nabla (\psi_0 - \psi_{0 \kappa})|^2_{\varMan} \diff \X 
		+ \tau \int_\Gamma |\psi_0 - \psi_{0 \kappa}|^2 \dS \right) \\
	\leq {}&
		C \int_\Omega
		(\varMan - \varMa) \nabla \psi_0 \cdot \nabla (\psi_0 - \psi_{0 \kappa})
		+ (c_0^+ - c_0^- - \varphi_{0 \kappa})  (\psi_0 - \psi_{0 \kappa})
		\diff \X
    \\&+ C \int_\Gamma
        (\xi - \xi_\kappa) (\psi - \psi_\kappa)
    \dS\\
	\leq{}&
		C \left( 
			\norm{(\varMan - \varMa) \nabla \psi_0}^2_{L^2(\Omega)} 
			+ \norm{c_0^+ - c_0^- - \varphi_{0 \kappa}}^2_{L^2(\Omega)}
            + \norm{\xi - \xi_\kappa}^2_{L^2(\Gamma)}
		\right)\\
		&+ \delta \norm{\psi_0 - \psi_{0 \kappa}}^2_{W^{1,2}(\Omega)}
 \end{split}
\end{align}
for all $\delta > 0$. Thus, we can absorb the second term on the right-hand side into the left-hand side and by the strong convergence of $\varphi_{0 \kappa}$, 
\textit{cf.~}\eqref{eq:strong_conv_phi_0_final},
the boundedness of $\norm{\nabla \psi_0}_{L^2(\Omega)} \leq C \sigmp{\norm{\cpm_0}_{L^2(\Omega)}}$, which follows by a simple testing, 
the strong convergence of $\bd_\kappa $ in $L^\infty(\Omega)$,
and the strong convergence of $\xi_\kappa$ in $W^{1,2}(0,T;L^2(\Gamma))$
the first term on right-hand side goes to zero and we find \eqref{eq:strong_conv_psi_0}.

The right-hand side of the first energy inequality, \textit{cf.}~\eqref{eq:reg_en_ineq_1}, can be bounded independently of $\kappa$.
This can be seen by
\begin{multline}\label{eq:en_0_kappa}
	\left| \E_{\mathrm{reg}}(\bv_{0 \kappa}, \cpm_{0 \kappa}, \psi_{0 \kappa}) \right|\\
	\leq 
		C \left( 
			\norm{\bv_{0 \kappa}}^2_{L^2(\Omega)}
			+ (1 + \kappa) \sigmp{\norm{\cpm_{0 \kappa}}^2_{L^2(\Omega)}} 
			+ (1 + \norm{\bd_\kappa}^2_{L^\infty(\Omega)}) \norm{\psi_{0 \kappa}}^2_{W^{1,2}(\Omega)}
		\right)
	\leq C,
    \end{multline}
with $C > 0 $ independent of $\kappa$,  since convergent sequences are bounded,
where we used the elliptic estimate for the Robin Laplacian in $L^2(\Omega)$, 
\textit{cf.~}Lemma~\ref{lem:robin_semigroup_L2} and Lemma~\ref{lem:prop_reg_op},
and the fact that
\begin{equation}\label{eq:xi_kappa_bounded}
	\{\xi_\kappa \} \text{ is bounded in } 
	W^{1,2}(0,T;L^\infty(\Gamma))
\end{equation}
independently from $\kappa$. 
This boundedness follows from the fact that $\{ \xi_\kappa \}$ is bounded in $W^{1,2}(0,T;W^{2,2}(\Gamma))$, since it is convergent in that space,
and from the embedding $W^{2,2}(\Gamma) \hookrightarrow L^\infty(\Gamma)$, \textit{cf.~}\cite[Thm.~3.81]{demengel_demengel}.
Thus we have the boundedness of $\{ \bv_\kappa \}$ in 
\begin{equation}\label{eq:L103_bound_v}
	\V := L^\infty(0,T;L^2_\sigma(\Omega)) \cap L^2(0,T;W^{1,2}(\Omega)) \hookrightarrow L^{10/3}(\Omega \times (0,T) ),
\end{equation}
where the embedding is a simple consequence of H\"{o}lder's and Sobolev's inequalities.
By the boundedness of $\{ \bv_\kappa \}$ in the space $\V$ the weak convergence~\eqref{eq:weak_conv_v} follows directly.
We use an Aubin--Lions lemma to deduce the strong $L^2(\Omega \times (0,T) )-$convergence of $\{ \bv_\kappa \}$. 
To get an estimate for the time derivative of $\{ \bv_\kappa \}$ we use the equation 
\eqref{eq:full_reg_ns} and write
\begin{equation}\label{eq:est_time_deriv_v}
    \partial_t \bv_\kappa 
    = 
        -A_2 \bv_\kappa - P((R_\kappa(\bv_\kappa) \cdot \nabla )    \bv_\kappa)) - P(R^{1/2}_\kappa((c^+_\kappa - c^-_\kappa) \nabla \psi_\kappa)).
\end{equation}
The first term on the right-hand side is bounded in $L^2(0,T;W^{-1,2}(\Omega))$ and the second term is bounded in $L^{5/4}(0,T;L^{5/4}(\Omega))$, which we will prove below.
Over all the right-hand side is bounded in 
$L^{5/4}(0,T;W^{-1,2}(\Omega))$. Using the compact Aubin--Lions embedding
\[
    L^2(0,T;W^{1,2}(\Omega)) \cap W^{1,5/4}(0,T;W^{-1,2}(\Omega)) \hookrightarrow \hookrightarrow L^2(0,T;L^2(\Omega))
\]
we obtain the strong convergence \eqref{eq:strong_conv_v}.
The boundedness of $\{ \bv_\kappa \}$ in $L^\infty(0,T;L^2(\Omega))$ from the energy inequality~\eqref{eq:reg_en_ineq_1} together with \cite[Prop.~4.9]{droniou_eymard_talbot_2015}, gives \eqref{eq:weak_cont_conv_v}, at least along a subsequence.
We now show that the second term on the right-hand side of \eqref{eq:est_time_deriv_v} is indeed bounded in $L^{5/4}(\Omega \times (0,T) )$ independently of $\kappa$,
\begin{align*}
	\norm{P((R_\kappa(\bv_\kappa) \cdot \nabla ) \bv_\kappa))}_{L^{5/4}(\Omega \times (0,T) )} 
	\underset{\textnormal{H\"{o}lder}}&{\leq} 
		\norm{\nabla \bv_\kappa}_{L^2(\Omega \times (0,T) )} \norm{R_\kappa(\bv_\kappa)}_{L^{10/3}(\Omega \times (0,T) )}\\
	\underset{\mathrm{Lem.~}\ref{lem:prop_reg_op}}&{\leq}
		C \norm{\nabla \bv_\kappa}_{L^2(\Omega \times (0,T) )} \norm{\bv_\kappa}_{L^{10/3}(\Omega \times (0,T) )}
	\underset{\eqref{eq:L103_bound_v}}{\leq}
		C.
\end{align*}
Next, we estimate the third term on the right-hand side of \eqref{eq:est_time_deriv_v}
\begin{align*}
	\norm{R^{1/2}_\kappa(P((c^+_\kappa - c^-_\kappa) \nabla \psi_\kappa))}_{L^{5/4}(\Omega \times (0,T) )} 
	\underset{\mathrm{Lem.~}\ref{lem:prop_reg_op}}&{\leq}
		C \norm{(c^+_\kappa - c^-_\kappa) \nabla \psi_\kappa}_{L^{5/4}(\Omega \times (0,T) )}\\
	\underset{\textnormal{H\"{o}lder}}&{\leq} 
		C \sigmp{\snorm{\sqrt{\cpm_\kappa}}_{ L^{10/3}(\Omega \times (0,T) ) }
		\snorm{\sqrt{\cpm_\kappa} \nabla \psi_\kappa}_{L^2(\Omega \times (0,T) )}},
	% \leq C,
\end{align*}
by the energy estimates \eqref{eq:reg_en_ineq_1} and \eqref{eq:reg_en_ineq_2} the right-hand side is bounded, since by these estimates we have that $\big\{ \sqrt{\cpm_\kappa \,} \, \big\}$ is bounded in 
\begin{equation}\label{eq:bound_sqrt_c}
	L^\infty(0,T;L^2(\Omega)) \cap L^2(0,T;W^{1,2}(\Omega)) \hookrightarrow L^{{10}/{3}}(\Omega \times (0,T) ).
\end{equation}
%The right-hand side of the second energy inequality \eqref{eq:reg_en_ineq_2} is indeed bounded independently of $\kappa$, since
%by \eqref{eq:en_0_kappa} we have $|\E_{\mathrm{reg}}(\bv_{0 \kappa}, \cpm_{0 \kappa}, \psi_{0 \kappa})| \leq C$ independently of $\kappa$ 
%and by the assumption \eqref{eq:d_W2inf_bound} we have $\kappa \norm{\bd_\kappa}^2_{W^{2,\infty}(\Omega)} \leq C$,
%where we again used the boundedness of $\{ \xi_\kappa \}$ from above \eqref{eq:xi_kappa_bounded}.

We finished proving that the third term on the right-hand side of \eqref{eq:est_time_deriv_v} is bounded in $L^{5/4}(\Omega \times (0,T) )$ 
and thus we have proven the strong convergence in 
\eqref{eq:strong_conv_v}.
%and thus also the pointwise everywhere convergence \eqref{eq:strong_conv_v_t} holds, at least along a subsequence.
With this strong convergence we can deduce $R_\kappa(\bv_\kappa) \to \bv_\kappa$ and $R^{1/2}_\kappa(\bv_\kappa) \to \bv_\kappa$ in 
$L^2(\Omega \times (0,T) )$ by Lemma~\ref{lem:prop_reg_op}.

From the boundedness of $\big \{ \sqrt{\cpm_\kappa} \; \big\}$ in \eqref{eq:bound_sqrt_c} and the equation $\nabla \cpm_\kappa = 2 \sqrt{\cpm_\kappa} \nabla \sqrt{\cpm_\kappa}$, we infer the estimate,
\begin{align}\label{ineq:cpm}
    \| \nabla \cpm_\kappa\|_{L^{5/4}(\Omega\times (0,T))} \leq 2 \big \| \sqrt{\cpm_\kappa}\big\|_{L^{10/3}(\Omega\times (0,T))} \big \| \nabla \sqrt{\cpm_\kappa}\big\| _{L^{2}(\Omega\times (0,T))} \leq C \,.
\end{align}
Using the equations~\eqref{eq:full_reg_np}, we may estimate the time derivative $\partial_t \cpm_\kappa $ via 
\begin{align*}
    \| \partial_t \cpm_\kappa \|_{L^{10/9}(0,T;W^{-1,10/9}(\Omega))} 
    \leq {}& \big \| - R^{1/2}_\kappa(\bv_\kappa) \cpm_\kappa 
		+  \lamMan (\nabla \cpm_\kappa \pm \cpm_\kappa \nabla \psi_\kappa) \big\|_{L^{10/9}(\Omega\times (0,T))} 
  \\
  \leq {}& \big \|  R^{1/2}_\kappa(\bv_\kappa)  \big\|_{L^{10/3}(\Omega\times (0,T))} \big \| \sqrt{\cpm_\kappa}\big\|_{L^{10/3}(\Omega\times (0,T))}^2 
  \\&+ C \| \lamMan \| _{L^\infty(\Omega)}  \| \nabla \cpm_\kappa\|_{L^{5/4}(\Omega\times (0,T))}
  \\&+ C \| \lamMan \| _{L^\infty(\Omega)} \snorm{\sqrt{\cpm_\kappa}}_{ L^{10/3}(\Omega \times (0,T) ) }
		\snorm{\sqrt{\cpm_\kappa} \nabla \psi_\kappa}_{L^2(\Omega \times (0,T) )} \,.
\end{align*}
Hence, we infer that $\{\cpm_\kappa\}$ is bounded in \begin{equation}\label{eq:time_reg_sqrt_c}
	W^{1,10/9}(0,T;W^{-1,10/9}(\Omega) ) \cap L^{5/4}(0,T;W^{1,5/4}(\Omega)) \hookrightarrow \hookrightarrow L^{5/4}(0,T;L^{5/4}(\Omega))
\end{equation}
due to the Aubin--Lions lemma. From the \textit{a priori} estimate providing the boundedness of $ \{ \cpm_{\kappa}\}$ in 
\[
    L^\infty(0,T;L^1(\Omega)) \cap L^1(0,T;L^3(\Omega)) \hookrightarrow L^{5/3}(\Omega\times (0,T))\,, 
\]
we infer $\cpm_\kappa \to \cpm $ in $L^p(\Omega\times (0,T))$ for $p< 5/3$ as well as $\cpm_\kappa \to \cpm $ in $L^1(0,T; L^r(\Omega))$ for $r<3$. This strong convergence and Lebesgue's theorem of dominated convergence provides that $\sqrt{\cpm_\kappa} \to \sqrt{\cpm}$ in $L^{q}(\Omega \times (0,T))$ for $q< 10/3$. 
The bound of the time derivative of $\{ \cpm_\kappa \}$ additionally gives us the convergence \eqref{eq:weak_star_cont_conv_c}. 
The embedding $ W^{1,10/9}(0,T;W^{-1,10/9}(\Omega)) \hookrightarrow C([0,T]; W^{-1,10/9}(\Omega)) $ even implies the pointwise convergence 
\begin{equation}
    \label{eq:pointwisec}
    \cpm_\kappa (t)  \rightharpoonup\cpm (t) \quad \text{ in } W^{-1,10/9}(\Omega) \text{ for all }t\in [0,T]\,.
\end{equation}
By the regularized energy \eqref{eq:reg_en_ineq_1}, which holds for all $t \in [0,T]$, we obtain for all $t \in [0,T]$ uniform (in $\kappa$) boundedness of $\cpm_\kappa(t)$ and $\cpm_\kappa(t) \ln \cpm_\kappa(t)$ in $L^1(\Omega)$
and by the de la Vall\'{e}e--Poussin theorem, \cite[Thm.~2]{claude_meyer_1978}, we can extract a weakly convergent 
sequence such that $\cpm_\kappa(t) \rightharpoonup a(t)$ in $L^1(\Omega)$.
With the help of the convergence \eqref{eq:pointwisec} we can identify
$a(t) = c(t)$ and thus we obtain \eqref{eq:weak_conv_c_ALL_t}.
By the regularized energy inequality we also obtain $\psi_\kappa(t) \rightharpoonup \chi(t)$ in $W^{1,2}(\Omega)$ for all $t \in [0,T]$
and by the convergence \eqref{eq:weak_conv_c_ALL_t} we can identify
$\chi(t) = \psi(t)$ through the equation and thus we obtain \eqref{eq:weak_conv_psi_ALL_t}.

Finally, we turn our attention to $\{ \psi_\kappa \}$.
% By the first energy inequality \eqref{eq:reg_en_ineq_1} we have the boundedness
% of $\{ \psi_\kappa \}$ in $L^\infty(0,T;W^{1,2}(\Omega))$ and we directly have the weak convergence \eqref{eq:weak_star_conv_psi}. 
By the boundedness of $\big \{ \sqrt{\cpm_\kappa} \nabla \psi_\kappa \big \}$ in $L^2(0,T;L^2(\Omega))$,
by the second energy inequality \eqref{eq:reg_en_ineq_2}, we find $\sqrt{\cpm_\kappa} \nabla \psi_\kappa \rightharpoonup \sqrt{\cpm} \nabla \psi$
in $L^2(0,T;L^2(\Omega))$ which gives \eqref{eq:weak_conv_c_nabla_psi}.
%By the strong convergence of $\cpm_\kappa$ in $L^1(0,T;L^2(\Omega))$ we find similarly to \eqref{eq:strong_conv_phi_0}
%that $\varphi_\kappa$ the solution to \eqref{eq:approx_sys}
%with right-hand side $c^+_\kappa - c^-_\kappa$ converges to $(c^+ - c^-)$ in %$L^1(0,T;L^2(\Omega))$.
% Similarly to \eqref{eq:conv_psi_0_test} we can deduce that \eqref{eq:strong_conv_psi} holds,
% which also gives \eqref{eq:strong_conv_psi_t} along a subsequence.
From the second energy inequality \eqref{eq:reg_en_ineq_2} we also find $\psi_\kappa \rightharpoonup \psi$ in $L^2(0,T;W^{2,2}(\Omega))$.

With the continuity properties of $R_\kappa$ and $R^{1/2}_\kappa$, \textit{cf.~}Lemma~\ref{lem:prop_reg_op},
the convergences \eqref{eq:convergences} are
enough to pass to the limit in the weak formulation of~\eqref{eq:full_reg_system} and thus the limit is a weak solution to \eqref{eq:main_system}.
We only show the proof of the convergence of the nonlinear terms, as the convergence of the linear terms is 
straight forward.
By the strong convergence of $R_\kappa(\bv_\kappa)$ in $L^2(\Omega \times (0,T) )$ and
the weak convergence of $\nabla \bv_\kappa$ in $L^2(\Omega \times (0,T) )$ we get the weak convergence of 
$(R_\kappa(\bv_\kappa) \cdot \nabla) \bv_\kappa$ in $L^1(\Omega \times (0,T) )$. 
By the strong convergence $\bd_\kappa \to \bd$ in $L^\infty(\Omega)$
together with the weak convergence of $\nabla \cpm_\kappa$, \textit{cf.~}\eqref{eq:weak_conv_nabla_c}, 
%we first find $\bd_\kappa \cdot \nabla \cpm_\kappa \rightharpoonup \bd \cdot \nabla \cpm$ in $L^{5/4}(0,T;L^{5/4}(\Omega))$. 
%The boundedness of the sequence follows by~\eqref{ineq:cpm} and the associated weak limit can be identified using the strong convergence of the sequence $\{\bd_\kappa\}$ and the convergence~\eqref{eq:weak_conv_nabla_c}. 
% for all $p \in [1,3/2)$, 
% where we made use of a convergence result of the product of a weakly and a strongly convergent sequence
% in an anisotropic Lebesgue space, which is a consequence of H\"{o}lder's inequality.
%Again applying the same rule we obtain 
we find
\[
	\bd_\kappa \otimes \bd_\kappa \nabla \cpm_\kappa \rightharpoonup \bd \otimes \bd \, \nabla \cpm 
	\quad\quad \text{ in } L^{5/4}(0,T;L^{5/4}(\Omega)) \,.%\text{ for all } p \in [1,3/2).
\]
From the estimate $ \| \cpm_\kappa\nabla \psi_\kappa\|_{L^{5/4}(\Omega\times (0,T))} \leq \| \sqrt {\cpm_\kappa}\|_{L^{10/3}(\Omega\times(0,T))}\| \sqrt{\cpm_\kappa}\nabla \psi_\kappa\|_{L^2(\Omega\times (0,T))}$, we infer weak convergence and can immediately identify the weak limit due to the strong convergence~\eqref{eq:strong_conv_c} such that
$$\cpm_\kappa \nabla \psi_\kappa \rightharpoonup \cpm \nabla \psi \quad\quad \text{ in } L^{5/4}(0,T;L^{5/4}(\Omega))\,.$$
%
% By the convergence \eqref{eq:strong_conv_c} and \eqref{eq:weak_star_conv_psi} we find 
% $\cpm_\kappa \nabla \psi_\kappa \rightharpoonup \cpm \nabla \psi$ in $L^1(0,T;L^p(\Omega))$ for all $p \in [1,6/5)$.
% This can be seen by 
% \begin{equation}\label{eq:weak_conv_c_nabla_psi_L1}
% 	\int_0^T \int_\Omega \cpm_\kappa \nabla \psi_\kappa \cdot w \diff \X \diff t
% 	= 
% 		\int_0^T \int_\Omega (\cpm_\kappa - \cpm) \nabla \psi_\kappa \cdot w \diff \X \diff t
% 		+ \int_0^T \int_\Omega \cpm \nabla \psi_\kappa \cdot w \diff \X \diff t
% \end{equation}
% for arbitrary test function $w \in L^\infty(0,T;L^q(\Omega))$ for $q = p'$. 
% Since $p < 6/5 $ we have $q > 6$ and thus we find that the first term goes to zero, 
% \begin{align*}
% 	\left| \int_0^T \int_\Omega (\cpm_\kappa - \cpm) \nabla \psi_\kappa \cdot w \diff \X \diff t \right|
% 	&\leq 
% 		\norm{\cpm_\kappa - \cpm}_{L^1(0,T;L^{3-\epsilon}(\Omega))} \norm{\nabla \psi_\kappa}_{L^\infty(0,T;L^2(\Omega))}
% 		\norm{w}_{L^\infty(0,T;L^q(\Omega))}\\
% 	&\leq 
% 		C \norm{\cpm_\kappa - \cpm}_{L^1(0,T;L^{3-\epsilon}(\Omega))} \to 0.
% \end{align*}
% The second term of \eqref{eq:weak_conv_c_nabla_psi_L1} converges since we have $\cpm w \in L^1(0,T;L^2(\Omega))$.
Again using that the product of a weakly and a strongly convergent sequence is again weakly convergent, we obtain
\[
	\cpm_\kappa \bd_\kappa \otimes \bd_\kappa \nabla \psi_\kappa \rightharpoonup \cpm \bd \otimes \bd \, \nabla \psi 
	\quad\quad \text{ in } L^{5/4}(\Omega\times (0,T)).
\]
The bound $ \|\cpm_\kappa R^{1/2}_\kappa(\bv_\kappa)  \| _{L^{10/9}(\Omega \times (0,T))}\leq  \|\cpm_\kappa  \| _{L^{5/3}(\Omega \times (0,T))} \| R^{1/2}_\kappa(\bv_\kappa)  \| _{L^{10/3}(\Omega \times (0,T))} \leq C$, the 
 continuity of $R^{1/2}_\kappa$ together with the convergences \eqref{eq:weak_cont_conv_v}
and \eqref{eq:strong_conv_c},
also implies $$ \cpm_\kappa R^{1/2}_\kappa(\bv_\kappa) \rightharpoonup \cpm \bv \quad\quad \text{ in }L^{10/9}(\Omega \times (0,T) )\,.$$
Now, we have shown that we can pass to the limit in the weak formulation \eqref{eq:weak_form_c} for the charged particles $\cpm$.
Next, we turn to the Poisson equation.
By the pointwise convergence~\eqref{eq:weak_conv_psi_ALL_t}, we infer 
% By the strong convergence \eqref{eq:strong_conv_psi_t} we find 
\[
	\left( \bd_\kappa \cdot \nabla \psi_\kappa(t) \right) \bd_\kappa \rightharpoonup^* (\bd \cdot \nabla \psi(t)) \bd \quad \quad \text{ in } L^2(\Omega) \text{ for all }t\in[0,T]
 % \text{ for all } s \in [1,2) \text{ and a.e. in }(0,T)
 \,.
\]
Thus $(\bv, \cpm, \psi)$ fulfills the weak formulation \eqref{eq:weak_form_v}--\eqref{eq:weak_form_psi} for all $t \in [0,T]$.
%
%Lastly, we still need to proof that $\psi $ is even more than a weak solution, namely a strong $L^p-$solution. 
%For that we show the additional regularity $\psi \in L^1(0,T;W^{3,\frac{3}{2}}(\Omega))$.
%Which simply follows by elliptic regularity and the fact that $\cpm \in L^1(0,T;W^{1,\frac{3}{2}}(\Omega))$, by
%\begin{align*}
%	\norm{\nabla \cpm}_{L^1(0,T;L^{3/2}(\Omega))} 
%	\leq 2 \norm{\sqrt{\cpm}}_{L^2(0,T;L^6(\Omega))} \norm{\nabla \sqrt{\cpm}}_{L^2(0,T;L^2(\Omega))} \leq C.
%\end{align*}
\end{proof}

\begin{remark}
	By the uniqueness of the regularized solution it is straight forward to show the existence of a global weak solution to the regularized system.
	One can then show the existence of a global weak solution to the original system by a standard diagonal sequence argument.
\end{remark}

We next show that for our weak solutions a variant of the  energy inequality for the regularized system from Proposition~\ref{propo:reg_en_ineq_1} transfers to the limit.
\begin{propo}[Energy inequality for \eqref{eq:main_system}]\label{propo:stronger_en_ineq}
	Let Assumption~\ref{ass:1} hold and let $(\bv, \cpm, \psi)$ be the weak solution given by the limit from Proposition~\ref{propo:convergences} then 
	\begin{multline}\label{eq:strong_en_ineq}
		\left[ \int_\Omega \frac{1}{2} |\bv|^2 + \sigmp{\cpm (\ln \cpm + 1)} + \frac{1}{2} |\nabla \psi|^2_{\varMa} \diff \X
		+ \frac{\tau}{2} \int_\Gamma |\psi|^2 \dS \right] \bigg|_0^t\\
		+ \int_0^t \int_\Omega 
			|\nabla \bv|^2 
			+ \sigmp{|\nabla \sqrt{\cpm} \pm \sqrt{\cpm} \nabla \psi|_{\lamMa}^2} 
		\diff \X \diff s
		\leq 
		\int_0^t \int_\Gamma \psi \partial_t \xi \dS \diff s
	\end{multline}
	holds for \underline{all} $t \in (0,T)$.
    %and for some constant $C(\xi)$ dependent on $\xi$ it holds
	%\begin{align}\label{eq:strong_en_ineq_after_gronwall}
	%	\E(\bv, \cpm, \psi)(t) + \int_0^t \Diss(\bv, \cpm, \psi) \diff s 
	%	\leq 
	%		e^t \left( \E (\bv_0, \cpm_0, \psi_0) + C(\xi) \right),
	%\end{align}
	%for almost all $t \in (0,T)$.
\end{propo}

\begin{proof}
	Let $\{\bv_\kappa, \cpm_\kappa, \psi_\kappa\}$, $\{ \bd_\kappa \}$ and $\{ \xi_\kappa \}$ 
	be the sequences from Proposition~\ref{propo:convergences}, then 
	from \eqref{eq:full_reg_tested_added} in the proof of Proposition~\ref{propo:reg_en_ineq_1}, we have
	\begin{align}\label{eq:strong_en_ineq_kappa}
    \begin{split}
		\bigg[ \int_\Omega 
			\frac{1}{2} |\bv_\kappa|^2 
			&+ \sigmp{\cpm_\kappa (\ln \cpm_\kappa + 1)} 
			+ \frac{1}{2} |\nabla \psi_\kappa|^2_{\varMan} 
		\diff \X
		+ \frac{\tau}{2} \int_\Gamma |\psi_\kappa|^2 \dS \bigg] \bigg|_0^t 
        + \frac{\kappa}{2} \int_\Omega |\varphi_\kappa(t)|^2 \diff \X\\
		&+ \int_0^t \int_\Omega 
			|\nabla \bv_\kappa|^2 
			+ \sigmp{ |\nabla \sqrt{\cpm_\kappa} \pm \sqrt{\cpm_\kappa} \nabla \psi_\kappa|_{\lamMan}^2 } 
		\diff \X \diff s\\
		&\hspace{-1em}\leq 
			\frac{\kappa}{2} \norm{\varphi_{0 \kappa}}^2_{L^2(\Omega)}
			+ \int_0^t \int_\Gamma 
				\psi_\kappa \partial_t \xi_\kappa - \kappa \partial_t \xi_\kappa \varphi _\kappa
			\dS \diff s
			- \kappa \int_\Gamma \xi_\kappa \varphi_\kappa \dS \bigg|_0^t 
    \end{split}
	\end{align}
    for all $t \in (0,T)$, where $\varphi_\kappa$ is the solution to \eqref{eq:approx_sys} with right-hand side $c^+_\kappa - c^-_\kappa$.
    First, we observe that all terms with the factor $\kappa$ vanish as $\kappa\to 0$. 
	This follows from a Young estimate, the elliptic estimates for the Robin Laplacian in $L^2(\Omega)$, 
	\textit{cf.~}Lemma~\ref{lem:robin_semigroup_L2}, and for the Robin Laplacian in $L^1(\Omega)$,
	\textit{cf.~}Lemma~\ref{lem:robin_semigroup_L1} together with Lemma~\ref{lem:prop_reg_op} item \ref{item:R_kappa_bound}.
	These elliptic estimates give us
	\begin{subequations}
	\begin{align}
		\norm{\varphi_{0 \kappa}}_{L^2(\Omega)} &\leq \sigmp{\norm{\cpm_{0 \kappa }}_{L^2(\Omega)}} \leq C,\label{eq:ellip_L2_est_c_0}\\
		\norm{\varphi_{0 \kappa}}_{W^{2,2}(\Omega)} &\leq \frac{1}{\kappa} \sigmp{\norm{\cpm_{0 \kappa }}_{L^2(\Omega)}}
		\leq \frac{C}{\kappa},\label{eq:ellip_W22_est_c_0}\\
		%\norm{\varphi_\kappa(t)}_{L^1(\Omega)} &\leq \sigmp{\norm{\cpm_\kappa(t)}_{L^1(\Omega)}} = \sigmp{\norm{\cpm_{\kappa 0}}_{L^1(\Omega)}} \leq C, \label{eq:ellip_L1_est}\\
		\norm{\varphi_\kappa(t)}_{W^{1,q}(\Omega)} &\leq \frac{C}{\kappa} \sigmp{\norm{\cpm_\kappa(t)}_{L^1(\Omega)}}
		= \frac{C}{\kappa} \sigmp{\norm{\cpm_{\kappa 0}}_{L^1(\Omega)}}
		\leq \frac{C}{\kappa} \text{ for } q \in[1, 3/2), \label{eq:ellip_W1q_est}\\
		\norm{\varphi_\kappa(t)}_{W^{2,2}(\Omega)} 
		&\leq \frac{1}{\kappa} \sigmp{\norm{\cpm_\kappa(t)}_{L^2(\Omega)}} \label{eq:ellip_W22_est}
	\end{align}
	\end{subequations}
	for almost all $t \in (0,T)$ and $C>0$ independent of $\kappa$ and in particular independent of $\bd_\kappa$, 
	\textit{cf.~}Remark~\ref{rem:C_indep_d}, 
	and by the regularized energy inequality \eqref{eq:reg_en_ineq_1} we have
	\begin{equation}\label{eq:en_est_phi}
		\sqrt{\kappa} \norm{\varphi_\kappa}_{L^\infty(0,T;L^2(\Omega))} \leq C.
	\end{equation}
	To see that the terms on the right-hand side of \eqref{eq:strong_en_ineq_kappa} with prefactor $\kappa$ vanish,
	we use Gagliardo--Nirenberg's inequality \cite[Thm.~1.24]{roubicek},
	\[
		\norm{ \varphi}_{W^{1,2}(\Omega)} 
		\leq 
			C_{\mathrm{GN}} \norm{\varphi}^{1/2}_{L^2(\Omega)} \norm{\varphi}^{1/2}_{W^{2,2}(\Omega)}.
	\]
	With that, we estimate
	\begin{align*}
		\left|\int_0^t \int_\Gamma \kappa \partial_t \xi_\kappa \varphi_\kappa \dS \diff s \right|
		\leq{}&
			\kappa \norm{\partial_t \xi_\kappa}_{L^2(0,T;L^\infty(\Gamma))} \norm{\varphi_\kappa}_{L^2(0,T;L^1(\Gamma))}\\
		\underset{\mathrm{trace~estimate}}{\phantom{estimate}\leq}{}&
			C \kappa \norm{\partial_t \xi_\kappa}_{L^2(0,T;L^\infty(\Gamma))} \norm{\varphi_\kappa}_{L^2(0,T;W^{1,2}(\Omega))}\\
	% 	\underset{\mathrm{Young}}{\phantom{Yo}\leq}{}&
	% 		\sqrt{\kappa} C \norm{\partial_t \xi_\kappa}^2_{L^2(0,T;L^\infty(\Gamma))}
	% 		+ \kappa^{3/2} \norm{\varphi_\kappa}^2_{L^2(0,T;W^{1,2}(\Omega))}\\
		\underset{\mathrm{GN}}{\leq}{}&
			{\kappa} C \norm{\partial_t \xi_\kappa}_{L^2(0,T;L^\infty(\Gamma))}
% \left[\int_0^T 
% 				\norm{\varphi_\kappa(t)}^2_{L^2(\Omega)} 
	 C_{\mathrm{GN}} \norm{\varphi_\kappa}_{L^\infty(0,T;L^2(\Omega))}^{1/2} \norm{ \varphi_\kappa}_{L^1(0,T;W^{2,2}(\Omega))}^{1/2}
			% \diff t
   % \right]^{1/2}
   \\
		\underset{\eqref{eq:en_est_phi}, \eqref{eq:ellip_W22_est}}{\phantom{M}\leq}{}&
			\kappa^{1/4} C \norm{\partial_t \xi_\kappa}_{L^2(0,T;L^\infty(\Gamma))}
			% + \kappa^{3/2} C T
			% + \sqrt{\kappa} C 
   		\sigmp{\norm{\cpm_\kappa}_{L^1(0,T;L^2(\Omega))}^{1/2}}
		 \searrow 0,
	\end{align*}
	where we used that $\norm{\cpm_\kappa}_{L^1(0,T;L^2(\Omega))}$ is bounded independently of $\kappa$, 
	since $\cpm_\kappa$ is convergent in $L^1(0,T;L^2(\Omega))$, \textit{cf.~}\eqref{eq:strong_conv_c},
	and the boundedness of $\{ \xi_\kappa \}$, \textit{cf.~}\eqref{eq:xi_kappa_bounded}.
	Since $\cpm_{0 \kappa}$ is bounded in $L^2(\Omega)$ independently of $\kappa$ 
	the elliptic estimate for the Robin Laplacian \eqref{eq:ellip_L2_est_c_0} implies 
	\[
		\frac{\kappa}{2}\norm{\varphi_{0 \kappa}}_{L^2(\Omega)} \leq \frac{\kappa}{2} \sigmp{\norm{\cpm_{0 \kappa}}_{L^2(\Omega)}} 
		\leq C \frac{\kappa}{2}  \searrow 0
	\]
	and the elliptic estimate \eqref{eq:ellip_W22_est_c_0} implies from $ W^{1,2}(\Omega) \hookrightarrow H^{1/2}(\Gamma ) \hookrightarrow L^2(\Gamma) \hookrightarrow L^1(\Gamma) $ that
	\begin{align*}
		\left| \kappa \int_\Gamma \xi_\kappa(0) \varphi_{0 \kappa} \dS \right| 
		\leq{}&
			\kappa \norm{\xi_\kappa}_{C([0,T];L^\infty(\Gamma))} \norm{\varphi_{0 \kappa}}_{L^1(\Gamma)}
		\leq 
			C \kappa \norm{\xi_\kappa}_{C([0,T];L^\infty(\Gamma))} \norm{\varphi_{0 \kappa}}_{W^{1,2}(\Omega)}\\
		\leq{}&
			C \kappa \norm{\xi_\kappa}_{C([0,T];L^\infty(\Gamma))} 
			\norm{\varphi_{0 \kappa}}^{1/2}_{L^2(\Omega)} \norm{\varphi_{0 \kappa}}^{1/2}_{W^{2,2}(\Omega)}\\
		\leq{}&
			C \kappa^{1/2} \norm{\xi_\kappa}_{C([0,T];L^\infty(\Gamma))} 
        \searrow 0.	
	\end{align*}
	For $s,p \in \R$ with $p \in (1,2)$ and $s > 1/p$ it holds $W^{s,p}(\Omega) \hookrightarrow W^{s - 1/p, p}(\Gamma)$, 
	\cite[Thm.~5.2]{lions_magenes_1961},
	where $W^{s,p}(\Omega)$ and $W^{s - 1/p, p}(\Gamma)$ denote the Sobolev--Slobodeckij spaces, 
	see \cite[Def.~2.2 and Def.~2.3]{lions_magenes_1961}. 
	Additionally, we use a generalized Gagliardo--Nirenberg inequality for Sobolev--Slobodeckij spaces \cite[Thm.~II.3-3]{oru_1998}.
	For $0 \leq s_1 < s_2 < \infty$, $1 < p_1, p_2 < \infty$, $\theta \in (0,1)$ and 
	$f \in W^{s_1,p_1}(\Omega) \cap W^{s_2,p_2}(\Omega)$ it holds
	\begin{equation}\label{eq:frac_interpol}
		\norm{f}_{W^{s,p}(\Omega)}
		\leq 
			C \norm{f}^{\theta}_{W^{s_1,p_1}(\Omega)} \norm{f}^{1 - \theta}_{W^{s_2,p_2}(\Omega)},
	\end{equation}
	for $s = \theta s_1 + (1-\theta) s_2 $ and $\frac{1}{p} = \frac{\theta}{p_1} + \frac{1 - \theta}{p_2}$.
	By the elliptic estimate \eqref{eq:ellip_W1q_est} with $q = 5/4$ and \eqref{eq:frac_interpol} with $\theta = 1/4 $, $s_1=0$, $p_1=2$,
	$s_2 = 1$ and $p_2 = 5/4$, that is with $s=3/4$ and $p = 40/29$ we obtain:
	\begin{align*}
		\left| \kappa \int_\Gamma \xi_\kappa(t) \varphi_\kappa(t) \dS \right|
		\leq{}&
			\kappa \norm{\xi_\kappa}_{C([0,T];L^\infty(\Gamma))} \norm{\varphi_\kappa(t)}_{L^1(\Gamma)}
		\leq 
			C \kappa \norm{\xi_\kappa}_{C([0,T];L^\infty(\Gamma))}  \norm{\varphi_\kappa(t)}_{L^{40/29}(\Gamma)}\\
		\leq{}&
			C \kappa \norm{\xi_\kappa}_{C([0,T];L^\infty(\Gamma))}  \norm{\varphi_\kappa(t)}_{W^{3/4 - 29/40, 40/29}(\Gamma)}\\
		\underset{\mathrm{trace~estimate}}{\phantom{estimate}\leq}{}&
			C \kappa \norm{\xi_\kappa}_{C([0,T];L^\infty(\Gamma))}  \norm{\varphi_\kappa(t)}_{W^{3/4, 40/29}(\Omega)}\\
		\underset{\mathrm{interpolation~}\eqref{eq:frac_interpol}}{\phantom{interpolati}\leq}{}&
			C \kappa \norm{\xi_\kappa}_{C([0,T];L^\infty(\Gamma))} 
			\norm{\varphi_\kappa(t)}^{1/4}_{L^2(\Omega)}
			\norm{\varphi_\kappa(t)}^{3/4}_{W^{1, 5/4}(\Omega)}\\
		\underset{\mathrm{Young~with~} p = 8}{\phantom{Young with}\leq}{}&
			\frac{\kappa}{2} \norm{\varphi_\kappa(t)}^2_{L^2(\Omega)}
			+ C \kappa \norm{\xi_\kappa}^{8/7}_{C([0,T];L^\infty(\Gamma))}
			\norm{\varphi_\kappa(t)}^{6/7}_{W^{1, 5/4}(\Omega)}\\
		\underset{\eqref{eq:ellip_W1q_est}}{\phantom{M}\leq}{}&
			\frac{\kappa}{2} \norm{\varphi_\kappa(t)}^2_{L^2(\Omega)}
			+ C \kappa \norm{\xi_\kappa}^{8/7}_{C([0,T];L^\infty(\Gamma))} \kappa^{- 6/7}\\
		={}&
			\frac{\kappa}{2} \norm{\varphi_\kappa(t)}^{2}_{L^2(\Omega)}
			+ C  \kappa^{1/7} \norm{\xi_\kappa}^{8/7}_{C([0,T];L^\infty(\Gamma))}.
	\end{align*}
	The first term can be absorbed into the left-hand side of \eqref{eq:strong_en_ineq_kappa} 
	and the other term vanishes for $\kappa \searrow 0$.
	
	The only thing left to do is to argue that the convergences given by Proposition~\ref{propo:convergences} 
	are enough to pass to the limit in \eqref{eq:strong_en_ineq_kappa}.
	This follows from
 	% the convergence \eqref{eq:strong_conv_psi_t} implies 
	% the convergence of $\psi_\kappa(t) \to \psi(t)$ in $L^2(\Gamma)$ for almost all $t \in (0,T)$ and 
 	the convergence \eqref{eq:weak_conv_psi_W22}, which implies the convergence $\psi_\kappa \rightharpoonup \psi $ in $L^2(0,T;L^2(\Gamma))$
	by the continuity and linearity of the trace operator.
	Together with the strong convergence $\partial_t \xi_\kappa \to \partial_t \xi$ in $L^2(0,T;L^2(\Gamma))$ we obtain the weak
	convergence of the product $\partial_t \xi_\kappa \psi_\kappa \rightharpoonup \partial_t \xi \psi$ in $L^1(0,T;L^1(\Gamma))$.
	% The convergence of the $\cpm_\kappa (\ln \cpm_\kappa + 1)$ term in \eqref{eq:strong_en_ineq_kappa} 
	% follows from Vitali's Theorem, 
	% since we have pointwise almost everywhere convergence by \eqref{eq:ptw_conv_c} and by the strong convergence \eqref{eq:strong_conv_c_t} in 			$L^p(\Omega)$ for $p \in [1,3)$ we have that $\cpm_\kappa(t) (\ln \cpm_\kappa(t) + 1)$ is bounded in $L^\frac{p}{2}(\Omega)$ for $p \in (2,3)$ 
	% for almost all $t \in (0,T)$ and thus $\cpm_\kappa(t) (\ln \cpm_\kappa(t) + 1)$ is equiintegrable, which follows from
	% \begin{align*}
	% 	\norm{\cpm_\kappa (\ln \cpm + 1)}^{\frac{2}{p}}_{L^{\frac{p}{2}}(\Omega)} 
	% 	\leq{}&
	% 		\int_\Omega |\cpm_\kappa (\ln \cpm + 1)|^{\frac{p}{2}} \diff \X\\
	% 	={}&
	% 		\int_{ \{\cpm_\kappa \leq 1 \}}  |\cpm_\kappa (\ln \cpm + 1)|^{\frac{p}{2}} \diff \X
	% 		+ \int_{ \{ \cpm_\kappa > 1 \} }  |\cpm_\kappa (\ln \cpm + 1)|^{\frac{p}{2}} \diff \X\\
	% 	\underset{|x \ln(x)| \leq x^2,\; x \geq 1}{\phantom{MMMMN}\leq}{}&
	% 		|\Omega| e^{-p/2}
	% 		+ \int_{ \{\cpm_\kappa > 1 \} } 
	% 			|\left(\cpm_\kappa\right)^2 + \cpm_\kappa|^{\frac{p}{2}}  \diff \X\\
	% 	\leq{}&
	% 		|\Omega| + 2^{\frac{p}{2}} \int_\Omega |\cpm_\kappa|^p + |\cpm_\kappa|^{\frac{p}{2}} \diff \X
	% 	\leq C
	% \end{align*}
	% for almost all $t \in (0,T)$.
 %    Thus we can apply Vitali's Theorem to obtain 
 %    $\cpm_\kappa(t) (\ln \cpm_\kappa(t) + 1) \to \cpm(t) (\ln \cpm(t) + 1)$ in $L^1(\Omega)$.
	By the weak convergences \eqref{eq:weak_cont_conv_v}, \eqref{eq:weak_conv_c_ALL_t}, \eqref{eq:weak_conv_psi_ALL_t}, \eqref{eq:weak_conv_v}, \eqref{eq:weak_conv_sqrt_c} and \eqref{eq:weak_conv_c_nabla_psi}
	and the weak lower semicontinuity of the norm and the convex function $a\mapsto a(\ln a+1)$, we can pass to the limit in \eqref{eq:strong_en_ineq_kappa} keeping the inequality sign
	and thus we obtain \eqref{eq:strong_en_ineq}.
    %Note that the same argument works for the term~$\cpm_\kappa(\ln \cpm_\kappa  +1)$ due to the convexity of the function $a\mapsto a(\ln a+1)$. 
	%Using Gronwall's inequality in \eqref{eq:strong_en_ineq} we obtain \eqref{eq:strong_en_ineq_after_gronwall}.
\end{proof}

Putting the results of Proposition~\ref{propo:convergences} and \ref{propo:stronger_en_ineq} together we can prove our main theorem.
\begin{proof}[Proof (of the Theorem~\ref{thm:exist_weak_sol})]
    The limit from Proposition~\ref{propo:convergences} already fulfills the weak formulation \eqref{eq:weak_form_v}--\eqref{eq:weak_form_psi} 
    and the energy inequality \eqref{eq:strong_en_ineq_def} for all $t \in [0,T]$, \textit{cf.~}Proposition~\ref{propo:stronger_en_ineq}.
    The only thing left to prove for $(\bv, \cpm, \psi)$ to be a weak solution according to Definition~\ref{def:weak_sol} is
    \begin{align*}
        \cpm \in C_w([0,T];L^1(\Omega)) \quad \text{ and } \quad 
        \psi \in C_w([0,T];W^{1,2}(\Omega)).
    \end{align*}
  
    We will see that $\cpm \in C_w([0,T];L^1(\Omega))$ by an application of
    the de la Vall\'{e}e--Poussin theorem, \cite[Thm.~B.104]{leoni_2017}.
    We take an arbitrary $t \in [0,T]$ and any sequence $(t_n)_n \subseteq [0,T]$ such that 
    $t_n \to t$. 
    Now, we consider an arbitrary subsequence $(t_{n_k})_k$ and get the uniform boundedness 
    \[
        \sup_k \left( \norm{\cpm(t_{n_k})}_{L^1(\Omega)} 
        + \norm{\cpm(t_{n_k}) \ln(\cpm(t_{n_k}))}_{L^1(\Omega)} \right) \leq C,
    \]
    for some $C>0$ independent of $n$ and $k$,
    by the energy inequality for the limit \textit{cf.~}Proposition~\ref{propo:stronger_en_ineq}.
    Thus, by the de la Vall\'{e}e--Poussin theorem, we find a subsequence of $(t_{n_k})_n$ which we will not relabel such 
    that $\cpm(t_{n_k}) \rightharpoonup \cpm(t)$ in $L^1(\Omega)$,
    where we identified the limit with the help of $\cpm \in C([0,T];W^{-1,10/9}(\Omega))$ by \eqref{eq:weak_conv_c_ALL_t}.
    Since the subsequence $(t_{n_k})_k$ was arbitrary we obtain the convergence
    of the original sequence and thus $\cpm(t_n) \rightharpoonup \cpm(t)$ in $L^1(\Omega)$ and we have proven $\cpm \in C_w([0,T];L^1(\Omega))$.
    By the boundedness of $(\psi(t_{n_k}))_k$ in $W^{1,2}(\Omega)$ from Proposition~\ref{propo:stronger_en_ineq}
    and the reflexivity of $W^{1,2}(\Omega)$, we can extract a weakly convergent subsequence and we identify the limit with the help of the continuity of 
    $\cpm \in C_w([0,T]; L^1(\Omega))$. Thus, our proof is complete.
\end{proof}

\end{subsection}

\end{section}

\bibliography{phd.bib}
\bibliographystyle{plainurl}

\renewcommand{\thesection}{A}
\setcounter{section}{0}
\setcounter{thm}{0}
\renewcommand{\thethm}{A.\arabic{thm}}
\numberwithin{equation}{section}
\begin{section}{Appendix}

    \begin{subsection}{Functions on the boundary}\label{sec:bd_func}

In this section we recall the definitions of Sobolev spaces on the boundary and state some important properties,
namely an integration by parts rule on the boundary and the relation of the tangential projection of a gradient of a function defined on the whole space 
and the surface gradient of that functions restriction to the boundary.
We begin by recalling the definition of the regularity of a domain in dimension $d \in \N$ with $d \geq 2$.
For that we need the notion of local coordinates, which are simply translated and rotated coordinates, \textit{cf.~}\cite[below Def.~9.57]{leoni_2017}.
\begin{definition}\label{def:local_coord}
	We call $T: \R^d \to \R^d$ rigid motion if it is given by $T(\X) = \bs c + R(\X)$ for some $\bs c \in \R^d$ and $R: \R^d \to \R^d$ a 
	rotation, \textit{cf.~}\cite[Def.~9.23]{leoni_2017}.
	We then call $\Y = T(\X)$ local coordinates.
\end{definition}
We can now define the notion of regularity of a domain, where we follow \cite[Sec.~1.1.3]{necas_2012}
and \cite[Def.~9.57]{leoni_2017}.
\begin{definition}\label{def:l_domain}
	We call a bounded domain $\Omega \subseteq \R^d $ of type $C^{k,1}$ with $k \in \N_0$
    and write $\Gamma = \partial \Omega \in C^{k,1}$, if there exist $\alpha, \beta > 0$, $M \in \N$ and local coordinates 
	$\Y_r = (y_{r1}, \dots, y_{rd}) = T_r(\X)$ for rigid motions $T_r$, $r = 1, \dots, M$ and functions 
	\[
		a_r: \Delta_r:= \set{\Y'_r = (y_{r1}, \dots, y_{r(d-1)}) \in \R^{d-1}}{|y_{ri}| < \alpha \text{ for } i = 1, \dots, d-1} \to \R
	\]
	with $a \in C^{k,1}(\Delta_r)$ 
	such that for all $\bs p \in \Gamma$ there exists an $r$ and $\Y'_r \in \Delta_r$ such that $T_r(\bs p) = (\Y'_r, a(\Y'_r))$
	and 
	\begin{align*}
		\set{\Y_r \in \R^d}{a_r(\Y'_r) < y_{rd} < a_r(\Y'_r) + \beta} &\subseteq T_r(\Omega) \text{ and }\\
		\set{\Y_r \in \R^d}{a_r(\Y'_r) - \beta < y_{rd} < a_r(\Y'_r)} &\subseteq T_r\left( \R^d \setminus \overline{\Omega} \right).
	\end{align*}
	We now define $\varphi^{-1}_r: \Delta_r \to \Gamma$ by $\varphi^{-1}_r(\Y'_r) = T_r^{-1}(\Y'_r, a(\Y'_r))$.
	This map is injective and of class $C^{k,1}$ and thus has an inverse $\varphi_r: V_r \subseteq \Gamma \to \Delta_r$, where $V_r := \varphi^{-1}(\Delta_r)$.
	We call $\varphi_r$ local chart of $\Gamma$ at $\bs p$ and the set $\set{(\varphi_r, V_r)}{r = 1, \dots, M}$ atlas of $\Gamma$.
	If $k = 0$, that is if the $a_r$ are only Lipschitz continuous, we call the domain $\Omega$ Lipschitz.
\end{definition}
\begin{remark}\label{rem:bi_Lipschitz}
	The function $\varphi^{-1}$ from Definition~\ref{def:l_domain} is even bi-Lipschitz continuous, since
	for $\X', \Y' \in \Delta$ we have
	\begin{align*}
		\norm{\left( \begin{array}{c} \X' \\ a(\X') \end{array} \right) - \left( \begin{array}{c} \Y' \\ a(\Y') \end{array} \right)}
		&=
			\sqrt{\norm{\X' - \Y'}^2 + |a(\X') - a(\Y')|}
		\geq 
			\norm{\X' - \Y'}\\
		\norm{\left( \begin{array}{c} \X' \\ a(\X') \end{array} \right) - \left( \begin{array}{c} \Y' \\ a(\Y') \end{array} \right)}
		&=
			\sqrt{\norm{\X' - \Y'}^2 + |a(\X') - a(\Y')|}
		\leq  
			(1 + L) \norm{\X' - \Y'},
	\end{align*}
	since $a$ is Lipschitz. 
	The bi-Lipschitz continuity of $\varphi^{-1}$ now follows from the fact that a rigid motion $T$, \textit{cf.~}Definition~\ref{def:local_coord},
	preserves angles and distances. 
	From the bi-Lipschitz property of $\varphi^{-1}$ the bi-Lipschitz property of $\varphi$ follows from a straight forward calculation.
\end{remark}
Next, we define $L^p-$spaces on the boundary, \textit{cf.~}\cite[Sec.~2.4.1]{necas_2012}, where $p \in (1,\infty)$ throughout this section and 
$p' = p/(p-1)$ its dual exponent.
For the remainder of this section we assume the domain to be Lipschitz unless explicitly stated otherwise. 
\begin{definition}\label{def:bd_Lp}
	Let $\set{(\varphi_r, V_r)}{r = 1, \dots, M}$ be an atlas of $\Gamma$ and $k \in \N$.
	The space $L^p(\Gamma)$ is defined as the set of functions $f: \Gamma \to \R^k$ defined almost everywhere 
	for which $f \circ \varphi_r^{-1}$ is in $L^p(\Delta_r)$ for all $r = 1, \dots, M$.
\end{definition}

\begin{remark}\label{rem:surf_measure}
	Almost everywhere in $\Gamma$ is to be understood with respect to the surface measure given by
	\[
		\mu(\Sigma) = \int_{\varphi(\Sigma)} \sqrt{\det\left( (D\varphi^{-1})^T D \varphi^{-1} \right)}\diff \Y'_r
	\]
	for $\Sigma \subseteq \varphi^{-1}(\Delta)$
	and generalized to arbitrary subsets of $\Gamma$ by a partition of unity, 
	\textit{cf.~}\cite[Prop.~B.4]{skrepek_2023}.
\end{remark}

And similarly we define Sobolev space of the boundary, \textit{cf.~}\cite[below Thm.~4.10]{necas_2012}.
\begin{definition}
	Let $\set{(\varphi_r, V_r)}{r = 1, \dots, M}$ be an atlas of $\Gamma$ and $k \in \N$.
	The space $W^{1,p}(\Gamma)$ is defined as the set of functions $f: \Gamma \to \R^k$ defined almost everywhere,
	such that $f \circ \varphi_r^{-1}$ is in $W^{1,p}(\Delta_r)$ for all $r = 1, \dots, M$.
\end{definition}

\begin{remark}
	These definitions are independent of the choice of atlas, \textit{cf.~}\cite[Sec.~3.1, Lem.~1.1]{necas_2012}.
    The proof is based on a coordinate transformation and the fact that if $u \in W^{1,p}(U)$ and $\phi: U \to \tilde{U}$ is a bi-Lipschitzian map 
	between two bounded open sets of $\R^d$
	than $u \circ \phi^{-1} \in W^{1,p}(\tilde{U})$,
	\textit{cf.~}\cite[Cor.~3.46]{demengel_demengel}.
\end{remark}
\begin{comment}
\begin{remark}
	Here the weak derivative is defined in the ``classical'' sense, 
	that is we define $ \partial_i (f \circ \varphi_r^{-1})$ as the function in $L^p(\Delta)$ fulfilling
	\[
		\int_{\Delta_r} \partial_i (f \circ \varphi_r^{-1}) \, \Phi \diff \Y'_r
		= 
			- \int_{\Delta_r} (f \circ \varphi_r^{-1}) \, \partial_i \Phi \diff \Y'_r
	\]
	for all smooth test function with compact support $\Phi \in C_0^\infty(\Delta)$,
	where $\partial_i \Phi$ denotes the classical partial derivative in local coordinates, that is we define
	\[
		\partial_i \Phi (\Y'_r) = \partial_{y_{ri}} \Phi (\Y'_r) := \lim_{h \to 0} \frac{\Phi(\Y'_r + h \bs{e}'_i) - \Phi(\Y'_r)}{h}
	\]
\end{remark}
\end{comment}
Following \cite[Sec.~2.4]{necas_2012}, we recall.
\begin{lem}[$L^p(\Gamma)$ is Banach]
	Let $\set{(\varphi_r, V_r)}{r = 1, \dots, M}$ be an atlas of $\Gamma$ and $\set{\alpha_r}{r = 1, \dots, M}$ a partition of unity
	subordinate to the open cover 
	$T^{-1}_r \left( \Delta_r \times \set{y_{rd} \in \R}{\exists \Y'_r \in \Delta_r : |y_{rd} - a(\Y'_r)| < \beta } \right)$ of $\Gamma$.
	%This is indeed an open cover by the definition of Lipschitz domains, \textit{cf.~}Definition~\ref{def:l_domain}
	%and the subordinate partition of unity exists due to \cite[Sec.~1.2, Prop.~2.3]{necas_2012}.
	Equipped with the norm $\norm{\cdot}_{L^p(\Gamma)}$ given by
	\begin{equation}\label{eq:bd_Lp_norm}
		\norm{f}_{L^p(\Gamma)} 
		:= 
			\left( \int_\Gamma |f|^p \diff \sigma \right)^{\frac{1}{p}} 
		:=
			\left( \sum_{r = 1}^M \int_{\Delta_r} |f  \circ \varphi_r^{-1}|^p \; 
			(\alpha_r \circ \varphi_r^{-1}) \; \sqrt{\det((D\varphi_r^{-1})^T D \varphi_r^{-1}) \;} \; \diff \Y'_r \right)^{\frac{1}{p}}
	\end{equation}
	the space $L^p(\Gamma)$ from Definition~\ref{def:bd_Lp} is a Banach space.
\end{lem}
\begin{proof}
	By \cite[Sec.~3.1, Lem.~1.3]{necas_2012} the definition of the norm in \eqref{eq:bd_Lp_norm} is independent of the atlas and of the partition of unity.
	The Banach property follows from \cite[Sec.~2.4, Thm.~4.1]{necas_2012}, \cite[Sec.~3.1, Lem.~1.1]{necas_2012} 
	and \cite[Sec.~3.1, Lem.~1.2]{necas_2012}.
\end{proof}
We next would like to write the derivative of functions in $W^{1,p}(\Gamma)$ more explicitly. 
So we recall some definitions and results for local coordinate systems of manifolds.
First of all the partial derivatives of $\varphi_r^{-1}$ form a basis for the tangent space, \cite[Rem.~10.5]{amann_escher_II}.
\begin{lem}
	Let $\bs p \in \Gamma$ and $\varphi$ a local chart at $\bs p$ and $\Y' \in \Delta$ such that $\varphi(\bs p) = \Y'$. 
	Then the set 
    \[
        \set{ \tau_i = \partial_i \varphi^{-1}_i(\Y')}{i = 1, \dots, d-1}
    \]
    forms a basis for the tangent space of $\Gamma$ in $\bs p$.
\end{lem}

Next, we define the first fundamental matrix $g$.
\begin{definition}\label{def:fund_matrix}
	Let $\bs p \in \Gamma$ and $\varphi$ a local chart at $\bs p$ and $\Y' \in \Delta$ such that $\varphi(\bs p) = \Y'$. 
	We then define the first fundamental matrix $g$ by
	\[
		g = (g_{ij})^{d-1}_{i,j=1} = \left( \tau_i \cdot \tau_j \right)_{i,j=1}^{d-1}.
	\]
\end{definition}

\begin{remark}
	The first fundamental matrix $g$ is positive definite, \cite[Rem.~10.3e)]{amann_escher_II}.
	Thus it is invertible. We will denote $g^{-1} = (g^{ij})_{i,j=1}^{d-1}$.
	So the entries of the inverse matrix are also called $g$ but with superscripts instead of subscripts.
	This might be confusing at first but as this seems to be the common notation
	we decided to use it as well.
	We will denote $\tau^i = \sum_{j=1}^{d-1} g^{ij} \tau_j$, 
	so again we are using superscripts instead of subscripts for the dual (``inverse'') basis.
\end{remark}

We now define the surface gradient of a scalar function on the boundary.
\begin{definition}\label{def:surf_grad_scal}
	Let $f \in W^{1,p}(\Gamma)$. We define the surface gradient $\nabla_\Gamma f : \Gamma \to \R^d$ 
	at $\bs p \in \Gamma$ with local chart $\varphi$ at $\bs p$ by
	\[
		\nabla_\Gamma f = \left( \sum_{i=1}^{d-1} \partial_i (f \circ \varphi^{-1}) \; \tau^{i} \right) \circ \varphi,
	\]
	where the derivative is to be understood in the weak sense.
\end{definition}
Similarly we define the gradient of a vector valued function on the boundary.
\begin{definition}\label{def:surf_grad_vec}
	Let $\bv \in W^{1,p}(\Gamma)^d$ we then define the surface gradient $\nabla_\Gamma \bv : \Gamma \to \R^{d \times d}$ 
	at $\bs p \in \Gamma$ with local chart $\varphi$ at $\bs p$ by
	\[
		\nabla_\Gamma \bv = \left( \sum_{i=1}^{d-1} \partial_i (\bv \circ \varphi^{-1}) \otimes \tau^{i} \right) \circ \varphi,
	\]
	where the derivative is to be understood in the weak sense.
\end{definition}

\begin{remark}\label{rem:surf_grad_tang}
	For these definitions to be meaningful they need to be independent of the chosen local chart.
	For the gradient of a scalar valued function, \textit{cf.~}Definition~\ref{def:surf_grad_scal}, 
	the independence of the local chart is proven in \cite[Prop.~B.3]{skrepek_2023}
    and the proof for the vector valued functions works analogously
    thus we omit it here.
	Additionally, it is easy to see that we have $\nabla_\Gamma f \cdot \bn = 0 $ and  $(\nabla_\Gamma \bv) \bn = 0$ on $\Gamma$. 
\end{remark}
Additionally the surface gradient fulfills the following product rules.
\begin{lem}\label{lem:surf_prod_rule}
	For $\Gamma \in C^{1,1}$, $f \in W^{1,p}(\Gamma)$, and $g \in W^{1,p'}(\Gamma)$ it holds 
	$\nabla_\Gamma (fg) = g \nabla_\Gamma f + f \nabla_\Gamma g$
	and for two tangential vector fields $\bv \in W^{1,p}(\Gamma)$ and $\bu \in W^{1,p'}(\Gamma)$ 
	it holds
	\begin{equation}\label{eq:surf_prod_rule}
		\nabla_\Gamma (\bv \cdot \bu) = (\nabla_\Gamma \bv)^T \bu + (\nabla_\Gamma \bu)^T \bv.
	\end{equation}
\end{lem}
\begin{proof}	
	The product rule for scalar functions is a simple consequence of the standard product rule.
	For the product rule for tangential vector fields we write
	$\bv = \sum^{d-1}_{i=1} v_i \tau_i$ and $\bu = \sum^{d-1}_{i=1} u_i \tau_i$.
	We first note
	\begin{align*}
		\partial_j (\bv \circ \varphi^{-1})
		= 
			\partial_j \left( \sum^{d-1}_{i=1} (v_i \circ \varphi^{-1}) \tau_i \right)
		=
			\sum^{d-1}_{i=1} \partial_j (v_i \circ \varphi^{-1}) \tau_i + (v_i \circ \varphi^{-1}) \partial_j \tau_i.
	\end{align*}
	With that we obtain
	\begin{multline*}
		\partial_j  ( (\bv \cdot \bu) \circ \varphi^{-1})
		=
			\sum^{d-1}_{i,l=1} \partial_j  (g_{il}(v_i u_l) \circ \varphi^{-1})\\
		=
			\sum^{d-1}_{i,l=1} \left( 
				g_{il} (u_l \circ \varphi^{-1}) \partial_j  (v_i \circ \varphi^{-1})
				+ g_{il} (v_i \circ \varphi^{-1}) \partial_j  (u_l \circ \varphi^{-1})
				+ u_l v_i \partial_j \tau_i \cdot \tau_l
				+ u_l v_i \partial_j \tau_l \cdot \tau_i
			\right)\\
		=
			\partial_j (\bv \circ \varphi^{-1}) \cdot \bu + \partial_j (\bu \circ \varphi^{-1}) \cdot \bv
	\end{multline*}
	and thus \eqref{eq:surf_prod_rule} follows.
\end{proof}

Now, that we have defined the surface gradient for scalar and vector valued functions living on the boundary
we would like to see how this definition relates to the (boundary) trace of the ``bulk'' gradient.
For that we come back to space dimension $d=3$ for the remainder of this section.
It turns out that the surface gradient coincides with the tangential projection of ``bulk'' gradient.
We will make this more explicit in the following paragraph relying on results for the weak tangential trace published in \cite{skrepek_2023}.
We start off by introducing the weak tangential trace, \textit{cf.~}\cite[Def.~4.1]{skrepek_2023}
\begin{definition}[Weak tangential trace]
		We say that $\bv \in W^{1,p}(\Omega)$ possess a weak tangential trace $\bs q$, if
		\[
			\bs{q} \in L_\tau^p(\Gamma) := \set{\bs{g} \in L^p(\Gamma)}{\bs{g} \cdot \bn = 0 \text{ a.e. in } \Gamma},
		\] 
		(where the almost everywhere is to be understood with respect to the measure on $\Gamma$)
		and it holds
		\begin{equation}\label{eq:tang_trace}
			\int_\Omega \bv \cdot (\nabla \times \bs \Phi) - (\nabla \times \bv) \cdot \bs \Phi \diff \X = \int_\Gamma \bs{q} \cdot (\bn \times \bs \Phi) \diff \sigma
		\end{equation}
		for all smooth test function $\bs \Phi \in C_0^\infty(\R^d)$
		and we write $\pi_\tau(\bv)= \bs{q}$.
\end{definition}

\begin{remark}
	The tangential trace operator $\pi_\tau$ is a well-defined operator from $W^{1,p}(\Omega)$ to $L^p(\Gamma)$.
	Defining $\bs{q} = S(\bv) - (S(\bv) \cdot \bn) \bn \in L^p(\Gamma)$ since for $\Gamma \in C^{0,1}$ we have $\bn \in L^\infty(\Gamma)$,
	we obtain the existence of a tangential trace for all $\bv \in W^{1,p}(\Omega)$ 
	and the uniqueness follows from the following observation.
	Assuming there are two tangential traces $\bs{q}$ and $\tilde{\bs{q}}$ we find
	\begin{align*}
		\int_\Gamma (\bs{q} - \tilde{\bs{q}}) \cdot \tilde{\bs \Phi} \dS 
		=
			\int_\Gamma (\bs{q} - \tilde{\bs{q}}) \cdot \left( \bn \times (\tilde{ \bs \Phi} \times \bn) - (\bn \cdot \tilde{ \bs \Phi}) \bn \right) \dS 
		=
			\int_\Gamma (\bs{q} - \tilde{\bs{q}}) \cdot  (\bn \times (\tilde{\bs \Phi} \times \bn))\dS.
	\end{align*}
	Since $\bn \in L^\infty(\Gamma)$, we have $(\tilde{\bs \Phi} \times \bn) \in L^{p'}(\Gamma)$ for $\tilde{\bs \Phi} \in L^{p'}(\Gamma)$,
	the right-hand side is zero by the definition of the tangential trace with $\bs \Phi = (\tilde{\bs \Phi} \times \bn)$ 
	using the density of $\{\bs \Phi|_{\Gamma} | \bs \Phi \in C_0^\infty(\R^d)\}$ in $L^{p'}(\Gamma)$, \cite[Prop.~2.8]{lions_magenes_1961}.
\end{remark}

Via a detour to this weak tangential trace and the density of smooth functions in $W^{1,p}(\Gamma)$ one gets the following result.
\begin{thm}\label{thm:surf_grad_vs_tan_pro_scalar}
	Let $f \in W^{2,p}(\Omega)$. 
	Then the surface gradient is just the tangential projection of the bulk gradient, that is
	\[
		\nabla_\Gamma S(f) = S(\nabla f) - (S(\nabla f) \cdot \bn) \bn
	\]
	holds almost everywhere.
\end{thm}
\begin{proof}
	Since we have $f \in W^{2,p}(\Omega)$ the trace $S(\nabla f) \in L^p(\Gamma)$
	and thus $S(\nabla f) - (S(\nabla f) \cdot \bn) \bn \in L^p_\tau(\Gamma)$.
	Using a $\mathrm{curl}-$integration by parts rule, 
	which is a simple consequence of Green's Theorem, see \cite[Sec.~3.1, Thm.~1.1]{necas_2012}, 
	and the Lagrange identity for the vector cross product
	\begin{equation}\label{eq:lagr_id}
		(\bs a \times \bs b) \cdot (\bs c \times \bs d) = (\bs a \cdot \bs c) (\bs b \cdot \bs d) - (\bs b \cdot \bs c) (\bs a \cdot \bs d)
	\end{equation}
	we find
	\begin{align*}
		\int_\Omega (\nabla f) &\cdot (\nabla \times \bs \Phi) - (\nabla \times (\nabla f)) \cdot \bs \Phi \diff \X\\
		\underset{\mathrm{curl-I.B.P.}}&{=}
			\int_\Gamma \bs \Phi \cdot \left( S(\nabla f) \times \bn \right) \diff \sigma
		=
			- \int_\Gamma \bs \Phi \cdot \left( \bn \times S(\nabla f) \right) \underbrace{(\bn \cdot \bn)}_{= 1} \diff \sigma\\
		&=
			\int_\Gamma
				 - \bs \Phi \cdot \left( \bn \times S(\nabla f) \right) (\bn \cdot \bn) 
				+ \underbrace{\bn \cdot (\bn \times S(\nabla f))}_{= 0} (\bs \Phi \cdot \bn)
			\diff \sigma\\
		\underset{\eqref{eq:lagr_id}}&{=}
			\int_\Gamma
				(\bn \times \bs \Phi) \cdot \left( (\bn \times S(\nabla f)) \times \bn \right)
			\diff \sigma
		= 
			\int_\Gamma
				(\bn \times \bs \Phi) \cdot \left(S(\nabla f) - ( S(\nabla f) \cdot \bn ) \bn \right)
			\diff \sigma
	\end{align*}
	for all $\bs \Phi \in C^\infty_0(\R^d)$.
	Thus we have $\pi_\tau(\nabla f) = \left(S(\nabla f) - ( S(\nabla f) \cdot \bn ) \bn \right) $
	and using \cite[Thm.~4.2]{skrepek_2023} 
	we find 
	\[
		\left(S(\nabla f) - ( S(\nabla f) \cdot \bn ) \bn \right) = \pi_\tau(\nabla f) = \nabla_\Gamma (S(f)),
	\]
	which finishes the proof.
\end{proof}

Next, we would like a similar result to Theorem~\ref{thm:surf_grad_vs_tan_pro_scalar} for vector valued functions. 
This is basically just using the result for scalar valued functions and writing the vector field in the ``right''
coordinate system, which unsurprisingly is the Euclidean system.
Nonetheless, we thought it important to do this short proof here.
\begin{thm}\label{thm:surf_grad_vs_tan_pro_vec}
	Let $\bv \in W^{2,p}(\Omega)$. We then have
	\[
		\nabla_\Gamma S(\bv) = S(\nabla \bv) - S(\nabla \bv) \bn \otimes \bn.
	\]
\end{thm} 
\begin{proof}
	Writing $\bv$ in Euclidean coordinates we find $v_i \in W^{2,p}(\Omega)$ for $i = 1, \dots, d$, such that
	$\bv = \sum^d_{i=1} v_i \bs{e}_i$, where the $ \bs{e}_i$ denote the standard basis vectors.
	We find $S(\bv) = \sum^d_{i=1} S(v_i) \bs{e}_i$ by the linearity of the trace operator $S$. 
	The claim follows by the definition of the surface gradient for vector valued functions, \textit{cf.~}Definition~\ref{def:surf_grad_vec},
	and Theorem~\ref{thm:surf_grad_vs_tan_pro_scalar}.
\end{proof}

Additionally, we define the divergence of a vector valued function on the boundary.
\begin{definition}\label{def:surf_div}
	Let $\bv \in W^{1,p}(\Gamma)$ we then define the surface divergence $\nabla_\Gamma \cdot \bv : \Gamma \to \R$ by
	\[
		\nabla_\Gamma \cdot \bv = \sum_{i=1}^{d-1} \nabla_\Gamma \bv (\tau^i \circ \varphi) \cdot (\tau_i \circ \varphi).
	\]
\end{definition}

One can again show that this definition is independent of the chosen chart,
which we will refrain from in order to keep things short.
We can again relate the surface divergence to the divergence in the bulk.
\begin{thm}\label{thm:surf_div_form}
	Let $\bv \in W^{2,p}(\Omega)$. We then have
	\begin{equation}\label{eq:surf_div_form}
		\nabla_\Gamma \cdot S(\bv) 
			= 
				S(\nabla \cdot \bv) - S(\nabla \bv) \bn \cdot \bn.
	\end{equation}
\end{thm}
\begin{proof}
	We again take the Euclidean representation of $\bv$, that is $v_i \in W^{2,p}(\Omega)$ for $i=1, \dots, d$ such that
	$\bv = \sum^d_{i=1} v_i \bs{e}_i$, where the $ \bs{e}_i$ denote the standard basis vectors.
	Then by the definition of the surface divergence, Theorem~\ref{thm:surf_grad_vs_tan_pro_vec}
	and the representations,
	\begin{equation}\label{eq:basis_rep}
		\bs{e}_i 
		= 
			(\bs{e}_i \cdot \bn) \bn + \sum_{j=1}^{d-1} (\bs{e}_i \cdot \tau^j) \tau^j 
		= 
			(\bs{e}_i \cdot \bn) \bn + \sum_{l=1}^{d-1} (\bs{e}_i \cdot \tau_l) \tau_l 
	\end{equation}
	we have
	\begin{align}\label{eq:surf_div_exp}
	\begin{split}
		S(\nabla \cdot \bv)
		&= 
			\sum_{i = 1}^d S(\nabla \bv) \bs{e}_i \cdot \bs{e}_i\\
		&=
			\left(\sum_{i = 1}^d \left( S(\nabla \bv) - S(\nabla \bv) (\bn \otimes \bn) \right) \bs{e}_i \cdot \bs{e}_i\right)
			+ \left( \sum_{i = 1}^d S(\nabla \bv) (\bn \otimes \bn) \bs{e}_i \cdot \bs{e}_i \right).
	\end{split}
	\end{align}
	The second term on the right-hand side is already the one we need, since we have
	\[
		\left( \sum_{i = 1}^d S(\nabla \bv) (\bn \otimes \bn) \bs{e}_i \cdot \bs{e}_i \right)
		=
			\left( \sum_{i = 1}^d (S(\nabla \bv) \bn)_i \bn_i \right)
		=
			(S(\nabla \bv) \bn) \cdot \bn.
	\]
	The first term on the right-hand side of \eqref{eq:surf_div_exp} becomes
	\begin{align*}
		\sum_{i = 1}^d &\left( S(\nabla \bv) - S(\nabla \bv) (\bn \otimes \bn) \right) \bs{e}_i \cdot \bs{e}_i\\
		&=
			\sum_{i = 1}^d \left( S(\nabla \bv) - S(\nabla \bv) (\bn \otimes \bn) \right) 
			\left(\sum_{j=1}^{d-1} (\bs{e}_i \cdot (\tau^j \circ \varphi)) (\tau^j  \circ \varphi) \right)
			\cdot \left( \sum_{l=1}^{d-1} (\bs{e}_i \cdot (\tau_l  \circ \varphi)) (\tau_l  \circ \varphi) \right)\\
		&=
			\sum_{j,l = 1}^{d-1} \left( S(\nabla \bv) - S(\nabla \bv) (\bn \otimes \bn) \right) (\tau^j  \circ \varphi) \cdot (\tau_l  \circ \varphi)
			\left( \sum_{i=1}^d (\bs{e}_i \cdot (\tau^j  \circ \varphi)) (\bs{e}_i \cdot (\tau_l \circ \varphi)) \right)\\
		&=
			\sum_{j,l = 1}^{d-1} \left( S(\nabla \bv) - S(\nabla \bv) (\bn \otimes \bn) \right) (\tau^j  \circ \varphi) \cdot (\tau_l  \circ \varphi) ((\tau^j  \circ \varphi) \cdot (\tau_l  \circ \varphi))\\
		&=
			\sum_{j,l = 1}^{d-1} \left( S(\nabla \bv) - S(\nabla \bv) (\bn \otimes \bn) \right) (\tau^j  \circ \varphi) \cdot (\tau_l  \circ \varphi) \delta_{lj}\\
		&=
			\sum_{l = 1}^{d-1} \nabla_\Gamma S(\bv) (\tau^l  \circ \varphi) \cdot (\tau_l  \circ \varphi)
		=
			\nabla_\Gamma \cdot S(\bv),
	\end{align*}
	where we used that the normal parts of the basis vectors $\bs{e}_i$ vanish due to the fact that the surface gradient is tangential.
	Thus we get \eqref{eq:surf_div_form}.
\end{proof}
Our next goal is to deduce an integration by parts rule on the boundary.
To that aim we first prove the following product rule.
\begin{lem}\label{lem:bd_product_rule}
	For $\Gamma \in C^{1,1},$ $f \in W^{1,p}(\Gamma)$ and $\bv \in W^{1,p'}(\Gamma)$ it holds
	\[
		\nabla_\Gamma (f \bv) = f \nabla_\Gamma \bv + \bv \otimes \nabla_\Gamma f \quad \text{ and } \quad 
		\nabla_\Gamma \cdot (f \bv) = f \; \nabla_\Gamma \cdot \bv + \bv \cdot \nabla_\Gamma f.
	\]
\end{lem}
\begin{proof}
	We first calculate the surface gradient of $f \bv$,
	\begin{align*}
		\nabla_\Gamma (f \bv) 
		&= 
			\left( \sum_{i=1}^{d-1} \partial_i \left( (f \circ \varphi^{-1}) (\bv \circ \varphi^{-1} ) \right) \otimes \tau^i \right) \circ \varphi\\
		&=
			\left( \sum_{i=1}^{d-1} \left( \partial_i (f \circ \varphi^{-1}) (\bv \circ \varphi^{-1} ) 
				+ (f \circ \varphi^{-1}) \partial_i (\bv \circ \varphi^{-1} )  \right)
			\otimes \tau^i \right) \circ \varphi 
		=
			\bv \otimes \nabla_\Gamma f + f \nabla_\Gamma \bv,
	\end{align*}
	where we used the product rule for weak derivatives in the bulk.
	Next, by the definition of the surface divergence, we have
	\begin{align*}
		\nabla_\Gamma \cdot (f \bv) 
		&= 
			\sum_{i=1}^{d-1} \nabla_\Gamma (f \bv) (\tau^i \circ \varphi) \cdot (\tau_i \circ \varphi)
		= 
			\sum_{i=1}^{d-1} (\bv \otimes \nabla_\Gamma f + f \nabla_\Gamma \bv) (\tau^i \circ \varphi) \cdot (\tau_i \circ \varphi)\\
		&= 
			\sum_{i=1}^{d-1} (\nabla_\Gamma f \cdot (\tau^i \circ \varphi)) (\tau_i \circ \varphi) \cdot \bv + f \nabla_\Gamma \cdot \bv\\
		&=
			\sum_{i,j=1}^{d-1} \left( \partial_j (f \circ \varphi^{-1}) (\tau^j \cdot \tau^i) \tau_i \right) \circ \varphi \cdot \bv 
			+ f \nabla_\Gamma \cdot \bv\\
		&=
			\sum_{i,j=1}^{d-1} \left( g^{ji} \partial_j (f \circ \varphi^{-1}) \tau_i \right) \circ \varphi	 \cdot \bv + f \nabla_\Gamma \cdot \bv
		=
			\sum_{j=1}^{d-1} \left( \partial_j (f \circ \varphi^{-1}) \tau^j \right) \circ \varphi \cdot \bv + f \nabla_\Gamma \cdot \bv\\
		&=
			\nabla_\Gamma f \cdot \bv + f \nabla_\Gamma \cdot \bv.
	\end{align*}
\end{proof}
Next, we turn to the integration by parts rule.
\begin{thm}\label{thm:bd_div}
	Let $\Gamma \in C^{2,1}$ and $\bv$ be a tangential vector field in $W^{1,p}(\Gamma)$,
	that is we can write $\bv$ as 
	\[
		\bv = \sum_{i=1}^{d-1} v_i  \left(\tau_i \circ \varphi \right) 
	\]
	with $v_i \in W^{1,p}(\Gamma)$.
	It then holds 
	\begin{equation}\label{eq:bd_gauss_thm}
		\int_\Gamma \nabla_\Gamma \cdot \bv \diff \sigma = 0.
	\end{equation}
\end{thm}

\begin{remark}
	In particular the above theorem does \underline{not} hold for non tangential vector fields, 
	which can be seen by explicitly calculating the left-hand side for the domain $\Omega = B_1(0) \subseteq \R^2$ 
	and $\bv = \bn = (x,y)^T$ on $\Gamma$.
\end{remark}

\begin{proof}
	The proof follows \cite[Sec.~2.1.6]{pruss_simonett_2016}.
	We first write the definition of the surface divergence in local coordinates
	to derive the following formula of the surface divergence in local coordinates, 
	\[
		\nabla_\Gamma \cdot \bv
		=
			\left( \frac{1}{\sqrt{\det(g)}} \sum_{i=1}^{d-1} \partial_i \left( \sqrt{\det(g)} (v_i \circ \varphi^{-1}) \right) \right) \circ \varphi.
	\]
    Then from the definition of the surface integral,
	using an atlas for $\Gamma$ and a subordinate partition of unity $(\alpha_r)_{r=1,\dots,M}$ we can use classical integration by parts in the bulk 
	to finish the proof.
	We start by writing the surface divergence in local coordinates.
	First of all we notice that
	\begin{align*}
		\partial_i (\bv \circ \varphi^{-1})
		=
			\partial_i \left( \sum_{j=1}^{d-1} (v_j \circ \varphi^{-1}) \tau_j \right)
		=
			\sum_{j=1}^{d-1}\left( \partial_i (v_j \circ \varphi^{-1}) \tau_j + (v_j \circ \varphi^{-1}) \partial_i \tau_j \right).
	\end{align*}
	With that, we can write the surface divergence as
	\begin{align*}
		\nabla_\Gamma \cdot \bv 
		&=
			\sum_{m = 1}^{d-1} \nabla_\Gamma \bv (\tau^m \circ \varphi) \cdot (\tau_m \circ \varphi)
		= 
			\sum_{m = 1}^{d-1} \left(
				\left( \sum_{i = 1}^{d-1} \partial_i (\bv \circ \varphi^{-1}) \otimes \tau^{i} \right) \; 
				\tau^m \cdot \tau_m
			\right) \circ \varphi \\
		&=
			\sum_{m, i = 1}^{d-1} \left(
				g^{i m} \partial_i (\bv \circ \varphi^{-1}) \cdot \tau_m
			\right) \circ \varphi 
		=
			\sum_{i = 1}^{d-1} \left(
				\partial_i (\bv \circ \varphi^{-1}) \cdot \tau^i
			\right) \circ \varphi\\
		&= 
			\sum_{i,j = 1}^{d-1} \left(
				\left(
					\partial_i (v_j \circ \varphi^{-1}) \delta_{i j}
					+ (v_j \circ \varphi^{-1}) (\partial_i \tau_j \cdot \tau^i)
				\right)
			\right) \circ \varphi\\
		&=
			\left( \left( \sum_{i=1}^{d-1} 
				\partial_i (v_i \circ \varphi^{-1})
			\right)
			+ \left( \sum_{i,j=1}^{d-1}
				(v_j \circ \varphi^{-1})  (\partial_i \tau_j \cdot \tau^i)
			\right) \right) \circ \varphi.
	\end{align*}
	Using $\partial_i \tau_j = \partial_i \partial_j \varphi^{-1} = \partial_j \partial_i \varphi^{-1} = \partial_j \tau_i$ 
	which follows from the fact that partial derivatives commute if they are continuous
	(here we use that $\Gamma \in C^{2,1}$ and thus $\varphi^{-1} \in C^{2,1}$),
	we can further rewrite the second term on the right-hand side as
	\begin{align*}
		\left( \sum_{i,j=1}^{d-1}
			(v_j \circ \varphi^{-1}) (\partial_i \tau_j \cdot \tau^i)
		\right) \circ \varphi
		&=
			\sum_{j=1}^{d-1} \left( 
				(v_j \circ \varphi^{-1}) \sum_{i=1}^{d-1}  (\partial_j \tau_i \cdot \tau^i)
			\right) \circ \varphi\\
		&=
			\frac{1}{2}  \left( \sum_{j=1}^{d-1}
				(v_j \circ \varphi^{-1}) \left(
					\sum_{k=1}^{d-1}  (\partial_j \tau_k \cdot \tau^k)
					+ \sum_{i=1}^{d-1}  (\partial_j \tau_i \cdot \tau^i)
				\right)
			\right) \circ \varphi\\
		&=
			\frac{1}{2} \left( \sum_{j=1}^{d-1} 
				(v_j \circ \varphi^{-1}) \left(
					\sum_{i,k=1}^{d-1} g^{k i} \left( \partial_j \tau_k \cdot \tau_i \right)
					+ \sum_{i,k =1}^{d-1} g^{i k}\left( \partial_j \tau_i \cdot \tau_k \right) 
				\right)
			\right) \circ \varphi\\
		\underset{g^{-1} \mathrm{\, sym.}}&{=}
			\frac{1}{2}  \left( \sum_{j=1}^{d-1} \left(
				(v_j \circ \varphi^{-1}) 
				\sum_{i, k=1}^{d-1} \left( g^{i k} \left( (\partial_j \tau_k \cdot \tau_i) + (\partial_j \tau_i \cdot \tau_k) \right) \right) 
			\right) \right) \circ \varphi\\
		&=
			\frac{1}{2} \left( \sum_{j=1}^{d-1} \left(
				 (v_j \circ \varphi^{-1}) \sum_{i, k=1}^{d-1} \left( g^{i k} \partial_j g_{k i} \right) 
			\right) \right) \circ \varphi \\
		&=
			\frac{1}{2} \left( \sum_{j=1}^{d-1}
				(v_j \circ \varphi^{-1})\mathrm{tr}(g^{-1} \partial_j g) 
			\right) \circ \varphi\\
		\underset{\mathrm{Jacobi's \; form.}}&{=}
			\frac{1}{2 \det(g)} \left( \sum_{j=1}^{d-1}
				 (v_j \circ \varphi^{-1}) \partial_j (\det(g)) 
			\right) \circ \varphi.
	\end{align*}
	Putting these two expansions together we obtain
	\begin{align*}
	\nabla_\Gamma \cdot \bv
		&=
			\left( \frac{1}{\sqrt{\det(g)}} \sum_{i=1}^{d-1} \left(
				\sqrt{\det(g)} \partial_i (v_i \circ \varphi^{-1}) + \frac{1}{2 \sqrt{\det(g)}} (v_i \circ \varphi^{-1}) \partial_i (\det(g)) 
			\right) \right) \circ \varphi\\
		&=
			\left( \frac{1}{\sqrt{\det(g)}} \sum_{i=1}^{d-1} \partial_i \left( \sqrt{\det(g)} (v_i \circ \varphi^{-1}) \right) \right) \circ \varphi.
	\end{align*}
	This finishes our prove of the surface divergence formula in local coordinates and we now turn to the proof of the integral identity \eqref{eq:bd_gauss_thm}.
	Choosing a partition of unity $(\alpha_r)_{r=1, \dots, M}$ subordinate to our atlas domains.
	That is we choose $(\alpha_r)_{r=1, \dots, M}$ subordinate to the open cover
	\[
		\left( 
			T_r^{-1} \left( \Delta_r \times \set{\Y_{rd} \in \R}{\exists \, \Y'_r \in \Delta_r: |\Y_{rd} - a_r(\Y'_r)| < \beta} \right) 
		\right)_{r=1, \dots, M}
		=: 
			\left( 
			T_r^{-1} (O_r)
		\right)_{r=1, \dots, M}
	\] 
	of $\Gamma$.
	We then have by definition of the surface integral, \textit{cf.~}\eqref{eq:bd_Lp_norm}, 
	recalling that $\det(g) = \det((D\varphi^{-1})^T D\varphi^{-1})$, \textit{cf.~}Definition~\ref{def:fund_matrix},
	\begin{align*}
		\int_\Gamma \nabla_\Gamma \cdot \bv \diff \sigma
		&=
			\sum_{r = 1}^M \int_\Gamma \alpha_r \nabla_\Gamma \cdot \bv \diff \sigma\\
		&=
			\sum_{r = 1}^M \int_{\Delta_r}
				\frac{1}{\sqrt{\det(g)}} \sum_{i=1}^{d-1} \partial_i \left( \sqrt{\det(g)} (v_i \circ \varphi^{-1}) \right)
				(\alpha_r \circ \varphi^{-1})
				\sqrt{\det(g)}
			\diff \X'_r\\
		&= 
			- \sum_{r = 1}^M \sum_{i=1}^{d-1} \int_{\Delta_r}
				\sqrt{\det(g)} (v_i \circ \varphi^{-1})
				\partial_i (\alpha_r \circ \varphi^{-1})
			\diff \X'_r\\
		&=
			- \sum_{i=1}^{d-1} \int_{\Delta_r}
				\sqrt{\det(g)} (v_i \circ \varphi^{-1})
				\partial_i (1)
			\diff \X'_r	
		= 
		 	0,
	\end{align*}
	where we used the standard integration by parts rule and the fact that the partition of unity has compact support in $\Delta_r$
	so that no boundary terms appear and that the $\alpha_r$ sum up to $1$.
	To see that $\alpha_r \circ \varphi^{-1}$ is indeed zero on the boundary of $\Delta_r$ we use that for a rigid motion $T$ we have
	$\partial (T(\Omega)) = T(\partial \Omega)$ for any set $\Omega \subseteq \R^d$.
	Thus for $\Y' \in \partial \Delta_r$ we have $(\Y', a_r(\Y')) \in \partial O_r$ 
	and thus $\varphi^{-1}(\Y') = T_r^{-1}(\Y', a_r(\Y')) \in \partial (T_r^{-1}(O_r))$.
\end{proof}

\begin{comment}
\begin{proof}
	First of all we note that $\Gamma$ is a manifold with empty boundary. \marginnote{Prove!}
	Secondly by Stokes theorem, \cite{fischer_2017},
	we have
	\[
		\int_\Gamma \diff \omega = 0
	\]
	for all continuous differential forms $\omega$.
	So we need to show that we can write the surface divergence in a suitable differential form.
	To that aim we write the surface divergence in local coordinates
	and to that aim we first write the surface gradient of $\bv$ in local coordinates,
	\begin{multline*}
		\nabla_\Gamma \cdot \bv 
		=
			\dots\\
		= 
			\sum_{i=1}^{d-1} \left( \partial_i (v_i \circ \varphi^{-1}) + \sum_{j=1}^{d-1} v_j \Lambda^i_{ij} \right)
		=
			\sum_{i=1}^{d-1} \left( \frac{1}{\sqrt{\det(g)}} \partial_i (\sqrt{\det(g)} v_i) \right).
	\end{multline*}
	Using \cite[Rem.~6.7]{amann_escher_III} we have
	\[
		\sum_{i=1}^{d-1} \left( \frac{1}{\sqrt{\det(g)}} \partial_i (\sqrt{\det(g)} v_i) \right) \diff \omega_\Gamma = \diff \omega
	\]
	and thus by the definition of the volume element $\omega_\Gamma$
	\begin{multline*}
		\int_\Gamma \nabla_\Gamma \cdot \bv \diff \sigma
		=
			\int_\Gamma \sum_{i=1}^{d-1} \left( \frac{1}{\sqrt{\det(g)}} \partial_i (\sqrt{\det(g)} v_i) \right) \diff \sigma
		=
			\int_\Gamma \sum_{i=1}^{d-1} \left( \frac{1}{\sqrt{\det(g)}} \partial_i (\sqrt{\det(g)} v_i) \right) \diff \omega_\Gamma\\
		=
			\int_\Gamma \diff \omega
		= 
			0.	
	\end{multline*}
\end{proof}
\end{comment}

\begin{cor}\label{cor:ibp_boundary}
	For $\Gamma \in C^{2,1}$ and a general (not necessarily tangential) vector field $\bv \in W^{1,p}(\Gamma)$ and $f \in W^{1,p'}(\Gamma)$
	the following integration by parts holds
	\[
		\int_\Gamma f \; \nabla_\Gamma \cdot \bv - f (\bv \cdot \bn) \nabla_\Gamma \cdot \bn + \bv \cdot \nabla_\Gamma f \diff \sigma  = 0.
	\]
\end{cor}
The proof is taken from \cite[Sec.~2.1]{pruss_simonett_2016}.
\begin{proof}
	This is a simple consequence of Theorem~\ref{thm:bd_div}.
	For a possible non tangential vector field $\bv$ we can write $\bv = (\bv - (\bv \cdot \bn) \bn) + (\bv \cdot \bn) \bn =: \bv_\tau + (\bv \cdot \bn) \bn$.
	The first part is a tangential vector field and thus by Theorem~\ref{thm:bd_div} and the product rule from Lemma~\ref{lem:bd_product_rule}
	we have
	\begin{multline*}
		\int_\Gamma \nabla_\Gamma f \cdot \bv \diff \sigma
		=
			\int_\Gamma  \nabla_\Gamma f \cdot \bv_\tau \diff \sigma
		=
			- \int_\Gamma 
				f \; \nabla_\Gamma \cdot \bv_\tau
			\diff \sigma
		= 
			- \int_\Gamma 
				f \; \nabla_\Gamma \cdot \bv
				- f \; \nabla_\Gamma  \cdot \left( (\bv \cdot \bn) \bn \right)
			\diff \sigma\\
		=
			- \int_\Gamma 
				f \; \nabla_\Gamma \cdot \bv
				- f \; \nabla_\Gamma (\bv \cdot \bn) \cdot \bn
				- f \; (\bv \cdot \bn) \nabla_\Gamma \cdot \bn
			\diff \sigma
		=
			- \int_\Gamma 
				f \; \nabla_\Gamma \cdot \bv
				- f \; (\bv \cdot \bn) \nabla_\Gamma \cdot \bn
			\diff \sigma,
	\end{multline*}
	where we used that the surface gradient is tangential, such that $\nabla_\Gamma (\bv \cdot \bn) \cdot \bn = 0$, see \cite{skrepek_2023}.
\end{proof}

%\begin{lem}\label{lem:sec_surf_grad}
%	Let $\Gamma \in C^{1,1}$, then second surface gradient of a scalar function $f \in W^{2,2}(\Gamma)$ is well-defined.
%\end{lem}
    
\begin{subsubsection}{Matrix-valued functions on the boundary}

All results from the previous section generalize to matrix-valued functions on the boundary.
We first define.
\begin{definition}
	Let $\bs A: \Gamma \to \R^{d \times d}$ be defined almost everywhere. Then $\bs A$ is in $W^{1,p}(\Gamma)$
	if for an atlas $\set{(\varphi_r, V_r)}{r= 1, \dots, M}$ of $\Gamma$ we have $\bs A \circ \varphi_r^{-1} \in W^{1,p}(\Delta_r)$
	for all $r = 1, \dots, M$.
\end{definition}
As above we define the surface gradient.
\begin{definition}
	For $\bs A \in W^{1,p}(\Gamma)$ we define
	\[
		\nabla_\Gamma \bs A = \left( \sum_{i=1}^{d-1} \partial_i (\bs A \circ \varphi^{-1}) \otimes \tau^i \right) \circ \varphi \in \R^{d \times d \times d}.
	\]
\end{definition}
Using this definition we also define the surface divergence.
\begin{definition}\label{def:surf_div_matrix}
	For $\bs A \in W^{1,p}(\Gamma)$ we define
	\[
		\nabla_\Gamma \cdot \bs A = \sum_{m=1}^{d-1} (\nabla_\Gamma \bs A \cdot (\tau^m \circ \varphi))  (\tau_m \circ \varphi) \in \R^d.
	\]
\end{definition}
These definitions are all independent of the chosen atlas, which follows as before with the help of coordinate transformations.
Additionally, we get the following product rule.
\begin{lem}\label{lem:surf_prod_rule_matrix}
	For $\bs A \in W^{1,p}(\Gamma)$ and $\bv \in W^{1,p}(\Gamma)$ it holds
	\(
		\nabla_\Gamma \cdot (\bs A \bv) = (\nabla_\Gamma \cdot \bs A^T) \cdot \bv + \bs A^T : \nabla_\Gamma \bv.
	\)
\end{lem}
\begin{proof}
	Writing $\bs A=(a_{ij})_{i,j=1}^d$ and $\bv$ in Euclidean coordinates, that is
	$\bs A = \sum_{i,j=1}^d a_{ij} \bs{e}_i \otimes \bs{e}_j$  and $\bv = \sum_{i=1}^d v_i \bs{e}_i$
	the identity simply follows from the standard product rule and calculating both side explicitly.
\end{proof}

\begin{cor}
	For $\Gamma \in C^{2,1}$ and $\bs A \in W^{1,p}(\Gamma)$ and $\bv \in W^{1,p}(\Gamma)$ it holds
	\begin{equation}\label{eq:bd_ibp_matrix}
		\int_\Gamma 
			(\nabla_\Gamma \cdot (\bs A^T)) \cdot \bv 
			+ \bs A^T : \nabla_\Gamma \bv  
			- (\bs A \bv \cdot \bn) \, \nabla_\Gamma \cdot \bn = 0.
	\end{equation}
\end{cor}
\begin{proof}
	This is a very simple consequence of the product rule from Lemma~\ref{lem:surf_prod_rule_matrix} and the integration by parts rule from
	Corollary~\ref{cor:ibp_boundary}.
	By Corollary~\ref{cor:ibp_boundary} we find
	\[
		\int_\Gamma \nabla_\Gamma \cdot (\bs A \bv) - (\bs A \bv \cdot \bn ) \nabla_\Gamma \cdot \bn \dS = 0
	\]
	and the first term on the left-hand side can be rewritten as
	\(
		\nabla_\Gamma (\bs A \bv) = (\nabla_\Gamma \cdot \bs A^T) \cdot \bv + \bs A^T : \nabla_\Gamma \bv
	\)
	and we are done.
\end{proof}
Now, we still would like to relate the surface divergence of a matrix-valued function on the boundary to the divergence of the bulk divergence.
\begin{cor}\label{cor:surf_div_form_matrix}
	For $\bs A \in W^{2,p}(\Omega)$ we have
	\[
		\nabla_\Gamma \cdot S(\bs A) = \nabla \cdot \bs A - (\nabla \bs A \cdot \bn) \bn.
	\]
\end{cor}
\begin{proof}
	This simply follows from writing $\bs A$ in Euclidean coordinates and using Theorem~\ref{thm:surf_div_form}.
\end{proof}

\end{subsubsection}
    \end{subsection}
    
\begin{subsection}{Auxiliary results for the flow model}

\begin{lem}\label{lem:reg_en_terms}
	For $\psi \in W^{4,10/3}(\Omega), \xi \in W^{3,10/3}(\Gamma)$ and $\bd \in W^{4,\infty}(\Omega)$ 
	fulfilling $\varMa \nabla \psi \cdot \bn + \tau \psi = \xi$ and $\bd \cdot \bn = 0$ on $\Gamma$, it holds
	\begin{align}\label{eq:reg_en_terms}
    \begin{split}
		(\varepsilon + \lambda) \int_\Omega
	 		\nabla (\Delta \psi) \cdot \varMa \nabla (\nabla \cdot ((\bd \cdot &\nabla \psi) \bd))
	 	\diff \X\\
	 	&=
	 		(\varepsilon + \lambda) \int_\Omega
	 			|\nabla^3 \psi \cdot \bd |^2 + \varepsilon |(\nabla^3 \psi \cdot \bd) \bd|^2
	 		\diff \X
	 		+ \mathrm{l.o.t}_{\kappa}
    \end{split}
	\end{align}
	for 
	\begin{align}\label{eq:def_kappa_lot}
    \begin{split}
		&\frac{1}{(\varepsilon + \lambda)} \mathrm{l.o.t}_{\kappa} 
		=
			\frac{1}{(\varepsilon + \lambda)} \left( 
				\int_\Omega \mathrm{l.o.t}_{\Omega \kappa} \diff \X
				+ \int_\Gamma \mathrm{l.o.t}_{\Gamma \kappa} \dS
			\right)\\
		&\hspace{1em}:=
			\int_\Omega
	 			\nabla (\Delta \psi) \cdot \varMa 
	 			\, \nabla \cdot \left( (\bd \cdot \nabla \psi) \nabla \bd^T + (\nabla \bd^T \nabla \psi) \otimes \bd \right)
	 		\diff \X\\
	 		&\hspace{2em}+ \int_\Omega
	 			\nabla^3 \psi \threedotsbin \left( 
	 			(\nabla^2 \psi \nabla \bd) \otimes_M \bd
				+ (\nabla^2 \psi \bd)  \otimes \nabla \bd
				\right)^T
	 		\diff \X\\
	 		&\hspace{2em}- \varepsilon \int_\Omega
	 			\nabla \cdot (\nabla^2 \psi \bd) (\nabla \bd : ((\nabla^2 \psi \bd) \otimes \bd))
	 			+ (\nabla \bd :  \nabla^2 \psi) (\nabla \cdot ((\bd \otimes (\nabla^2 \psi \bd)) \bd))
	 		\diff \X\\
	 		&\hspace{2em}+  \varepsilon \int_\Omega
	 			(\nabla \bd :  \nabla^2 \psi) (\nabla \bd : ((\nabla^2 \psi \bd) \otimes \bd))
	 		\diff \X
	 		+ \varepsilon \int_\Omega
		 		\nabla^2 \psi \nabla \bd  : (\nabla^3 \psi \cdot \bd) \bd \otimes \bd
	 		\diff \X\\
	 		&\hspace{2em}+ \varepsilon \int_\Omega
	 			( \nabla^3 \psi \cdot \bd + \nabla^2 \psi \nabla \bd) :
				\left( 
					(\nabla^2 \psi \bd \cdot \bd) \nabla \bd^T
		 			+ (\nabla \bd^T \nabla^2 \psi \bd) \otimes \bd 
		 			+ (\nabla \bd^T \nabla^2 \psi \bd) \otimes \bd
		 		\right)
		 	\diff \X\\
	 		&\hspace{2em}- \int_\Gamma
	 			\left[ (\nabla^2 \psi \bd) \otimes \nabla (\bd \cdot \bn) - ((\nabla^2 \psi \bd) \otimes \bd) \nabla \bn \right] : \nabla^2 \psi
	 		\dS\\
	 		&\hspace{2em}+ \int_\Gamma 
	 			(\nabla^2 \psi \bn) \cdot 
	 			\left[ 
	 				((\nabla^2 \psi \bd) \otimes \nabla (\bd \cdot \bn)) \bn
					- (((\nabla^2 \psi \bd) \otimes \bd) \nabla \bn) \bn
				\right]
	 		\dS\\
	 		&\hspace{2em}- \int_\Gamma
				\nabla^2_\Gamma (\xi - \tau \psi) : ((\nabla^2 \psi \bd) \otimes \bd) 
				+ \left( ((\nabla^2 \psi \bd) \otimes \bd)^T \nabla \bn^T  \nabla \psi \right) \cdot \bn \; \nabla_\Gamma \cdot \bn
			\dS\\
			&\hspace{2em}+ \int_\Gamma
				(\nabla (\nabla \psi \cdot \bn) \cdot \bn)  \nabla_\Gamma \cdot \left( 
					(\nabla_\Gamma (\xi - \tau \psi) \cdot \bd) \bd
					- (\nabla \bn^T \nabla \psi \cdot \bd) \bd
				\right)
		\dS\\
		&\hspace{2em}- \int_\Gamma
			(\nabla (\nabla \psi \cdot \bn) \cdot \bn) ((\nabla^2 \psi \bd) \otimes \bd) : \nabla_\Gamma \bn
			+ \nabla_\Gamma (\nabla \bn^T \nabla \psi) : ((\nabla^2 \psi \bd) \otimes \bd) 
		\dS\\
		&\hspace{2em}- \varepsilon \int_\Gamma
			(\nabla^2 \psi \bd \cdot \bd) \nabla (\bd \cdot \bn) \cdot \nabla^2\psi  \bd
			- (\nabla^2 \psi \bd \cdot \bd)  \nabla \bn^T \bd \cdot \nabla^2 \psi \bd
		\dS\\
		&\hspace{2em}+ \varepsilon \int_\Gamma
			(\nabla^2 \psi \bd \cdot \bn) (\nabla^2 \psi \bd \cdot \bd) \nabla(\bd \cdot \bn) \cdot \bn
			- (\nabla^2 \psi \bd \cdot \bn) (\nabla^2 \psi \bd \cdot \bd) \nabla \bn^T \bd \cdot \bn
		\dS\\
		&\hspace{2em}+ \varepsilon \int_\Gamma
			\nabla_\Gamma \left( \nabla \bn^T \nabla \psi \cdot \bd \right) \cdot ((\nabla^2 \psi \bd \cdot \bd) \bd)
			- \nabla_\Gamma \left( \nabla_\Gamma (\xi - \tau \psi) \cdot \bd \right) \cdot ((\nabla^2 \psi \bd \cdot \bd) \bd))
		\dS.
    \end{split}
	\end{align}
	\begin{comment}
		+ \varepsilon \int_\Gamma
			|\bd|^2 (\nabla \bn)^T \nabla^2 \psi \bd \cdot (\nabla^2 \psi \bd)
			- |\bd|^2 (\nabla^2 \psi \bd \cdot \bn) (\nabla \bn)^T \nabla^2 \psi \bd \cdot \bn
		\dS\\
		+ \varepsilon \int_\Gamma
			\nabla_\Gamma (|\bd|^2 (\nabla \bn)^T \nabla \psi \cdot \bd) \cdot ( \nabla^2 \psi \bd)
			- \nabla_\Gamma (|\bd|^2 \nabla_\Gamma (\xi - \tau \psi) \cdot \bd) \cdot (\nabla^2 \psi \bd)
		\dS\\
		+ \varepsilon \int_\Gamma
			\nabla_\Gamma ((\nabla \bn)^T \nabla \psi \cdot \bd) \cdot (|\bd|^2 \nabla^2 \psi \bd)
			- \nabla_\Gamma (\nabla_\Gamma (\xi - \tau \psi) \cdot \bd) \cdot (|\bd|^2 \nabla^2 \psi \bd)
		\dS\\
		+ \varepsilon \int_\Gamma
			((\nabla_\Gamma (\xi - \tau \psi) \cdot \bd) - ((\nabla \bn)^T \nabla \psi \cdot \bd))
			(|\bd|^2 \nabla^2 \psi \bd \cdot \bn) \nabla_\Gamma \cdot \bn 
		\dS.
	\end{comment}
\end{lem}

\begin{proof}
	Using
	\begin{align*}
		\nabla (\nabla \cdot ((\bd \cdot \nabla \psi) \bd)) 
		= 
			\nabla \cdot (\nabla ((\bd \cdot \nabla \psi) \bd))^T
		=
			\nabla \cdot \left( (\bd \cdot \nabla \psi) \nabla \bd^T + (\nabla \bd^T \nabla \psi) \otimes \bd + (\nabla^2 \psi \bd) \otimes \bd \right)
	\end{align*}
	and expanding the matrix $\varMa$ on the left-hand side of \eqref{eq:reg_en_terms} we obtain
	\begin{align}\label{eq:kappa_exp_1}
    \begin{split}
		\int_\Omega
	 		&\nabla (\Delta \psi) \cdot \varMa \nabla (\nabla \cdot ((\bd \cdot \nabla \psi) \bd))
	 	\diff \X\\
	 	&=
	 		\int_\Omega
	 			\nabla (\Delta \psi) \cdot \varMa \nabla \cdot ((\nabla^2 \psi \bd) \otimes \bd) 
	 			+ \nabla (\Delta \psi) \cdot \varMa \,
	 			\nabla \cdot \left( (\bd \cdot \nabla \psi) \nabla \bd^T + (\nabla \bd^T \nabla \psi) \otimes \bd \right)
	 		\diff \X\\
	 	&=
	 		\int_\Omega
	 			\nabla (\Delta \psi) \cdot \nabla \cdot ((\nabla^2 \psi \bd) \otimes \bd) 
	 			+ \varepsilon (\nabla (\Delta \psi) \cdot \bd) (\nabla \cdot ((\nabla^2 \psi \bd) \otimes \bd) \cdot \bd)
	 		\diff \X\\
	 		&\hspace{1em}+ \int_\Omega
	 			\nabla (\Delta \psi) \cdot \varMa 
	 			\, \nabla \cdot \left( (\bd \cdot \nabla \psi) \nabla \bd^T + (\nabla \bd^T \nabla \psi) \otimes \bd \right)
	 		\diff \X.
    \end{split}
	\end{align}
	The integral on the last line is no longer quadratic in the third order derivative of $\psi$ and will be collected in 
	$\mathrm{l.o.t.}_\kappa$.
	We now turn to the first term on the right-hand side of \eqref{eq:kappa_exp_1}.
	For the terms quadratic in the third order derivatives of $\psi$ to occur with a good sign
	we need to integrate by parts two times and handle the occurring boundary terms.
	Using the integration by parts rule for matrices and tensors,
	we find
	\begin{align}\label{eq:double_ibp}
    \begin{split}
		\int_\Omega
	 		\nabla (\Delta \psi) &\cdot \nabla \cdot ((\nabla^2 \psi \bd) \otimes \bd) 
	 	\diff \X
	 	=
	 		\int_\Omega
	 			\nabla \cdot (\nabla^2 \psi) \cdot \nabla \cdot ((\nabla^2 \psi \bd) \otimes \bd) 
	 		\diff \X\\
	 	&=
	 		- \int_\Omega
	 			(\nabla^2 \psi) : \nabla \cdot (\nabla((\nabla^2 \psi \bd) \otimes \bd)^T)
	 		\diff \X
	 		+ \int_\Gamma 
	 			(\nabla^2 \psi \bn) \nabla \cdot ((\nabla^2 \psi \bd) \otimes \bd) 
	 		\dS\\
	 	&=
	 		\int_\Omega
	 			(\nabla^3 \psi) \threedotsbin (\nabla((\nabla^2 \psi \bd) \otimes \bd)^T)
	 		\diff \X
	 		- \int_\Gamma
	 			[(\nabla((\nabla^2 \psi \bd) \otimes \bd)^T) \cdot \bn] : \nabla^2 \psi
	 		\dS\\
	 		&\hspace{1em}+ \int_\Gamma 
	 			(\nabla^2 \psi \bn) \nabla \cdot ((\nabla^2 \psi \bd) \otimes \bd) 
	 		\dS.
    \end{split}
	\end{align}
	The first term gives a term with a good sign quadratic in the third order derivative of $\psi$ 
	and all other terms can be estimated and are part of $\mathrm{l.o.t.}_\kappa$.
	That the first term on the right-hand side gives us terms with a good sign, can be seen by expanding the tensor product and using
	$\nabla (a \otimes b ) = \nabla a \otimes_M b + a \otimes \nabla b$
	and $\nabla (A b) = (\nabla A)^T \cdot b + A \nabla b$,
	where for a matrix $A\in \R^{d\times d}$ and a vector $b\in\R^d$ we have $A \otimes_M b := (a_{ik} b_j)^d_{i,j,k=1}$ and $b \otimes A := (b_i a_{jk})^d_{i,j,k=1}$.
	First we note
	\begin{align*}
		\nabla (\nabla^2 \psi \bd \otimes \bd) 
		&=
			\nabla  (\nabla^2 \psi \bd) \otimes_M \bd + (\nabla^2 \psi \bd) \otimes \nabla \bd\\
		&=
			((\nabla^3 \psi)^T \cdot \bd) \otimes_M \bd 
			+ (\nabla^2 \psi \nabla \bd) \otimes_M \bd
			+ (\nabla^2 \psi \bd)  \otimes \nabla \bd
	\end{align*}
	and then, using $(\nabla^3 \psi)^T = \nabla^3 \psi$,
    where the transposed of a three dimensional tensor $\mathbb{G} = (g_{ijk})_{i,j,k=1}^3$ is defined as
    $\mathbb{G}^T = (g_{ikj})_{i,j,k=1}^3$,
    since we can interchange the order of derivatives,
	($\psi \in W^{4,10/3}(\Omega)$ 
	and thus $\partial_{x_i} \partial_{x_j} \partial_{x_k} \psi \in W^{1,10/3}(\Omega) \hookrightarrow C(\overline{\Omega})$)
	we get, using $ (\bs A \otimes_M \bs b)^T = \bs A \otimes \bs b$ for $\bs A \in \R^{3 \times 3}$ and $\bs b \in \R^3$
	and that for two tensors $\mathbb{G}, \mathbb{B} \in \R^{3 \times 3 \times 3}$ it holds 
	$\mathbb{G}^T \threedotsbin \mathbb{B} = \mathbb{G} \threedotsbin \mathbb{B}^T$,
	\begin{multline*}
		\nabla^3 \psi \threedotsbin ((\nabla^3 \psi)^T \cdot \bd) \otimes_M \bd)^T
		=	
			\nabla^3 \psi^T \threedotsbin ((\nabla^3 \psi)^T \cdot \bd) \otimes \bd)
		=
			\nabla^3 \psi \threedotsbin ((\nabla^3 \psi)^T \cdot \bd) \otimes \bd)\\
		=
			\sum_{i,j,k,l=1}^3 \left( (\partial_{x_k} \partial_{x_j} \partial_{x_i} \psi)   (\partial_{x_l} \partial_{x_j} \partial_{x_i} \psi) \bd_l \bd_k \right)
		= 
			\sum_{i,j=1}^3 (\nabla^3 \psi \cdot \bd)^2_{ij}
		=
			|\nabla^3 \psi \cdot \bd|^2.
	\end{multline*}
	Using $\nabla \cdot (\bs A \bs b) = \nabla \cdot (\bs A^T) \cdot \bs b + (\nabla \bs b) : \bs A^T$ 
	we can rewrite the second term on the right-hand side of \eqref{eq:kappa_exp_1} as
	\begin{align*}
		\varepsilon \int_\Omega
	 		(\nabla &(\Delta \psi) \cdot \bd) (\nabla \cdot ((\nabla^2 \psi \bd) \otimes \bd) \cdot \bd)
	 	\diff \X\\
	 	&=
	 		\varepsilon \int_\Omega
	 			\left[ \nabla \cdot (\nabla^2 \psi \bd) - \nabla \bd :  \nabla^2 \psi \right]
	 			\left[ \nabla \cdot ((\bd \otimes (\nabla^2 \psi \bd)) \bd) - \nabla \bd : ((\nabla^2 \psi \bd) \otimes \bd) \right]
	 		\diff \X\\
	 	&= 
	 		\varepsilon \int_\Omega
	 			\nabla \cdot (\nabla^2 \psi \bd) \nabla \cdot ((\bd  \otimes (\nabla^2 \psi \bd)) \bd)
	 			- \nabla \cdot (\nabla^2 \psi \bd) (\nabla \bd : ((\nabla^2 \psi \bd) \otimes \bd))
	 		\diff \X\\
	 		&\hspace{1em}+ \varepsilon \int_\Omega
	 			(\nabla \bd :  \nabla^2 \psi) (\nabla \bd : ((\nabla^2 \psi \bd) \otimes \bd))
	 			- (\nabla \bd :  \nabla^2 \psi) (\nabla \cdot ((\bd \otimes (\nabla^2 \psi \bd)) \bd))
	 		\diff \X.
	\end{align*}
	The first term also has to be integrated by parts two times to see that it has a good sign 
	and all other terms are in $\mathrm{l.o.t.}_\kappa$.
	Performing the integration by parts we find
	\begin{align}\label{eq:double_ibp_d}
    \begin{split}
		\varepsilon \int_\Omega \nabla &\cdot (\nabla^2 \psi \bd) \nabla \cdot ((\bd  \otimes (\nabla^2 \psi \bd)) \bd) \diff \X
		=
			\varepsilon \int_\Omega \nabla \cdot (\nabla^2 \psi \bd) \nabla \cdot ((\nabla^2 \psi \bd \cdot \bd) \bd) \diff \X\\
		&=
			- \varepsilon \int_\Omega
				\nabla^2 \psi \bd \cdot \nabla \cdot (\nabla((\nabla^2 \psi \bd \cdot \bd) \bd)^T)
			\diff \X
			+\varepsilon \int_\Gamma
				(\nabla^2 \psi \bd \cdot \bn)  \nabla \cdot ((\nabla^2 \psi \bd \cdot \bd) \bd)
			\dS\\
		&=
			\varepsilon \int_\Omega
				\nabla(\nabla^2 \psi \bd) : (\nabla((\nabla^2 \psi \bd \cdot \bd) \bd)^T)
			\diff \X
			- \varepsilon \int_\Gamma
				((\nabla((\nabla^2 \psi \bd \cdot \bd) \bd)^T) \bn) \cdot (\nabla^2 \psi \bd)
			\dS\\
			&\hspace{1em}+ \varepsilon \int_\Gamma
				(\nabla^2 \psi \bd \cdot \bn) \nabla \cdot ((\nabla^2 \psi \bd \cdot \bd) \bd)
			\dS.
    \end{split}
	\end{align}
	The boundary integrals are in $\mathrm{l.o.t.}_\kappa$, which we will see below, 
	and the matrix scalar product gives a term with a good sign and lower order terms.
	Using
	\begin{align*}
		(\nabla ((\nabla^2 \psi \bd \cdot \bd) \bd)^T
		 =
		 	(\nabla^2 \psi \bd \cdot \bd) \nabla \bd^T
		 	+ (\nabla (\nabla^2 \psi \bd)^T \bd) \otimes \bd 
		 	+ (\nabla \bd^T \nabla^2 \psi \bd) \otimes \bd
	\end{align*}
	and
	\begin{align*}
		\nabla (\nabla^2 \psi \bd)^T 
		=
			(\nabla^3 \psi \cdot \bd)^T + (\nabla^2 \psi \nabla \bd)^T
		=
			(\nabla^3 \psi \cdot \bd) + \nabla \bd^T \nabla^2 \psi,
	\end{align*}
	where we again used the fact that we can interchange the order of third order partial derivatives of $\psi$ and thus have 
	$\nabla^3 \psi \cdot \bd = (\nabla^3 \psi \cdot \bd)^T$,
	we obtain
	\begin{align}\label{eq:append_matrix_exp_2}
    \begin{split}
		\nabla(\nabla^2 &\psi \bd) : (\nabla ((\nabla^2 \psi \bd \cdot \bd) \bd)^T\\
		&=
			( \nabla^3 \psi \cdot \bd + \nabla^2 \psi \nabla \bd) : (\nabla ((\nabla^2 \psi \bd \cdot \bd) \bd)^T
		=
			( \nabla^3 \psi \cdot \bd) : (\nabla^3 \psi \cdot \bd) \bd \otimes \bd\\
			&\hspace{1em}+ ( \nabla^3 \psi \cdot \bd + \nabla^2 \psi \nabla \bd) :
			\left( 
				(\nabla^2 \psi \bd \cdot \bd) \nabla \bd^T
		 		+ (\nabla \bd^T \nabla^2 \psi \bd) \otimes \bd 
		 		+ (\nabla \bd^T \nabla^2 \psi \bd) \otimes \bd
		 	\right)\\
		 	&\hspace{1em}+ \nabla^2 \psi \nabla \bd  : (\nabla^3 \psi \cdot \bd) \bd \otimes \bd.
    \end{split}
	\end{align}
	Using $\bs a \otimes \bs b : \bs A= \bs a \cdot \bs A b$ we find for the first term on the right-hand side
	\begin{align*}
		( \nabla^3 \psi \cdot \bd) : (\nabla^3 \psi \cdot \bd) \bd \otimes \bd
		=
			|(\nabla^3 \psi \cdot \bd) \bd|^2.
	\end{align*} 
	All other terms in \eqref{eq:append_matrix_exp_2} are of lower order, that is not quadratic in the third order derivative of $\psi$.
	To see that also the boundary terms are of lower order, where we call boundary terms of lower order if there are no third order derivatives of $\psi$, 
	we need some involved integration by parts to be able to plug in the Robin boundary 
	condition for the electric potential $\psi$ and reduce the order of derivatives by one.
	For the first boundary integral in \eqref{eq:double_ibp} we find
	that the third order term in the second term vanishes since $\bd \cdot \bn = 0$ and
    for $\bs A \in C^1( \Omega)$ and $\bs b \in C^1(\Omega)$ $\nabla(\bs A \bs b) = (\nabla \bs A)^T \cdot \bs b + \bs A \nabla \bs b$,
	\begin{align}\label{eq:bd_kappa_vanish_1}
    \begin{split}
		\nabla((\nabla^2 \psi \bd) \otimes \bd)^T \cdot \bn
		&= 
			\nabla (((\nabla^2 \psi \bd) \otimes \bd) \bn)
			- ((\nabla^2 \psi \bd) \otimes \bd) \nabla \bn \\
		&=
			\nabla ((\bd \cdot \bn) \nabla^2 \psi \bd)
			- ((\nabla^2 \psi \bd) \otimes \bd) \nabla \bn\\
		&=
			(\bd \cdot \bn) \nabla (\nabla^2 \psi \bd)
			+ (\nabla^2 \psi \bd) \otimes \nabla (\bd \cdot \bn)
			- ((\nabla^2 \psi \bd) \otimes \bd) \nabla \bn\\
		&=
			(\nabla^2 \psi \bd) \otimes \nabla (\bd \cdot \bn)
			- ((\nabla^2 \psi \bd) \otimes \bd) \nabla \bn
    \end{split}
	\end{align}
	and thus
	\begin{align*}
		- \int_\Gamma
	 		(\nabla((\nabla^2 \psi \bd) \otimes \bd)^T \cdot \bn) : \nabla^2 \psi
	 	\dS
	 	=
	 		- \int_\Gamma
	 			\left[ (\nabla^2 \psi \bd) \otimes \nabla (\bd \cdot \bn) - ((\nabla^2 \psi \bd) \otimes \bd) \nabla \bn \right] : \nabla^2 \psi
	 		\dS.
	\end{align*}
	Now, we turn to the second boundary integral in \eqref{eq:double_ibp}.
	Using the characterization of the surface divergence of a matrix from Corollary~\ref{cor:surf_div_form_matrix} we find
	\begin{multline}\label{eq:kappa_bd_exp_1}
		\int_\Gamma 
	 		(\nabla^2 \psi \bn) \cdot \nabla \cdot ((\nabla^2 \psi \bd) \otimes \bd) 
	 	\dS\\
	 	=
	 		\int_\Gamma 
	 			(\nabla^2 \psi \bn) \cdot \nabla_\Gamma \cdot ((\nabla^2 \psi \bd) \otimes \bd) 
	 			+ (\nabla^2 \psi \bn) \cdot ((\nabla ((\nabla^2 \psi \bd) \otimes \bd) \cdot \bn) \bn)
	 		\dS.
	\end{multline}
	The higher order term in the second term again vanishes, 
	which can be seen by using $(\mathbb{G} \cdot \bs a) \bs b = (\mathbb{G}^T \cdot \bs b) \bs a$,
    for $\mathbb{G} \in \R^{3 \times 3 \times 3}$ and $\bs a, \bs b \in \R^3$
	\begin{align*}
		(\nabla ((\nabla^2 \psi \bd) \otimes \bd) \cdot \bn) \bn
	 	= 
	 		(\nabla ((\nabla^2 \psi \bd) \otimes \bd)^T \cdot \bn) \bn
	\end{align*}
	and then proceeding as above in \eqref{eq:bd_kappa_vanish_1}.
	We now turn to the first term in \eqref{eq:kappa_bd_exp_1}. To see that we can reduce the order of derivatives by one 
	so that no terms with third order derivatives appear on the boundary we use the differential operators on the boundary
    \textit{cf.~}Theorem~\ref{thm:surf_grad_vs_tan_pro_scalar}
	and integration by parts on the boundary, \textit{cf.~}Corollary~\ref{cor:ibp_boundary}.
	We first note that for any test function $\bv \in L^2(\Gamma)$ we have
	\begin{align}\label{eq:nabla2psi_tested}
    \begin{split}
		\int_\Gamma (\nabla^2 \psi \bn)  \cdot \bv \dS
		&=
			\int_\Gamma
				\nabla (\nabla \psi \cdot \bn) \cdot \bv 
				- (\nabla \bn^T \nabla \psi) \cdot \bv
			\dS\\
		&=
			\int_\Gamma
				\nabla_\Gamma (\nabla \psi \cdot \bn) \cdot \bv 
				+ (\nabla (\nabla \psi \cdot \bn) \cdot \bn) (\bv \cdot \bn)
				- (\nabla \bn^T \nabla \psi) \cdot \bv
			\dS\\
		&=
			\int_\Gamma
				\nabla_\Gamma (\xi - \tau \psi) \cdot \bv 
				+ (\nabla (\nabla \psi \cdot \bn) \cdot \bn) (\bv \cdot \bn)
				- (\nabla \bn^T \nabla \psi) \cdot \bv
			\dS.
    \end{split}
	\end{align}
	Plugging this back into the first term of \eqref{eq:kappa_bd_exp_1}
	and using $\nabla_\Gamma \cdot (\bs A^T \bs b) = (\nabla_\Gamma \cdot \bs A) \cdot \bs b + \bs A : \nabla_\Gamma \bs b$,
	\textit{cf.~}Lemma~\ref{lem:surf_prod_rule_matrix},
	and the integration by parts formular \eqref{eq:bd_ibp_matrix},
	we obtain,
	\begin{align}\label{eq:kappa_bd_exp_2}
    \begin{split}
		\int_\Gamma 
	 		&(\nabla^2 \psi \bn) \cdot \nabla_\Gamma \cdot ((\nabla^2 \psi \bd) \otimes \bd) 
	 	\dS\\
	 	&=
	 		\int_\Gamma
				\nabla_\Gamma (\xi - \tau \psi) \cdot \nabla_\Gamma \cdot ((\nabla^2 \psi \bd) \otimes \bd) 
				+ (\nabla (\nabla \psi \cdot \bn) \cdot \bn) \nabla_\Gamma \cdot ((\nabla^2 \psi \bd) \otimes \bd) \cdot \bn
            \dS\\
            &\hspace{1em}- \int_\Gamma
				((\nabla \bn)^T \nabla \psi) \cdot \nabla_\Gamma \cdot ((\nabla^2 \psi \bd) \otimes \bd) 
			\dS\\
		&=
			\int_\Gamma
				- \nabla^2_\Gamma (\xi - \tau \psi) : ((\nabla^2 \psi \bd) \otimes \bd) 
				+ \left( ((\nabla^2 \psi \bd) \otimes \bd)^T \nabla_\Gamma (\xi - \tau \psi) \right) \cdot \bn \; \nabla_\Gamma \cdot \bn 
			\dS\\
			&\hspace{1em}+ \int_\Gamma
				(\nabla (\nabla \psi \cdot \bn) \cdot \bn) \nabla_\Gamma \cdot ( (\bd \otimes (\nabla^2 \psi \bd)) \bn)
				- (\nabla (\nabla \psi \cdot \bn) \cdot \bn) ((\nabla^2 \psi \bd) \otimes \bd) : \nabla_\Gamma \bn
			\dS\\
			&\hspace{1em}+ \int_\Gamma
				\nabla_\Gamma ((\nabla \bn)^T \nabla \psi) : ((\nabla^2 \psi \bd) \otimes \bd) 
				-  \left( ((\nabla^2 \psi \bd) \otimes \bd)^T \nabla \bn^T \nabla \psi \right) \cdot \bn \; \nabla_\Gamma \cdot \bn
			\dS,
    \end{split}
	\end{align}
	where the second term on the right-hand side vanishes due to $\bd \cdot \bn = 0$ on $\Gamma$.
	Using
	\begin{align*}
		(\bd \otimes (\nabla^2 \psi \bd)) \bn
		&=
			((\nabla^2 \psi \bd) \cdot \bn) \bd
		=
			((\nabla^2 \psi \bn) \cdot \bd) \bd
		=
			(\nabla (\nabla \psi \cdot \bn) \cdot \bd) \bd
			- ((\nabla \bn^T \nabla \psi) \cdot \bd) \bd\\
		&=
			(\nabla_\Gamma (\nabla \psi \cdot \bn) \cdot \bd) \bd
			+ (\nabla (\nabla \psi \cdot \bn) \cdot \bn) (\bn \cdot \bd) \bd
			- ((\nabla \bn^T \nabla \psi) \cdot \bd) \bd\\
		&=
			(\nabla_\Gamma (\xi - \tau \psi) \cdot \bd) \bd
			- ((\nabla \bn^T \nabla \psi) \cdot \bd) \bd
	\end{align*}
	we can reduce the last term in \eqref{eq:kappa_bd_exp_2} including third order derivatives of $\psi$
	and we find
	\begin{align*}
	\int_\Gamma 
	 	(\nabla^2 \psi \bn) &\cdot \nabla_\Gamma \cdot ((\nabla^2 \psi \bd) \otimes \bd) 
	 \dS
	 =
	 	-\int_\Gamma
			\nabla^2_\Gamma (\xi - \tau \psi) : ((\nabla^2 \psi \bd) \otimes \bd)
		\dS\\
		&+ \int_\Gamma
			(\nabla (\nabla \psi \cdot \bn) \cdot \bn)  \nabla_\Gamma \cdot \left( 
				(\nabla_\Gamma (\xi - \tau \psi) \cdot \bd) \bd
				- (\nabla \bn^T \nabla \psi \cdot \bd) \bd
			\right)
		\dS\\
		&- \int_\Gamma
			(\nabla (\nabla \psi \cdot \bn) \cdot \bn) ((\nabla^2 \psi \bd) \otimes \bd) : \nabla_\Gamma \bn
			+ \nabla_\Gamma (\nabla \bn^T \nabla \psi) : ((\nabla^2 \psi \bd) \otimes \bd) 
		\dS\\
		&- \int_\Gamma
			\left( ((\nabla^2 \psi \bd) \otimes \bd)^T \nabla \bn^T  \nabla \psi \right) \cdot \bn \; \nabla_\Gamma \cdot \bn
		\dS.
	\end{align*}
	%where no curvature term appears when we integrate by parts since 
	%\[
	%	(\nabla_\Gamma (\xi - \tau \psi) \cdot \bd) \bd - (\nabla \bn^T \nabla \psi \cdot \bd) \bd
	%\] 
	%is indeed a tangential vector field. 
	Next, we turn to the boundary integrals in \eqref{eq:double_ibp_d} and proceed as above.
	For the first boundary integral on the right-hand side of \eqref{eq:double_ibp_d}
	we find as above that the third order term vanishes
	\begin{align*}
		\nabla ((\nabla^2 \psi \bd \cdot \bd) \bd)^T \bn
		&=
			\nabla ((\nabla^2 \psi \bd \cdot \bd) (\bd \cdot \bn)) - (\nabla^2 \psi \bd \cdot \bd) \nabla \bn^T \bd\\
		&=
			(\bd \cdot \bn ) \nabla (\nabla^2 \psi \bd \cdot \bd) + (\nabla^2 \psi \bd \cdot \bd) \nabla (\bd \cdot \bn)
			- (\nabla^2 \psi \bd \cdot \bd) \nabla \bn^T \bd\\
		&=
			(\nabla^2 \psi \bd \cdot \bd) \nabla (\bd \cdot \bn)
			- (\nabla^2 \psi \bd \cdot \bd) \nabla \bn^T \bd.
	\end{align*}
	Next, we turn to the second boundary integral on the right-hand side of \eqref{eq:double_ibp_d}.
    Using the characterization of the surface divergence, \textit{cf.~}Theorem~\ref{thm:surf_div_form}, we find
	\begin{multline*}
		\varepsilon \int_\Gamma
			(\nabla^2 \psi \bd \cdot \bn) \nabla \cdot ((\nabla^2 \psi \bd \cdot \bd) \bd)
		\dS\\
		= 
			\varepsilon \int_\Gamma
				(\nabla^2 \psi \bd \cdot \bn) \nabla_\Gamma \cdot ((\nabla^2 \psi \bd \cdot \bd) \bd)
				+ (\nabla^2 \psi \bd \cdot \bn) \nabla ((\nabla^2 \psi \bd \cdot \bd) \bd) \bn \cdot \bn
			\dS.
	\end{multline*}
	The third order term in the second term on the right again vanishes as above.
	For the first term we use \eqref{eq:nabla2psi_tested} and integrate by parts
	\begin{align*}
		\varepsilon \int_\Gamma
			(\nabla^2 \psi &\bd \cdot \bn) \nabla_\Gamma \cdot ((\nabla^2 \psi \bd \cdot \bd) \bd)
		\dS\\
		&=
			\varepsilon \int_\Gamma
				\nabla_\Gamma (\xi - \tau \psi) \cdot \bd (\nabla_\Gamma \cdot ((\nabla^2 \psi \bd \cdot \bd) \bd))
				- \nabla \bn^T \nabla \psi \cdot \bd \nabla_\Gamma \cdot ((\nabla^2 \psi \bd \cdot \bd) \bd)
			\dS\\
		&=
			\varepsilon \int_\Gamma
				\nabla_\Gamma \left( \nabla \bn^T \nabla \psi \cdot \bd \right) \cdot ((\nabla^2 \psi \bd \cdot \bd) \bd)
				- \nabla_\Gamma \left( \nabla_\Gamma (\xi - \tau \psi) \cdot \bd \right) \cdot ((\nabla^2 \psi \bd \cdot \bd) \bd))
			\dS.
	\end{align*}
	Putting everything together we obtain \eqref{eq:reg_en_terms}.
\end{proof}

\begin{lem}\label{lem:lot_kappa_est}
	Under the assumptions from Lemma~\ref{lem:reg_en_terms} and $\kappa >0$ small enough, 
	that is for $\kappa$ such that $C \kappa \left( 1 + \norm{\bd}^2_{W^{2,\infty}(\Omega)} \right) \leq 1/32$,
	we have
	\[
		\kappa \mathrm{l.o.t.}_\kappa 
		\leq 
			C 
			+ \frac{\kappa}{2} \norm{\nabla (\Delta \psi)}^2_{L^2(\Omega)} 
			+ \frac{1}{4} \norm{\nabla^2 \psi}^2_{L^2(\Omega)}
	\]
	for some constant dependent on $\norm{\psi}_{W^{1,2}(\Omega)}$ and $\norm{\bd}_{W^{1,\infty}(\Omega)}$ and the outer normal field $\bn$.
\end{lem}
\begin{proof}
	We first turn to the volume terms of $\mathrm{l.o.t.}_\kappa$.
	Using H\"{o}lder's and Young's inequalities we find
	\begin{align*}
		\kappa &\left|\int_\Omega
	 		\mathrm{l.o.t.}_{\Omega \kappa}
		 \diff \X \right|
		 \leq 
		 	\kappa C \norm{\nabla^3 \psi}_{L^2(\Omega)}
		 	\left(
		 		1 + \norm{\nabla^2 \psi}_{L^2(\Omega)} + \norm{\bd}_{W^{2,\infty}(\Omega)}
		 	\right)
		 	+ \kappa C \norm{\nabla^2 \psi}^2_{L^2(\Omega)}\\
		\underset{\ref{lem:full_third_vs_div}}&{\leq}
			\kappa C \left( \norm{\nabla (\Delta \psi)}_{L^2(\Omega)} + \norm{\nabla^2 \psi}_{L^2(\Omega)} + 1 \right)
		 	\left(
		 		1 + \norm{\nabla^2 \psi}_{L^2(\Omega)} + \norm{\bd}_{W^{2,\infty}(\Omega)}
		 	\right)
		 	+ \kappa C \norm{\nabla^2 \psi}^2_{L^2(\Omega)}\\
		 &\leq
		 	\frac{\kappa}{4} \norm{\nabla (\Delta \psi)}^2_{L^2(\Omega)}
		 	+ \kappa C 
		 	\left(
		 		1 + \norm{\nabla^2 \psi}^2_{L^2(\Omega)} + \norm{\bd}^2_{W^{2,\infty}(\Omega)}
		 	\right)\\
		 &\leq
		 	\frac{\kappa}{4} \norm{\nabla (\Delta \psi)}^2_{L^2(\Omega)}
		 	+ \frac{1}{8} \norm{\nabla^2 \psi}^2_{L^2(\Omega)}
		 	+ C,
	\end{align*}
	where we used Lemma~\ref{lem:full_third_vs_div} and the smallness assumption on $\kappa$ for the last inequality.
	Now, we turn to the boundary integrals in $\mathrm{l.o.t.}_\kappa$, \textit{cf.~}\eqref{eq:def_kappa_lot}.
	For that we use that we can estimate the surface differential operators with the bulk one, 
	\textit{cf.~}Lemma~\ref{lem:est_surf_grad} below, to obtain
	\begin{align*}
		\kappa (\varepsilon &+ \lambda) \left| - \int_\Gamma
	 		\left[ (\nabla^2 \psi \bd) \otimes \nabla (\bd \cdot \bn) - ((\nabla^2 \psi \bd) \otimes \bd) \nabla \bn \right] : \nabla^2 \psi
	 	\dS \right|\\
		&\leq
			\kappa C \left( 1 + \norm{\bd}_{W^{2,\infty}(\Omega)} \right) \norm{\nabla^2 \psi}^2_{L^2(\Gamma)}
		\leq
			\kappa C \left( 1 + \norm{\bd}^2_{W^{2,\infty}(\Omega)} \right) \norm{\nabla^2 \psi}^2_{L^2(\Omega)}
			+ \frac{\kappa}{32} \norm{\nabla^3 \psi}^2_{L^2(\Omega)}\\
		&\leq
			\kappa C \left( 1 + \norm{\bd}^2_{W^{2,\infty}(\Omega)} \right) \norm{\nabla^2 \psi}^2_{L^2(\Omega)}
			+ \frac{\kappa}{16} \norm{\nabla (\Delta \psi)}^2_{L^2(\Omega)}
			+ C\\
		&\leq
			C
			+ \frac{1}{32} \norm{\nabla^2 \psi}^2_{L^2(\Omega)}
			+ \frac{\kappa}{16} \norm{\nabla (\Delta \psi)}^2_{L^2(\Omega)},
	\end{align*}
	where we used the trace estimate~\cite[Prop.~8.2]{diBenedetto} to find
	\[
		 \norm{\nabla^2 \psi}^2_{L^2(\Gamma)} 
		 \leq 
		 	\delta \norm{\nabla^3 \psi}^2_{L^2(\Omega)} + (1 + \delta^{-1}) \norm{\nabla^2 \psi}^2_{L^2(\Omega)}
	\]
	with $\delta = \frac{1}{32} \left( 1 + \norm{\bd}_{W^{2,\infty}(\Omega)} \right)^{-1}$,
	Lemma~\ref{lem:full_third_vs_div} to estimate
	\[
		\norm{\nabla^3 \psi }^2_{L^2(\Omega)} 
		\leq 2 \norm{\nabla (\Delta \psi)}^2_{L^2(\Omega)} + C \left( \norm{\nabla^2 \psi}^2_{L^2(\Omega)} + 1 \right)
	\]
	and the smallness assumption for $\kappa$.
	All other boundary terms in $\mathrm{l.o.t.}_\kappa$ can be estimated similarly
	and we obtain 
	\[
		\kappa \left| \int_\Gamma \mathrm{l.o.t}_{\Gamma \kappa} \dS \right|
		\leq 
			\frac{1}{8} \norm{\nabla^2 \psi}^2_{L^2(\Omega)}
			+ \frac{\kappa}{4} \norm{\nabla (\Delta \psi)}^2_{L^2(\Omega)}
			+ C.
	\]
\end{proof}

\begin{lem}\label{lem:all_lower_ord_terms}
	Under the assumptions of Proposition~\ref{propo:reg_en_ineq_2}, 
	we can estimate $\mathrm{l.o.t.}_\kappa$ from Lemma~\ref{lem:reg_en_terms}, 
	$\mathrm{l.o.t.}_\Omega$ given in \eqref{eq:omega_lot},
	and $\mathrm{l.o.t.}_{\Gamma}$ given in \eqref{eq:gamma_lot}
	by:
	\begin{multline}\label{eq:all_lot_est}
	\left| 2 \int_0^{\Tmax} \left(
		\int_\Omega \mathrm{l.o.t.}_\Omega \diff \X 
	 	+ \int_\Gamma \mathrm{l.o.t.}_\Gamma \dS 
	 	+ \kappa \mathrm{l.o.t.}_\kappa \right)
	 \diff t \right|
		\leq 
			C
			+ \frac{\tau}{2}  \norm{\nabla_\Gamma \psi}^2_{L^2(0,\Tmax;L^2(\Gamma))}\\
			+ \tau \frac{\varepsilon + \lambda}{2} \norm{\nabla_\Gamma \psi \cdot \bd}^2_{L^2(0,\Tmax;L^2(\Gamma))}
			+ \frac{7}{8}\norm{\nabla^2 \psi}^2_{L^2(0,\Tmax;L^2(\Omega))}
			+ \frac{\kappa}{2} \norm{\nabla (\Delta \psi)}^2_{L^2(0,\Tmax;L^2(\Omega))}.
	\end{multline}
\end{lem}
\begin{proof}
	We recall
	\[
		\mathrm{l.o.t.}_\Omega 
		= (\varepsilon + \lambda) \nabla^2 \psi : \left( (\bd \cdot \nabla \psi) \nabla \bd + \bd \otimes  (\nabla \bd^T \nabla \psi) \right)
	\]
	from \eqref{eq:omega_lot} and find
	\begin{multline*}
		\left| 2 \int_0^{\Tmax} \int_\Omega \mathrm{l.o.t.}_\Omega  \diff \X \diff t \right| 
		\leq 
			2 (\varepsilon + \lambda) \int_0^{\Tmax} \int_\Omega
				 \left| \nabla^2 \psi : \left( (\bd \cdot \nabla \psi) \nabla \bd + \bd \otimes  (\nabla \bd^T \nabla \psi) \right) \right|
			\diff \X \diff t\\
		\leq
			\frac{1}{8} \norm{\nabla^2 \psi}^2_{L^2(0,\Tmax;L^2(\Omega)}
			+ C \norm{\bd}^4_{W^{1,\infty}(\Omega)} \norm{\nabla \psi}^2_{L^2(0,\Tmax;L^2(\Omega))}
		\leq
			\frac{1}{8} \norm{\nabla^2 \psi}^2_{L^2(0,\Tmax; L^2(\Omega))} + C.
	\end{multline*}
	Next, we recall the definition of $\mathrm{l.o.t.}_\Gamma$, \textit{cf.~}\eqref{eq:gamma_lot},
	\begin{align}\label{eq:gamma_lot_est}
    \begin{split}
		&\mathrm{l.o.t.}_\Gamma = 
			- 2 \nabla_\Gamma \xi \cdot \nabla_\Gamma \psi
			+ \nabla \bn^T \nabla \psi \cdot \nabla \psi
			+ (\xi - \tau \psi)^2 \nabla_\Gamma \cdot \bn
			- \nabla \bn^T \nabla \psi \cdot \bn (\xi - \tau \psi)\\
			&\hspace{1em}+ (\varepsilon + \lambda)\left( 
				(\bd \cdot \nabla \psi) \nabla \bn^T\bd  \cdot \nabla \psi
				- (\nabla_\Gamma \xi \cdot \bd) (\nabla_\Gamma \psi \cdot \bd)
				- (\xi - \tau \psi) (\bd \cdot \nabla \psi) \nabla \bn^T \bd \cdot \bn
			\right).
    \end{split}
	\end{align}
	We estimate each term individually.
	First, using H\"{o}lder's inequality, we find
	\begin{align*}
		\left| 2 \int_0^{\Tmax} \int_\Gamma (- 2 \nabla_\Gamma \xi \cdot \nabla_\Gamma \psi) \dS \diff t \right| 
		\leq
			C \norm{\nabla_\Gamma \xi}^2_{L^2(0,\Tmax;L^2(\Gamma))}
			+ \frac{\tau}{2}\norm{\nabla_\Gamma \psi}^2_{L^2(0,\Tmax;L^2(\Gamma))}.
	\end{align*}
	For the second term we use a trace estimate
	and a Sobolev--Slobodeckij interpolation \cite[Thm.~II.3-3]{oru_1998}, to estimate the gradient of $\psi$ on the boundary by
	\begin{equation*}
		\norm{\nabla \psi}^2_{L^2(\Gamma)} 
		\leq 
			\norm{\nabla \psi}^2_{W^{1/4, 2}(\Gamma)}
		\leq	
			C \norm{\nabla \psi}^2_{W^{3/4,2}(\Omega)}
		\leq
			C \norm{\psi}^2_{W^{7/4, 2}(\Omega)}
		\leq
		C \norm{\psi}^{1/2}_{W^{1,2}(\Omega)} \norm{\psi}^{3/2}_{W^{2,2}(\Omega)}.
	\end{equation*}
	With this estimate the second term in \eqref{eq:gamma_lot_est} can be estimated as follows,
	\begin{align}\label{eq:lot_est_interpol}
    \begin{split}
		\hspace*{-1em}\bigg| 2 \int_0^{\Tmax} \int_\Gamma \nabla \bn^T \nabla \psi \cdot \nabla \psi \dS \diff t \bigg| 
		&\leq 
			C \norm{\nabla \psi}^2_{L^2(0,\Tmax;L^2(\Gamma))}
		\leq 
			C 
			\int_0^{\Tmax} \norm{\psi}^{1/2}_{W^{1,2}(\Omega)} \norm{\psi}^{3/2}_{W^{2,2}(\Omega)} \diff t\\
		\underset{\mathrm{Young}}&{\leq}
			C
			\norm{\psi}^{2}_{L^2(0,\Tmax;W^{1,2}(\Omega))}
			+ \frac{1}{8} \norm{\nabla^2 \psi}^2_{L^2(0,\Tmax; L^2(\Omega))}.
    \end{split}
	\end{align} 
	The last term on the right-hand side can be absorbed into the $|\nabla^2 \psi|^2$ term of \eqref{eq:reg_en_ineq_2_collection} 
	with good sign and the first term is bounded by the first energy inequality.
	For the third term in \eqref{eq:gamma_lot_est} we find
	\begin{align*}
		\left| 2 \int_0^{\Tmax} \int_\Gamma (\xi - \tau \psi)^2 \nabla_\Gamma \cdot \bn \dS \diff t \right| 
		&\leq 
			C \norm{\xi - \tau \psi}^2_{L^2(0, \Tmax; L^2(\Gamma))}\\
		&\leq
			C
			\left( 
				\norm{\xi}^2_{L^2(0, \Tmax; L^2(\Gamma))}
				+ C \norm{\psi}_{L^2(0, \Tmax; W^{1,2}(\Omega))}
			\right)
		\leq 
			C.
	\end{align*}
	The fourth term in \eqref{eq:gamma_lot_est} can be handled, using the trace estimate \cite[Prop.~8.2]{diBenedetto},
	which gives us that for all $p \in [1, d)$ there exists $C>0$ such that for all $\delta > 0$ and all $u \in W^{1,p}(\Omega)$ 
	\begin{align}
		\norm{u}_{L^q(\Gamma)} \leq \delta \norm{\nabla u}_{L^p(\Omega)} + C \left( 1 + \frac{1}{\delta} \right) \norm{u}_{L^p(\Omega)}
	\end{align}
	holds, for $q \in \left[1,\frac{(d - 1)p}{(d - p)} \right]$.
	With $q=2=p$, which is a valid choice for $d = 3$, we find
	\begin{align*}
		\bigg| 2 \int_0^{\Tmax} &\int_\Gamma (- (\nabla \bn)^T \nabla \psi \cdot \bn (\xi - \tau \psi)) \dS \diff t \bigg| \\
		&\leq
			2 \norm{\nabla \bn}_{L^\infty(\Gamma)} \norm{\bn}_{L^\infty(\Gamma)}
			\int_0^{\Tmax}
				\norm{\nabla \psi}_{L^2(\Gamma)} \norm{\xi}_{L^2(\Gamma)}
				+ \tau \norm{\nabla \psi}_{L^2(\Gamma)} \norm{\psi}_{L^2(\Gamma)}
			\diff t\\
		&\leq
			\int_0^{\Tmax}
				\frac{1}{8} \norm{\nabla^2 \psi}^2_{L^2(\Omega)}
				+ C \left(
					\norm{\nabla \psi}^2_{W^{1,2}(\Omega)} 
					+ \norm{\xi}^2_{L^2(\Gamma)}
				\right)
			\diff t
		\leq
			C + \frac{1}{8} \norm{\nabla^2 \psi}^2_{L^2(0,\Tmax; L^2(\Omega))}.
	\end{align*}
	For the fifth term in \eqref{eq:gamma_lot_est}, which is the first on the second line,
	we proceed analogously to \eqref{eq:lot_est_interpol} and find
	\begin{align*}
		\bigg| 2 (\varepsilon + \lambda) \int_0^{\Tmax} \int_\Gamma (\bd &\cdot \nabla \psi) \nabla \bn^T\bd  \cdot \nabla \psi\dS \diff t \bigg| \\
		&\leq 
		 	C
		 	\norm{\bd}^2_{W^{1,\infty}(\Omega)} 
		 	\norm{\nabla \psi}^2_{L^2(0,\Tmax;L^2(\Gamma))}\\
		 &\leq
		 	C
		 	\norm{\bd}^8_{W^{1,\infty}(\Omega)} 
		 	\norm{\psi}^2_{L^2(0,\Tmax;W^{1,2}(\Omega))}
		 	+ \frac{1}{8}\norm{\nabla^2 \psi}^2_{L^2(0,\Tmax;L^2(\Omega))}.
	\end{align*}
	Next we turn to the second term on the second line of \eqref{eq:gamma_lot_est} and 
	find
	\begin{multline*}
		\left| 2 (\varepsilon + \lambda) \int_0^{\Tmax} \int_\Gamma 
			(- (\nabla_\Gamma \xi \cdot \bd) (\nabla_\Gamma \psi \cdot \bd)) 
		\dS \diff t \right| \\
		\leq 
			C \norm{\xi}^2_{L^2(0,\Tmax; W^{1,2}(\Gamma))} \norm{\bd}^2_{W^{1,\infty}(\Omega)}
			+ \tau \frac{\varepsilon + \lambda}{2} \norm{\nabla_\Gamma \psi \cdot \bd}^2_{L^2(0,\Tmax;L^2(\Gamma))}.
	\end{multline*}
	Finally, we turn to the last term on the second line of \eqref{eq:gamma_lot_est}
	and find
	\begin{align*}
		\bigg| 2 (\varepsilon + \lambda) &\int_0^{\Tmax} \int_\Gamma 
			(- (\xi - \tau \psi) (\bd \cdot \nabla \psi) \nabla \bn^T \bd \cdot \bn)
		\dS \diff t \bigg| \\
		&\leq
			C \norm{\bd}^2_{W^{1, \infty}(\Omega)}
			\int_0^{\Tmax} \norm{\nabla \psi}_{L^2(\Gamma)} \left(\norm{\xi}_{L^2(\Gamma)} + \tau \norm{\psi}_{L^2(\Gamma)} \right) \diff t\\
		&\leq 
			C \norm{\bd}^2_{W^{1, \infty}(\Omega)}
			\int_0^{\Tmax} 
				\norm{\psi}_{W^{2,2}(\Omega)} \left(\norm{\xi}_{L^2(\Gamma)} + \tau \norm{\psi}_{W^{1,2}(\Omega)} \right) 
			\diff t\\
		&\leq
			C \left( \norm{\bd}^4_{W^{1,\infty}(\Omega)} + 1 \right)
			\left(\norm{\xi}^2_{L^2(0,\Tmax;L^2(\Gamma))} + \norm{\psi}^2_{L^2(0,\Tmax;W^{1,2}(\Omega))} \right)\\
			&\hspace{1em}+ \frac{1}{8} \norm{\nabla^2 \psi}^2_{L^2(0,\Tmax;L^2(\Omega))}.
	\end{align*}
	By Lemma~\ref{lem:lot_kappa_est} we have
	\[
		\kappa \mathrm{l.o.t.}_\kappa 
		\leq 
			C 
			+ \frac{\kappa}{2} \norm{\nabla (\Delta \psi)}^2_{L^2(0,\Tmax;L^2(\Omega))} 
			+ \frac{1}{4} \norm{\nabla^2 \psi}^2_{L^2(0,\Tmax;L^2(\Omega))}.
	\]
	Putting this inequality and the estimates for $\mathrm{l.o.t.}_\Omega$ and $\mathrm{l.o.t.}_\Gamma$
	together we obtain \eqref{eq:all_lot_est}.
\end{proof}

\begin{lem}\label{lem:full_third_vs_div}
	For $\psi \in W^{4,10/3}(\Omega), \xi \in W^{3,10/3}(\Gamma)$ and $\bd \in W^{4,\infty}(\Omega)$ 
	fulfilling $\varMa \nabla \psi \cdot \bn + \tau \psi = \xi$ and $\bd \cdot \bn = 0$ on $\Gamma$, it holds
	\[
		\norm{\nabla^3 \psi}_{L^2(\Omega)} 
		\leq 
			2 \norm{\nabla (\Delta \psi)}_{L^2(\Omega)}
			+ C \left( \norm{\nabla^2 \psi}_{L^2(\Omega)} + 1  \right)
	\]
    for some constant dependent on $\norm{\psi}_{W^{1,2}(\Omega)}$.
\end{lem}
\begin{proof}
	Assuming that $\psi \in W^{5,10/3}(\Omega)$ the fourth order derivatives are continuous and we can interchange them 
	to obtain $\nabla \cdot (\nabla^3 \psi) = \nabla (\nabla \cdot \nabla^2 \psi)$.
	Using integration by parts two times, we find
	\begin{multline}\label{eq:full_third_exp}
		\norm{\nabla^3 \psi}^2_{L^2(\Omega)} = \int_\Omega \nabla^3 \psi \threedotsbin \nabla^3 \psi \diff \X
		=
			- \int_\Omega \nabla^2 \psi : \nabla \cdot \nabla^3 \psi \diff \X
			+ \int_\Gamma
				(\nabla^3 \psi \cdot \bn) : \nabla^2 \psi
			\dS\\
		=
			\int_\Omega (\nabla \cdot \nabla^2 \psi) \cdot (\nabla \cdot \nabla^2 \psi) \diff \X
			- \int_\Gamma
				(\nabla^2 \psi \bn) \cdot (\nabla \cdot \nabla^2 \psi)
			\dS
			+ \int_\Gamma
				(\nabla^3 \psi \cdot \bn) : \nabla^2 \psi
			\dS.
	\end{multline}
    By the density of smooth functions in $W^{4,10/30}(\Omega)$ this also holds for $\psi \in W^{4,10/30}(\Omega)$.
	The volume term is already the one we aimed for, since
	\[
		\nabla \cdot \nabla^2 \psi = \nabla \cdot (\nabla (\nabla \psi))^T
		= \nabla (\nabla \cdot \nabla \psi) = \nabla (\Delta \psi)
	\] 
	and we turn to the boundary integrals.
	The first boundary term on the right-hand side can be rewritten using the surface divergence of a matrix, 
	\textit{cf.~}Definition~\ref{def:surf_div_matrix} 
	and Corollary~\ref{cor:surf_div_form_matrix}:
	\begin{multline}\label{eq:full_third_bd_exp_1}
		- \int_\Gamma
			(\nabla^2 \psi \bn) \cdot (\nabla \cdot \nabla^2 \psi)
		\dS
		=
			- \int_\Gamma
				(\nabla^2 \psi \bn) \cdot (\nabla_\Gamma \cdot \nabla^2 \psi)
				+ (\nabla^2 \psi \bn) (\nabla^3 \psi \cdot \bn) \bn
			\dS\\
		=
			- \int_\Gamma
				(\nabla^2 \psi \bn) \cdot (\nabla_\Gamma \cdot \nabla^2 \psi)
				+ (\nabla^2 \psi \bn) \cdot (\nabla (\nabla^2 \psi \bn) \bn)
				-  (\nabla^2 \psi \bn) \cdot (\nabla^2 \psi \nabla \bn \bn)
			\dS.
	\end{multline}
	The last term is already of lower order, the second term will cancel with part of the other boundary term in \eqref{eq:full_third_exp}
	and the first one has to be estimated.
	For that we first note that for any test function $\bv \in L^2(\Gamma)$ we have
	\begin{align*}
		\int_\Gamma (\nabla^2 \psi \bn)  \cdot \bv \dS
		&=
			\int_\Gamma
				\nabla (\nabla \psi \cdot \bn) \cdot \bv 
				- (\nabla \bn^T \nabla \psi) \cdot \bv
			\dS\\
		&=
			\int_\Gamma
				\nabla_\Gamma (\nabla \psi \cdot \bn) \cdot \bv 
				+ (\nabla (\nabla \psi \cdot \bn) \cdot \bn) (\bv \cdot \bn)
				- (\nabla \bn^T \nabla \psi) \cdot \bv
			\dS\\
		&=
			\int_\Gamma
				\nabla_\Gamma (\xi - \tau \psi) \cdot \bv 
				+ (\nabla (\nabla \psi \cdot \bn) \cdot \bn) (\bv \cdot \bn)
				- (\nabla \bn^T \nabla \psi) \cdot \bv
			\dS.
	\end{align*}
	Plugging this back in and using $\nabla_\Gamma \cdot (\bs A^T \bs b) = (\nabla_\Gamma \cdot \bs A) \cdot \bs b + \bs A : \nabla_\Gamma \bs b$, 
	\textit{cf.~}Lemma~\ref{lem:surf_prod_rule_matrix}, we obtain,
	\begin{align}\label{eq:full_third_bd_exp_2}
    \begin{split}
		-{}& \int_\Gamma
			(\nabla^2 \psi \bn) \cdot (\nabla_\Gamma \cdot \nabla^2 \psi)
		\dS\\
		={}&
			\int_\Gamma
				\nabla_\Gamma (\xi - \tau \psi) \cdot  (\nabla_\Gamma \cdot \nabla^2 \psi)
				+ (\nabla (\nabla \psi \cdot \bn) \cdot \bn) ((\nabla_\Gamma \cdot \nabla^2 \psi) \cdot \bn)
				- (\nabla \bn^T \nabla \psi) \cdot (\nabla_\Gamma \cdot \nabla^2 \psi)
			\dS\\
		={}&
			\int_\Gamma
				- \nabla^2_\Gamma (\xi - \tau \psi) : \nabla^2 \psi
				+ \nabla_\Gamma (\xi - \tau \psi) \cdot (\nabla^2 \psi \bn) \, \nabla_\Gamma \cdot \bn
				+ (\nabla (\nabla \psi \cdot \bn) \cdot \bn) ((\nabla_\Gamma \cdot (\nabla^2 \psi \bn))
			\dS\\
			&+ \int_\Gamma
				\nabla_\Gamma (\nabla \bn^T \nabla \psi) : \nabla^2 \psi
                - (\nabla (\nabla \psi \cdot \bn) \cdot \bn) (\nabla^2 \psi : \nabla_\Gamma \bn)
            \dS\\
            &- \int_\Gamma
				(\nabla \bn^T \nabla \psi) \cdot (\nabla^2 \psi \bn) \, \nabla_\Gamma \cdot \bn
			\dS,
    \end{split}
	\end{align}
	where we used the integration by parts rule \eqref{eq:bd_ibp_matrix}.
	All terms except the third on the right-hand side are already of lower order and for this term we note
	\begin{align}\label{eq:full_third_bd_exp_3}
    \begin{split}
		\int_\Gamma
			&(\nabla (\nabla \psi \cdot \bn) \cdot \bn) ((\nabla_\Gamma \cdot (\nabla^2 \psi \bn))
		\dS\\
		&=
			- \int_\Gamma
				\nabla_\Gamma (\nabla (\nabla \psi \cdot \bn) \cdot \bn) \cdot (\nabla^2 \psi \bn)
				-  (\nabla (\nabla \psi \cdot \bn) \cdot \bn) ((\nabla^2 \psi \bn) \cdot \bn) \nabla_\Gamma \cdot \bn
			\dS\\
		&=
			- \int_\Gamma
				\nabla_\Gamma (\xi - \tau \psi) \cdot \nabla_\Gamma (\nabla (\nabla \psi \cdot \bn) \cdot \bn) 
				+ (\nabla (\nabla \psi \cdot \bn) \cdot \bn) (\nabla_\Gamma (\nabla (\nabla \psi \cdot \bn) \cdot \bn) \cdot \bn)
			\dS\\
			&\hspace{1em}+ \int_\Gamma
				(\nabla \bn^T \nabla \psi) \cdot \nabla_\Gamma (\nabla (\nabla \psi \cdot \bn) \cdot \bn)
				+ (\nabla (\nabla \psi \cdot \bn) \cdot \bn) ((\nabla^2 \psi \bn) \cdot \bn) \nabla_\Gamma \cdot \bn
			\dS\\
		&=
			\int_\Gamma
				\Delta_\Gamma (\xi - \tau \psi) (\nabla (\nabla \psi \cdot \bn) \cdot \bn) 
				- \nabla_\Gamma \cdot (\nabla \bn^T \nabla \psi) (\nabla (\nabla \psi \cdot \bn) \cdot \bn)
				\dS\\
				&\hspace{1em}+ \int_\Gamma
					 (\nabla (\nabla \psi \cdot \bn) \cdot \bn) 
					 (\nabla \bn^T \nabla \psi) \cdot \bn \, \nabla_\Gamma \cdot \bn
				+ (\nabla (\nabla \psi \cdot \bn) \cdot \bn) ((\nabla^2 \psi \bn) \cdot \bn) \nabla_\Gamma \cdot \bn
			\dS,
    \end{split}
	\end{align}
    where the second term after the second equality sign vanishes, since the surface gradient is tangential and thus $\nabla_\Gamma (\nabla (\nabla \psi \cdot \bn) \cdot \bn) \cdot \bn = 0$ on $\Gamma$.
	Now we turn to the second boundary integral in \eqref{eq:full_third_exp},
	\begin{align}\label{eq:full_third_bd_exp_4}
    \begin{split}
		\int_\Gamma (\nabla^3 \psi \cdot \bn) : \nabla^2 \psi \dS
		&=
			\int_\Gamma
				\nabla (\nabla^2 \psi \bn) : \nabla^2 \psi - (\nabla^2 \psi \nabla \bn) : \nabla^2 \psi
			\dS\\
		&=
			\int_\Gamma
				\nabla_\Gamma (\nabla^2 \psi \bn) : \nabla^2 \psi 
				+ (\nabla (\nabla^2 \psi \bn) \bn \otimes \bn : \nabla^2 \psi)
				- (\nabla^2 \psi \nabla \bn) : \nabla^2 \psi
			\dS\\
		&=
			- \int_\Gamma
				(\nabla^2 \psi \bn) \cdot \nabla_\Gamma \cdot (\nabla^2 \psi)
				- (\nabla^2 \psi \bn) \cdot (\nabla^2 \psi \bn) \nabla_\Gamma \cdot \bn
			\dS\\
			&\hspace{1em}+ \int_\Gamma
				(\nabla (\nabla^2 \psi \bn) \bn) \cdot (\nabla^2 \psi \bn)
				- (\nabla^2 \psi \nabla \bn) : \nabla^2 \psi
			\dS.
   \end{split}
	\end{align}
	The first term on the last line indeed cancels with the second term in \eqref{eq:full_third_bd_exp_1} and 
    the first term on the first line is identical to the first term on the right-hand side \eqref{eq:full_third_bd_exp_1} and thus can be handled identical.
	Putting \eqref{eq:full_third_bd_exp_2}, \eqref{eq:full_third_bd_exp_3} and \eqref{eq:full_third_bd_exp_4} back into \eqref{eq:full_third_bd_exp_1}, 
    where the terms from \eqref{eq:full_third_bd_exp_2} and \eqref{eq:full_third_bd_exp_3} now appear twice due to the term from \eqref{eq:full_third_bd_exp_4}, we can rewrite \eqref{eq:full_third_exp} as 
	\begin{align*}
		&\norm{\nabla^3 \psi}^2_{L^2(\Omega)}
		=
			\norm{\nabla (\Delta \psi)}^2_{L^2(\Omega)}
			- 2 \int_\Gamma
				(\nabla^2 \psi \bn) \cdot (\nabla_\Gamma \cdot \nabla^2 \psi)
			\dS\\
			&\hspace{1em}+ \int_\Gamma
				|\nabla^2 \psi \bn|^2 \nabla_\Gamma \cdot \bn
				- (\nabla^2 \psi \nabla \bn) : \nabla^2 \psi
				+ (\nabla^2 \psi \bn) \cdot (\nabla^2 \psi \nabla \bn \bn)
			\dS\\
		&=
			\norm{\nabla (\Delta \psi)}^2_{L^2(\Omega)}
			- 2 \int_\Gamma
				\nabla^2_\Gamma (\xi - \tau \psi) : \nabla^2 \psi
				- \nabla_\Gamma (\xi - \tau \psi) \cdot (\nabla^2 \psi \bn) \nabla_\Gamma \cdot \bn
			\dS\\
			&\hspace{1em}- 2 \int_\Gamma
				(\nabla (\nabla \psi \cdot \bn) \cdot \bn) (\nabla^2 \psi : \nabla_\Gamma \bn)
				- \nabla_\Gamma (\nabla \bn^T \nabla \psi) : \nabla^2 \psi
				+ (\nabla \bn^T \nabla \psi) \cdot (\nabla^2 \psi \bn) \nabla_\Gamma \bn
			\dS\\
			&\hspace{1em}+ 2 \int_\Gamma
				\Delta_\Gamma (\xi - \tau \psi) (\nabla (\nabla \psi \cdot \bn) \cdot \bn) 
				- \nabla_\Gamma \cdot (\nabla \bn^T \nabla \psi) (\nabla (\nabla \psi \cdot \bn) \cdot \bn)
				\dS\\
			&\hspace{1em}+ \int_\Gamma
				(\nabla (\nabla \psi \cdot \bn) \cdot \bn) (\nabla \bn^T \nabla \psi) \cdot \bn \nabla_\Gamma \cdot \bn
				+ (\nabla (\nabla \psi \cdot \bn) \cdot \bn) ((\nabla^2 \psi \bn) \cdot \bn) \nabla_\Gamma \cdot \bn
			\dS\\
			&\hspace{1em}+ \int_\Gamma
				|\nabla^2 \psi \bn|^2 \nabla_\Gamma \cdot \bn
				- (\nabla^2 \psi \nabla \bn) : \nabla^2 \psi
				+ (\nabla^2 \psi \bn) \cdot (\nabla^2 \psi \nabla \bn \bn)
			\diff \dS\\
		&\leq
			\norm{\nabla (\Delta \psi)}^2_{L^2(\Omega)}
			+ C \norm{\nabla^2 \psi}^2_{L^2(\Gamma)}
			+ C \norm{\nabla \psi}^2_{L^2(\Gamma)}
			+ C\\
		&\leq
			\norm{\nabla (\Delta \psi)}^2_{L^2(\Omega)}
			+ \frac{1}{2} \norm{\nabla^3 \psi}^2_{L^2(\Omega)}
			+ C \norm{\nabla^2 \psi}^2_{L^2(\Omega)}
			+ C,
	\end{align*}
	where we used the estimates from Lemma~\ref{lem:est_surf_grad}
	and the trace embedding \cite[Prop.~8.2]{diBenedetto}.
	The term with the third order derivative can be absorbed into the left-hand side of \eqref{eq:full_third_exp} and thus our proof is complete.
\end{proof}

\begin{lem}\label{lem:est_surf_grad}
	For $f \in W^{2,2}(\Omega)$, $\bv \in W^{2,2}(\Omega)^3$ and $g \in W^{3,2}(\Omega)$ it holds
	\begin{align*}
		&\norm{\nabla_\Gamma S(f)}_{L^2(\Gamma)} \leq \norm{S(\nabla f)}_{L^2(\Gamma)}, \quad
		\norm{\nabla_\Gamma S(\bv)}_{L^2(\Gamma)} \leq \norm{S(\nabla \bv)}_{L^2(\Gamma)}, \\[1ex]
        &\norm{\nabla_\Gamma \cdot S(\bv)}_{L^2(\Gamma)} \leq \norm{S(\nabla \bv)}_{L^2(\Gamma)}, \\[1ex]
		&\norm{\nabla^2_\Gamma S(g)}_{L^2(\Gamma)} \leq \norm{S(\nabla^2 g)}_{L^2(\Gamma)} + C \norm{S(\nabla g)}_{L^2(\Gamma)} \text{ and }\\[1ex]
		&\norm{\Delta_\Gamma S(g)}_{L^2(\Gamma)} \leq \norm{S(\nabla^2 g)}_{L^2(\Gamma)} + C \norm{S(\nabla g)}_{L^2(\Gamma)}
	\end{align*}
	where the constant $C>0$ depends on $\norm{\nabla \bn}_{L^\infty(\Gamma)}$.
\end{lem}
\begin{proof}
	Using the characterization of the surface derivatives by the projection of the bulk derivatives, 
	\textit{cf.~}Theorem~\ref{thm:surf_grad_vs_tan_pro_scalar}, Theorem~\ref{thm:surf_grad_vs_tan_pro_vec}
	and Theorem~\ref{thm:surf_div_form} this follows from straight forward calculations.
\end{proof}
\end{subsection}

	\begin{subsection}{Some additional proofs for the interested reader}\label{sec:app_add_proofs}
	
		\begin{proof}[Proof (of Lemma~\ref{lem:prop_reg_op})]
			By the Hille--Yosida generation theorem we know that $(0,\infty)$ is in the Resolvent set $\rho(A)$ of $A$, 
			see for example \cite[Thm.~3.5]{engel_nagel}, and the Resolvent 
			$R(\lambda, A):= (\lambda - A)^{-1}: X \to D(A) \subseteq X$ is well-defined for all $\lambda \in \rho(A)$ and in $\mathcal{L}(X)$.
			Thus the well-definedness of $R_\kappa$ follows from
			\[
				R_\kappa = (I - \kappa A)^{-1} = \left( \kappa \left( \frac{1}{\kappa} - A \right) \right)^{-1} 
				= \frac{1}{\kappa} R(1/\kappa, A).
			\]
			We now prove the continuity from item \ref{item:R_kappa_strong_cont}. 
			For that we use \cite[Lem.~3.2]{pazy_1983}, by which we have $\lim_{\lambda \to \infty} \lambda R(\lambda, A) x = x $ for all $x \in X$
			and the uniform bound $\norm{R(\lambda, A)}_{\mathcal{L}(X)} \leq M/\lambda$ for some $M>0$, 
			which also comes from the Hille--Yosida generation theorem.
			With these tools, we can estimate
			\begin{align*}
			\begin{split}
				\norm{R_\kappa(x_\kappa) - x}_X 
				&\leq 
					\norm{R_\kappa (x_\kappa - x)}_X + \norm{R_\kappa(x) - x}_X\\
				&\leq
					\frac{1}{\kappa} \norm{R(1/\kappa, A)}_{\mathcal{L}(X)} \norm{x_\kappa - x}_X 
					+ \norm{\frac{1}{\kappa} R(1/\kappa, A)(x) - x}_X\\
				&\leq 
					M \norm{x_\kappa - x}_X 
					+ \norm{\frac{1}{\kappa} R(1/\kappa, A)(x) - x}_X,
			\end{split}
			\end{align*}
			by the above mentioned bound on the operator norm of the Resolvent.
			The right-hand side goes to zero, by the strong convergence of $\{ x_\kappa \}$ and the continuity result from \cite[Lem.~3.2]{pazy_1983}
			mentioned above.
			We next prove item \ref{item:R_kappa_weak_cont}. For arbitrary $\varphi \in X^*$ we have
			\begin{align*}
				\dualP{\varphi}{R_\kappa(x_\kappa) - x} 
				= 
					\dualP{\varphi}{R_\kappa(x_\kappa) - x_\kappa} + \dualP{\varphi}{x_\kappa - x}.
			\end{align*}
			The second term on the right-hand side goes to zero for $\kappa \to 0$ by the definition of weak convergence
			and for the first term we observe
			\begin{align*}
				|\dualP{\varphi}{R_\kappa(x_\kappa) - x_\kappa}|
				&=
					\left| \dualP{\varphi}{\left( \frac{1}{\kappa} R(1/\kappa, A) - I \right) x_\kappa} \right|
				\leq
					\norm{\varphi}_{X^*} \norm{x_\kappa}_X \norm{\frac{1}{\kappa} R(1/\kappa, A) - I}_{\mathcal{L}(X)}\\
				&\leq
					C \norm{\varphi}_{X^*} 
					\norm{\frac{1}{\kappa} R(1/\kappa, A) - \left( \frac{1}{\kappa} - A \right)R(1/\kappa, A) }_{\mathcal{L}(X)}\\
				&\leq 
					C \norm{\varphi}_{X^*} 
					\norm{\frac{1}{\kappa} - \left( \frac{1}{\kappa} - A \right)}_{\mathcal{L}(D(A), X)} \norm{R(1/\kappa, A)}_{\mathcal{L}(X)}\\
				&\leq
					C \norm{\varphi}_{X^*} \norm{A}_{\mathcal{L}(D(A), X)} \kappa,
			\end{align*}
			where we used that weakly convergent sequences are bounded 
			and the uniform bound on the operator norm of the Resolvent given by the Hille--Yosida theorem.
			The right-hand side goes to zero for $\kappa \searrow 0$ and the weak convergence of $(R_\kappa(x_\kappa))$ to $x$ follows.
			Item \ref{item:R_kappa_uniform_bound} is again a simple consequence of this uniform bound.
			Item \ref{item:R_kappa_bound} can be deduce by a simple estimation using $AR(\lambda,A) = \lambda R(\lambda, A) - I$,
			see for example \cite[Sec.~1, Chap.~IV]{engel_nagel}.
			For all $x \in X$ we have,
			\begin{align*}
				\norm{R_\kappa(x)}_Y 
				&\leq 
					C \norm{R_\kappa(x)}_{D(A)}
				=
					C \left( \norm{R_\kappa(x)}_X + \norm{A(R_\kappa(x))}_X \right)\\
				&=
					C \left( \norm{R_\kappa(x)}_X + \frac{1}{\kappa} \norm{A(R(1/\kappa, A)(x))}_X \right)\\
				&\leq 
					C \left( \norm{x}_X + \frac{1}{\kappa} \norm{\frac{1}{\kappa}R(1/\kappa, A)(x)}_X + \frac{1}{\kappa} \norm{x}_X \right)\\
				&\leq 
					C \left( 1 + \frac{1}{\kappa} \right) \norm{x}_X + C \frac{1}{\kappa} \norm{R_\kappa(x)}_X
				\leq 
					C (1 + 1/\kappa) \norm{x}_X
			\end{align*}
			for $C>$ independent of $\kappa$,
			where we used the estimate from item \ref{item:R_kappa_uniform_bound} for the first and the last inequality.
		\end{proof}
		
		\begin{proof}[Proof (of Lemma~\ref{lem:bounds_for_c})]
			By the non-negativity of $\cpm$ and mass conservation we obtain 
			\[
				\norm{\cpm}_{L^\infty(0,T;L^1(\Omega))} = \norm{\cpm_0}_{L^1(\Omega)}
			\]
			and thus by Corollary~\ref{cor:ellip_W1q_est} and the Sobolev embedding $W^{1,6/5}(\Omega) \hookrightarrow L^2(\Omega)$ 
			we obtain 
			\begin{equation}\label{eq:L2_est_phi}
				\norm{\varphi}_{L^\infty(0,T;L^2(\Omega))} 
				\leq C \norm{\varphi}_{L^\infty(0,T;W^{1,6/5}(\Omega))} 
				\leq \frac{C}{\kappa} \sigmp{\norm{\cpm_0}_{L^1(\Omega)}}.
			\end{equation}
			Testing \eqref{eq:reg_p} with $\psi$ we obtain
			\[
				\int_\Omega |\nabla \psi|^2_{\varMa} \diff \X + \tau \int_\Gamma |\psi|^2 \dS
				= \int_\Omega \varphi \psi \diff \X + \int_\Gamma \psi \xi \dS,
			\]
			applying Young's inequality we can estimate 
			\[
				\norm{\psi}_{L^\infty(0,T;L^2(\Omega))} 
				\leq  
					C \left( \norm{\varphi}_{L^\infty(0,T;L^2(\Omega))} + \norm{\xi}_{L^\infty(0,T;L^2(\Gamma))} \right).
			\]
			Combining this estimate with Agmon--Douglis--Nirenberg elliptic estimates, see \cite[Thm.~3.1.1]{lunardi}, we obtain
			\begin{align}\label{eq:W22_est_psi}
			\begin{split}
				\norm{\psi}_{L^\infty(0,T;W^{2,2}(\Omega))} 
				&\leq 
					C \left( 
						\norm{\psi}_{L^\infty(0,T;L^2(\Omega))}
						+ \norm{\varphi}_{L^\infty(0,T;L^2(\Omega))}
						+ \norm{E(\xi)}_{L^\infty(0,T;W^{1,2}(\Omega))}
					\right)\\
				&\leq
					C \left( 
						\norm{\varphi}_{L^\infty(0,T;L^2(\Omega))}
						+ \norm{\xi}_{L^\infty(0,T;W^{1,2}(\Gamma))}
					\right)\\
				&\leq
					\frac{C}{\kappa} \sigmp{\norm{\cpm_0}_{L^1(\Omega)}} + C \norm{\xi}_{L^\infty(0,T;W^{1,2}(\Gamma))},
			\end{split}
			\end{align}
			where $E$ denotes the trace extension operator.
   % from Lemma~\ref{lem:trace_ext}.
			Next, we test \eqref{eq:reg_np} with $\cpm$. By an integration by parts, see \cite[Cor.~8.1.10]{emmrich}, we obtain
			\begin{multline}\label{eq:reg_np_tested}
				 \frac{\diff}{\diff t} \frac{1}{2} \norm{\cpm(t)}^2_{L^2(\Omega)} + \int_\Omega |\nabla \cpm(t)|^2_{\lamMa} \diff \X
				 = 
				 	- \int_\Omega \cpm(t) \lamMa \nabla \psi(t) \cdot \nabla \cpm(t) \diff \X\\
				 \leq
				 	\norm{\cpm(t)}_{L^3(\Omega)} 
				 	\norm{\lamMa}_{L^\infty(\Omega)} 
				 	\norm{\nabla \psi(t)}_{L^6(\Omega)}
				 	\norm{\nabla \cpm(t)}_{L^2(\Omega)}
			\end{multline}
			for almost all $t \in (0,T)$, where we used that $\bv$ is divergence free and H\"{o}lder's inequality.
            Using Gagliardo--Nirenberg's inequality and a classical Gronwall argument we 
            obtain the bound of $\cpm$ in 
            \[
                L^\infty(0,T;L^2(\Omega)) \cap L^2(0,T;W^{1,2}(\Omega)) 
				\hookrightarrow L^\infty(0,T;L^2(\Omega)) \cap L^2(0,T;L^6(\Omega)) \hookrightarrow L^3(0,T;L^3(\Omega))
            \]
            and thus the boundedness of the first two terms in \eqref{eq:bounds_for_c} follows.
			For the third term in \eqref{eq:bounds_for_c} we note that 
			by elliptic regularity \cite[Thm.~3.1.1]{lunardi}, we have
			\begin{align*}
				\norm{\varphi}_{L^\infty(0,T;W^{2,2}(\Omega))} 
				\leq 
					C \left( 
						\frac{1}{\kappa} \sigmp{\norm{\cpm}_{L^\infty(0,T;L^2(\Omega))}} 
						+ \left( \frac{1}{\kappa} + 1\right) \norm{\varphi}_{L^\infty(0,T;L^2(\Omega))} 
					\right)
                \leq
                    C(\kappa)
                %\\
				%\leq
				%	\frac{C}{\kappa} \exp \left( C \norm{\lamMa}_{L^\infty(\Omega)}  T
				%		\left( \frac{\norm{\cpm_0}_{L^1(\Omega)}}{\kappa} + \%norm{\xi}_{L^\infty(0,T;W^{1,2}(\Gamma))} \right)\right)
				%		\norm{\cpm_0}_{L^2(\Omega)}
				%	+ C \left( \frac{1}{\kappa} + \frac{1}{\kappa^2} \right) \norm{\cpm_0}_{L^1(\Omega)}\\
				%\leq
				%	C \left( \frac{1}{\kappa} + \frac{1}{\kappa^2} \right)  
				%	\exp\left( 
				%		C \norm{\lamMa}_{L^\infty(\Omega)}  T
				%		\left( \frac{\norm{\cpm_0}_{L^1(\Omega)}}{\kappa} + \%norm{\xi}_{L^\infty(0,T;W^{1,2}(\Gamma))}
				%	\right) \right) \norm{\cpm_0}_{L^2(\Omega)}
			\end{align*}
			and thus by  higher order elliptic estimates \cite[Rem.~2.5.1.2]{grisvard}, and \eqref{eq:W22_est_psi} we obtain
			\begin{align*}
				\norm{\psi}_{L^\infty(0,T;W^{4,2}(\Omega))} 
				\leq 
					C \left( 
						\norm{\varphi }_{L^\infty(0,T;W^{2,2}(\Omega))} + \norm{\xi}_{L^\infty(0,T;W^{3,2}(\Gamma))}
					\right)
                \leq
                    C(\kappa),
            %\\
			%	= 
			%		C \left( \frac{1}{\kappa} + \frac{1}{\kappa^2} \right) 
			%		\exp \left( 
			%			C \norm{\lamMa}_{L^\infty(\Omega)}  T
			%			\left( \frac{\norm{\cpm_0}_{L^1(\Omega)}}{\kappa} + \norm{\xi}_{L^\infty(0,T;W^{1,2}(\Gamma))} \right)
			%		 \right) \norm{\cpm_0}_{L^2(\Omega)}\\
			%		+ C \norm{\xi}_{L^\infty(0,T;W^{3,2}(\Gamma))},
			\end{align*}
			which finishes the proof of Lemma~\ref{lem:bounds_for_c}.
		\end{proof}
	
	\end{subsection}

\end{section}

\end{document}